\documentclass{amsart}
\usepackage{amsmath,amssymb, amsbsy}

\usepackage{color,psfrag}
\usepackage[dvips]{graphicx}
\usepackage{enumerate}
\usepackage{hyperref}

\usepackage[a4paper,width={15cm},left=3cm,bottom=3cm, top=3cm]{geometry}

\newcommand{\R}{{\mathbb R}}
\newcommand{\N}{{\mathbb N}}

\newcommand{\C}{{\mathbb C}}

\newcommand{\SF}{{\mathbb S}^{N}_+}

\newcommand{\e }{\varepsilon}

\newcommand{\dive }{\mathop{\rm div}}

\def\ds{\displaystyle}
\def\endproof{\hfill $\Box$\par\vskip3mm}
\renewcommand{\ge }{\geqslant}
\renewcommand{\geq }{\geqslant}
\renewcommand{\le }{\leqslant}
\renewcommand{\leq }{\leqslant}

\newcommand{\nero }{\color{black}}

\newenvironment{pf}{\noindent{\sc Proof}.\enspace}{\hfill\qed\medskip}

\newtheorem{Theorem}{Theorem}[section]

\newtheorem{Lemma}[Theorem]{Lemma}
\newtheorem{Proposition}[Theorem]{Proposition}
\theoremstyle{definition}

\newtheorem{remark}[Theorem]{Remark}
\newtheorem{OP}[Theorem]{Open Problem}
\begin{document}

\title[Higher order fractional Laplace equation]{Unique continuation principles for a higher order fractional Laplace equation}

\author[Veronica Felli \and Alberto Ferrero]{Veronica Felli \and Alberto Ferrero}
\address{\hbox{\parbox{5.7in}{\medskip\noindent{Veronica Felli, \\
Universit\`a di Milano
        Bicocca,\\
        Dipartimento di Scienza dei Materiali, \\
        Via Cozzi
        55, 20125 Milano, Italy. \\[3pt]
        Alberto Ferrero, \\
        Universit\`a del Piemonte Orientale, \\
        Dipartimento di Scienze e Innovazione Tecnologica, \\
        Viale Teresa Michel 11, 15121 Alessandria, Italy. \\[3pt]
        \em{E-mail addresses: }{\tt veronica.felli@unimib.it,
          alberto.ferrero@uniupo.it}.}}}}

\date{September 25, 2018}

\thanks{
The authors are partially supported by the INDAM-GNAMPA 2018 grant ``Formula di monotonia e applicazioni: problemi frazionari e stabilit\`a spettrale rispetto a perturbazioni del dominio''.
 V. Felli is partially supported by the PRIN 2015
grant ``Variational methods, with applications to problems in
mathematical physics and geometry''.
}

\begin{abstract}
   \noindent
In this paper we prove strong unique continuation principle and unique
   continuation from sets of positive measure  for solutions of a higher order
   fractional Laplace equation in an open domain. Our proofs are based on
   the Caffarelli-Silvestre \cite{CS} extension method combined with an
   Almgren type monotonicity formula. The corresponding extended
   problem is formulated as a systems of two second order 
   equations with singular or degenerate weights in a half-space, for which asymptotics estimates are
   derived by a blow-up analysis.
\end{abstract}

\maketitle

\section{Introduction and main results}\label{intro}
We study the following higher order fractional Laplace equation
\begin{equation} \label{eq:problema}
(-\Delta)^s u=0  \qquad \text{in } \Omega,
\end{equation}
where $1<s<2$, $\Omega\subset \R^N$ is  an open  domain
with $N>2s$, and the fractional Laplacian $(-\Delta)^s$ of a
function $u$ defined over the whole $\R^N$ is defined by means of
the Fourier transform:
\[
\widehat{\!(-\Delta)^s u}(\xi)=|\xi|^{2s}\widehat u(\xi) \, .
\]
Here by Fourier transform in $\R^N$ we mean
\begin{equation*}
\widehat u(\xi)=\mathcal F u(\xi):=\frac 1{(2\pi)^{N/2}}
\int_{\R^N} e^{-ix \cdot\xi}u(x)\, dx \, .
\end{equation*}
In the sequel we will explain in more details what we mean by a
weak solution of \eqref{eq:problema}. Our main
purpose is to prove the validity of unique continuation principles
for solutions to \eqref{eq:problema}. 

Unique continuation properties and qualitative local behavior of
solutions to fractional elliptic problems are a subject which was
widely studied in the last years. In \cite{FF}, the authors study
a semilinear fractional elliptic problem containing a singular
potential of Hardy type, a perturbation potential with a lower
order singularity and a nonlinearity that is at most critical
with respect to a suitable Sobolev exponent. In that paper the
fractional differential operator is $(-\Delta)^s$ with power $0<s<1$;
see also \cite{FF2} for analogous results for relativistic Schr\"odinger
operators. Unique continuation for fractional
Laplacians with power $s\in(0,1)$ was also investigated in \cite{Ruland} in presence of rough
potentials and in \cite{Yu} for fractional operators with variable coefficients.

Other results concerning qualitative properties of solutions of
equations with the fractional Laplace operator $(-\Delta)^s$ can
be found in \cite{CaSi, JLX1, JLX2, TX}. For
more details on basic results on the fractional Laplace operator
see \cite{BCdPS, CS, DNPV}.

Up to our knowledge, unique continuation properties for higher order
fractional elliptic equations were first studied in the paper
\cite{Y}. Here the author states a strong unique continuation
property for the Laplace equation \eqref{eq:problema} for any
noninteger $s>0$.

More precisely, in \cite[Corollary 5.5]{Y} it is stated that
the solutions to \eqref{eq:problema} vanishing of infinite order at a
point are necessarily null in $\Omega$. In \cite{Y} the proof of this
result is not written in details; it is just observed that, following
the classical argument by Garofalo and Lin \cite{GL}, the boundedness
of the Almgren frequency function for solutions of some extended
problem, together with the Caffarelli-Silvestre type extension result given
in \cite{Y}, suffices to provide the strong unique continuation
property. However, we think that the boundedness of the frequency
function proved in \cite{Y} only shows the validity of a unique
continuation principle for the extended function $U$ (see
\eqref{eq:aux-prob}) and not for the solution $u$ of equation
\eqref{eq:problema}; indeed, it is nontrivial to exclude
that $u$ vanishes of infinite order at a point when $U$ does not. A
first goal of the present paper is to give a complete proof of
\cite[Corollary 5.5]{Y} excluding such an occurrence  by means of a
blow-up analysis  and a complete classification of local asymptotics
of solutions for the extended problem. 
Nevertheless, we acknowledge the fundamental role of paper
\cite{Y} since part of our approach to the unique continuation
principle takes inspiration from the Caffarelli-Silvestre procedure
\cite{CS} and the Almgren monotonicity formula
performed by \cite{Y} in the higher order setting.

The problem of unique continuation for higher order
fractional Laplacians was also studied by Seo in \cite{seo0, seo1, seo2}
in presence of potentials in Morrey spaces; more precisely,
in \cite{seo0, seo1, seo2} Seo
uses Carleman inequalities to prove a \emph{weak} unique
continuation result, i.e. vanishing of solutions which are zero on an
open set; we recall that the \emph{strong} unique
continuation  property instead requires the weaker assumption of
infinite vanishing order at  a point.

The major contribution of the present paper goes beyond bridging
monotonicity formula for the extended problem and unique continuation
for the original nonlocal equation, since our local analysis provides
sharp results on the asymptotic behavior of solutions for the above
mentioned extended problem, see \eqref{eq:aux-prob}, \eqref{eq:system}
below. Moreover our analysis allows us to prove a second version of
the unique continuation principle which has, as an assumption,
vanishing of solutions of \eqref{eq:problema} on sets of positive
measure.

As already mentioned above, our approach is based on the
Caffarelli-Silvestre procedure \cite{CS} and on an Almgren type
monotonicity formula. But differently from \cite{Y}, we combine the
Almgren formula with a blow-up procedure with the purpose of proving
asymptotic formulas for solutions of the extended problem.  And it is
by mean of this asymptotic formula that we are able to prove the
validity of the two versions of the unique continuation principle.  Up
to now, we succeeded in applying our method only to the fractional
Laplace equation but we believe that similar results can be obtained
in a more general setting by adding to equation \eqref{eq:problema}
linear terms with singular potentials and subcritical nonlinearities,
see Open Problem \ref{op:1.5} for a more detailed explanation. A first
step towards this goal is achieved in \cite{FeFe}, where we prove the
validity of an asymptotic formula and of unique continuation
principles for problem
\begin{equation*}
(-\Delta)^{3/2} u=h(x)u
\end{equation*}
in  open  domains of $\R^N$.
The special case $s=\frac 32$ represents the
``middle case'' between the classical Laplace operator $-\Delta$
and the bilaplacian $(-\Delta)^2$ and produces a significant
simplification when dealing with the Caffarelli-Silvestre
extension, see \eqref{eq:aux-prob} for more details.

Before stating the main results of the paper we introduce a suitable
notion of weak solutions to \eqref{eq:problema}. We define
$\mathcal D^{s,2}(\R^N)$ as the completion of the space
$C^\infty_c(\R^N)$ of $C^\infty$ real compact supported functions,
with respect to the scalar product
\begin{equation*}
(u,v)_{\mathcal D^{s,2}(\R^N)}:=\int_{\R^N} |\xi|^{2s}\, \widehat
u(\xi)\overline{\widehat v(\xi)} \, d\xi \, .
\end{equation*}
We define a solution of \eqref{eq:problema}
as a function $u\in \mathcal D^{s,2}(\R^N)$ satisfying
\begin{equation} \label{eq:problema-variazionale}
(u,\varphi)_{\mathcal D^{s,2}(\R^N)}=0 \qquad \text{for any }
\varphi\in C^\infty_c(\Omega) \, .
\end{equation}
For a motivation of this definition see \cite{DNPV}, where a detailed
treatise on fractional Sobolev spaces and on $(-\Delta)^s$ in the
case $0<s<1$ is provided. See also \cite[(7)]{FF} for the definition of
solution of a nonlinear problem with $(-\Delta)^s$ in the case
$0<s<1$.

The first main result of the paper is the
following strong unique continuation principle. 

\begin{Theorem} \label{t:SUCP-1} Assume that $1<s<2$ and $N>2s$. Let
  $u\in \mathcal D^{s,2}(\R^N)$ be a nontrivial solution of
  \eqref{eq:problema}. Let us also assume that
  $(-\Delta)^s u\in (\mathcal D^{s-1,2}(\R^N))^\star$, where
  $(\mathcal D^{s-1,2}(\R^N))^\star$ denotes the dual space of
  $\mathcal D^{s-1,2}(\R^N)$, in the sense that the linear functional
  $\varphi\mapsto \int_{\R^N} |\xi|^{2s} \widehat
  u(\xi)\overline{\widehat \varphi(\xi)} \, d\xi$,
  $\varphi\in C^\infty_c(\R^N)$, is continuous with respect to the
  norm induced by $\mathcal D^{s-1,2}(\R^N)$. If there exists
  $x_0 \in \Omega$ such that $u(x)=O(|x-x_0|^k)$ as $x\to x_0$ for any
  $k\in \N$, then $u\equiv 0$ in~$\Omega$.
\end{Theorem}

Now we state a second  version of unique
continuation principle where the condition on vanishing of
infinite order around a point assumed in Theorem \ref{t:SUCP-1} is
replaced by vanishing on a set of positive measure.

\begin{Theorem} \label{t:SUCP-2}
Assume that $1<s<2$ and $N>2s$. Let $u\in \mathcal D^{s,2}(\R^N)$
be a nontrivial solution of \eqref{eq:problema}. Let us also assume that
$(-\Delta)^s u\in (\mathcal D^{s-1,2}(\R^N))^\star$ in the sense explained in the statement of Theorem \ref{t:SUCP-1}. If there exists a
measurable set $E\subset \Omega$ of positive measure such that
$u\equiv 0$ on $E$, then $u\equiv 0$ in $\Omega$.
\end{Theorem}

As we mentioned before the statement of the main results, we
believe that it should be interesting to extend unique continuation principles to solutions of more
general elliptic fractional equations. We leave this question as
an open problem.

\begin{OP}\label{op:1.5} {\rm Let $u\in \mathcal D^{s,2}(\R^N)$ be a weak solution of
\begin{equation}  \label{eq:problema-nonh}
(-\Delta)^s u=h(x)u+f(x,u)  \quad \text{in } \Omega,
\end{equation}
with $h$ and $f$ satisfying
\begin{equation*}
h\in W^{1,\infty}_{{\rm loc}}(\Omega\setminus\{0\})\, , \qquad
|h(x)|+|x\cdot \nabla h(x)|\le C_h |x|^{-2(s-1)+\e} \quad \text{ for a.e. }
x\in \Omega,
\end{equation*}
and
\begin{equation*}
f\in C^1(\Omega\times \R) ,
\quad |f(x,\sigma)|\le C_f |\sigma|^{p-1} \qquad \text{for a.e. } x\in
\Omega \ \text{and} \ \sigma \in \R,
\end{equation*}
where $2<p<2^*(N,s-1):=\frac{2N}{N-2(s-1)}$. Study the validity of
the two versions of unique continuation principle contained in
Theorems \ref{t:SUCP-1}-\ref{t:SUCP-2} for solutions of
\eqref{eq:problema-nonh}.
\endproof }
\end{OP}

Now, we explain in more details what we mean by
the previously mentioned extended problem and we state which kind of
asymptotic estimate we will prove on its solutions. Let
$u\in \mathcal D^{s,2}(\R^N)$ be a solution of \eqref{eq:problema} in
the sense given in \eqref{eq:problema-variazionale} and let
$U\in \mathcal D_b$ be a solution to the problem
\begin{equation} \label{eq:aux-prob}
\begin{cases}
\Delta_b^2 U=0 & \qquad \text{in } \R^{N+1}_+\, , \\[5pt]
U(\cdot,0)=u(\cdot) & \qquad \text{in } \R^N \, ,\\[5pt]
\lim_{t \to 0^+ } t^b U_t(\cdot,t)\equiv 0 & \qquad \text{in }
\R^N \, ,
\end{cases}
\end{equation}
where $b=3-2s\in (-1,1)$, $\mathcal D_b$ is the functional space
introduced in Section \ref{s:alternative-formulation}, and $\Delta_b$
is the operator defined at the beginning of Section
\ref{s:preliminary}.  It is possible to prove existence and uniqueness
of solutions of \eqref{eq:aux-prob} for any function
$u\in \mathcal D^{s,2}(\R^N)$ as one can see from Section
\ref{s:alternative-formulation}.

Now, let $x_0\in \Omega$ and let $R>0$ be such that $B_{2R}'(x_0) \subset \Omega$ where, according with \eqref{eq:sets}, $B_{2R}'(x_0)$ denotes the open ball in $\R^N$ of radius $2R$ centered at $x_0$. Then, if $u\in \mathcal D^{s,2}(\R^N)$ is a solution of
\eqref{eq:problema}, putting $V:=\Delta_b U$, the couple $(U,V)\in H^1(B_R^+(x_0);t^b)\times H^1(B_R^+(x_0);t^b)$ weakly solves the system
\begin{equation}\label{eq:system}
  \begin{cases}
\Delta_b U=V &\text{in }B_R^+(x_0) \, ,\\
\Delta_b V=0 &\text{in }B_R^+(x_0) \, ,\\
\lim_{t \to 0^+ } t^b U_t(\cdot,0)= 0 &\text{in } B_R'(x_0) \, ,\\
\lim_{t \to 0^+ } t^b V_t(\cdot,0)=0 &\text{in } B_R'(x_0) \, ,
  \end{cases}
\end{equation}
see \eqref{eq:sets} and the successive part of Section \ref{s:preliminary} for the definition of the weighted Sobolev space $H^1(B_R^+(x_0);t^b)$. This means
that the couple $(U,V)$ satisfies
\begin{equation*}
\int_{B_R^+(x_0)} t^b \nabla U\nabla \varphi \, dz=-\int_{B_R^+(x_0)} t^b V\varphi \, dz \quad \text{and} \quad
\int_{B_R^+(x_0)} t^b \nabla V\nabla \varphi \, dz=0
\end{equation*}
for any $\varphi\in H^1_0(\Sigma_R^+(x_0);t^b)$ with $H^1_0(\Sigma_R^+(x_0);t^b)$ as in Section \ref{s:preliminary}.

In order to state our result on the local behavior of solutions of \eqref{eq:system}, we introduce the following eigenvalue problem:
\begin{equation} \label{eq:Eigen-Weight}
\begin{cases}
  -{\rm div}_{\SF}(\theta_{N+1}^b \nabla_{\SF} \Psi)=\mu\,
  \theta_{N+1}^b \Psi
  & \qquad \text{in } \SF, \\[6pt]
  {\ds \lim_{\theta_{N+1}\to 0^+} \theta_{N+1}^b \nabla_{\SF} \Psi
    \cdot {\bf e}_{N+1}=0} & \qquad \text{on }\partial \SF,
\end{cases}
\end{equation}
where ${\bf e}_{N+1}=(0,\dots,0,1)\in \R^{N+1}$, $\SF=\{(\theta_1,\dots,\theta_{N+1})\in \mathbb
S^N:\theta_{N+1}>0\}$ and $\mathbb S^N$ is
the $N$-dimensional unit sphere in $\R^{N+1}$.

\begin{remark}\label{rem:nuovo} We observe that the
    eigenfunctions of problem \eqref{eq:Eigen-Weight} cannot vanish
    identically on $\partial \SF$; indeed, if an eigenfunction $\Psi$
    vanishes on $\partial \SF$, then the function
    $W(z):=|z|^{\sigma_\ell^+}\Psi(z/|z|)$ (with
    $\sigma_\ell^+=-\tfrac{N+b-1}2+\sqrt{\mu_\ell+(N+b-1)^2/4}$
    and $\mu_\ell$ being the eigenvalue associated to $\Psi$) would be
    a weak solution to the equation ${\rm div}(t^b\nabla W)=0$ in
    $\R^{N+1}_+$ satisfying both Dirichlet and weighted Neumann
    homogeneous boundary conditions; then its trivial extension to the
    entire space $\R^{N+1}$ would violate the unique continuation
    principle for elliptic equations with Muckenhoupt weights proved
    in \cite{tao-zhang} (see also \cite{GL}, \cite[Corollary
    3.3]{STT}, and
    \cite[Proposition 2.2]{Ruland}).
\end{remark}

By classical spectral theory the eigenvalue problem
\eqref{eq:Eigen-Weight} admits a diverging sequence of real
eigenvalues with finite multiplicity. We denote these distinct
eigenvalues by $\mu_n$ and their multiplicity by
$M_n$ with $n\in \N\cup \{0\}$. Moreover, for any $n\ge
0$ let $\{Y_{n,m}\}_{m=1,\dots,M_n}$ be a
$L^2(\SF;\theta_{N+1}^b)$-orthonormal basis of the eigenspace of
$\mu_n$.

We now state the main result on solutions to system \eqref{eq:system}.

\begin{Theorem} \label{t:asym-est} Assume that $1<s<2$, $N>2s$ and let
  $b=3-2s\in (-1,1)$. For some $x_0\in \R^N$ let
  $(U,V)\in H^1(B_R^+(x_0);t^b)\times H^1(B_R^+(x_0);t^b)$ be a
  nontrivial weak solution of \eqref{eq:system}.  Then there
  exists $\delta_1>0$, a linear combination
    $\Psi_1\not\equiv 0$ of
    eigenfunctions of \eqref{eq:Eigen-Weight}, possibly corresponding
    to different eigenvalues, and $\alpha\in (0,1)$ such that
\begin{equation*}
 \lambda^{-\delta_1} U(z_0+\lambda(z-z_0)) \to |z-z_0|^{\delta_1} \Psi_1\left(\tfrac
{z-z_0}{|z-z_0|}\right)
\end{equation*}
in $H^1(B_1^+(x_0);t^b)$ and 
in $C^{1,\alpha}_{{\rm loc}}(B_1^+(x_0))$ as $\lambda\to 0^+$ where we
put $z_0=(x_0,0)\in \R^{N+1}$.  Furthermore, if $V\not\equiv0$, there
exists $\delta_2>0$, a linear combination $\Psi_2\not\equiv 0$ of
eigenfunctions of \eqref{eq:Eigen-Weight}, possibly corresponding to
different eigenvalues, and $\alpha\in (0,1)$ such that 
\begin{equation*}
\lambda^{-\delta_2} V(z_0+\lambda(z-z_0)) \to |z-z_0|^{\delta_2} \Psi_2\left(\tfrac
{z-z_0}{|z-z_0|}\right)
\end{equation*}
in $H^1(B_1^+(x_0);t^b)$ and  in $C^{1,\alpha}_{{\rm loc}}(B_1^+(x_0))$.
\end{Theorem}

We observe that Theorem \ref{t:asym-est} implies a unique
continuation principle from  boundary points for solutions to
\eqref{eq:system}; we refer to \cite{ae,aek,FF-edinb,KukavicaNystrom,tz1} for
unique continuation from the boundary established via 
Almgren monotonicity formula.

\begin{remark} \label{r:1.5} We observe that Theorem \ref{t:asym-est}
  in general does not provide a sharp asymptotic formula around
  $x_0\in \Omega$ for solutions to the original problem
  \eqref{eq:problema} when $u$ and $U$ are as in \eqref{eq:aux-prob},
  even if $u$ is the restriction to $B_R'(x_0)$ of $U$. This because
  we cannot exclude that the function $\Psi_1$ in Theorem
  \ref{t:asym-est} vanishes identically on $\partial \mathbb S^N$;
  what we can say is that this event cannot occur if $\Psi_1$ is an
  eigenfunction of \eqref{eq:Eigen-Weight} as explained in Remark
  \ref{rem:nuovo}. For this reason the unique continuation principles
  stated in Theorems \ref{t:SUCP-1}--\ref{t:SUCP-2} are not a direct
  consequence of Theorem \ref{t:asym-est} and additional arguments
  have to be employed in their proofs in order to exploit the
  asymptotic estimates of Theorem \ref{t:asym-est}.
\end{remark}

We observe that the proof of Theorem \ref{t:asym-est} presents
substantial additional difficulties with respect to the lower order
case $s\in(0,1)$ treated in \cite{FF}, since the corresponding
Dirichlet-to-Neumann local problem is a fourth order equation (see
\eqref{eq:aux-prob}) which is equivalent to the second order system
\eqref{eq:system} with singular/degenerate weights and Neumann
boundary conditions. In particular, several steps in our procedure,
such as regularity and blow-up analysis, turn out to be more delicate
for systems than for the single equation arising from the
Caffarelli-Silvestre extension in the lower order case $s\in(0,1)$.

We conclude this section by explaining how the rest of the paper
is structured. Section \ref{s:preliminary} is devoted to some
preliminary results and notations which will be used in the proofs
of the main statements. In Section \ref{s:alternative-formulation}
we introduce a Caffarelli-Silvestre type extension for functions
$u\in \mathcal D^{s,2}(\R^N)$ and we provide an alternative
formulation for problem \eqref{eq:problema}. In Section
\ref{s:Almgren} we introduce an Almgren-type function and
we prove a related monotonicity formula. In Section
\ref{s:blow-up} we perform a blow-up procedure and we prove
asymptotic estimates for the extended functions introduced in
Section \ref{s:alternative-formulation}. Section
\ref{s:main-results} contains the proofs of the main results
of the paper. Finally, Section \ref{s:Appendix} is an appendix
devoted to weighted Sobolev spaces and related inequalities,
H\"older regularity for solutions of a class of second order elliptic
equations and systems
with variable
coefficients, and
some properties of first kind Bessel functions.

\section{Preliminaries and notations} \label{s:preliminary}

{\bf Notations.} We list below some notations used throughout the
paper.

\begin{itemize}

\item $\R^{N+1}_+=\{z=(z_1,\dots,z_{N+1})\in
\R^{N+1}:z_{N+1}>0\}$.

\item $\mathbb S^N=\{z\in \R^{N+1}:|z|=1\}$ denotes the unit
$N$-dimensional sphere in $\R^{N+1}$.

\item $\SF=\{(\theta_1,\dots,\theta_{N+1})\in \mathbb
S^N:\theta_{N+1}>0\}=\mathbb S^N\cap \R^{N+1}_+$.

\item $dS$ denotes the surface element in boundary integrals.

\item $dz=dx\,dt$, $z=(x,t)\in \R^N\times \R$, denotes the
$(N+1)$-dimensional volume element.

\item $\Delta_b U=\Delta U+\frac bt U_t$ for any function
$U=U(x,t)$ with $x\in \R^N$ and $t\in\R$, where $\Delta U$ denotes
the classical Laplacian in $\R^{N+1}$ and $U_t$ the partial
derivative with respect to $t$.
\end{itemize}

\medskip

The main purpose of this section is to prove a regularity result
for the boundary value problem \eqref{eq:reg-H1}.
We observe that such a regularity result is needed to make the Almgren
quotient \eqref{eq:def-N} well-defined and seems to be taken for
granted in \cite{Y} although not at all trivial.
 To prove the needed regularity we introduce two auxiliary problems, namely the
eigenvalue problem \eqref{eq:eigen-prob} and the Poisson type
equation \eqref{eq:Poisson}.

For any $x_0\in \R^N$, $t_0\in \R$ and $R>0$ we define
\begin{align} \label{eq:sets}
& B_R(x_0,t_0):=\{(x,t)\in \R^{N+1}: |x-x_0|^2+|t-t_0|^2<R^2\} \, , \\[5pt]
\notag & B_R^+(x_0):=\{(x,t)\in B_R(x_0,0): t>0\} \, , \quad
B_R^-(x_0):=\{(x,t)\in B_R(x_0,0): t<0\} \, , \\[5pt]
\notag & B_R'(x_0):=\{x\in \R^N:|x-x_0|<R\} \, , \\[5pt]
\notag & S_R^+(x_0):=\{(x,t)\in \partial B_R(x_0,0):t>0\} \, , \quad S_R^-(x_0):=\{(x,t)\in \partial B_R(x_0,0):t<0\} \, , \\[5pt]
\notag & \Sigma_R^+(x_0):=B_R^+(x_0)\cup (B_R'(x_0)\times \{0\})
\, , \quad \Sigma_R^-(x_0):=B_R^-(x_0)\cup (B_R'(x_0)\times \{0\})
\, ,
\\[5pt]
\notag & Q_R(x_0):=B_R'(x_0)\times (-R,R) \, , \\[5pt]
\notag & Q_R^+(x_0):=B_R'(x_0)\times (0,R) \, , \quad
Q_R^-(x_0):=B_R'(x_0)\times (-R,0) \, ,
\\[5pt]
\notag & \Gamma_R^+(x_0):=B_R'(x_0)\times [0,R) \, , \quad
\Gamma_R^-(x_0):=B_R'(x_0)\times (-R,0] \, .
\end{align}
Given $b\in (-1,1)$, for any $x_0\in \R^N$ and $R>0$ we define the
weighted Sobolev space $H^1(B_R^+(x_0);t^b)$ of functions $U\in
L^2(B_R^+(x_0);t^b)$ such that $\nabla U\in
L^2(B_R^+(x_0);t^b)$, endowed with the norm
$$
\|U\|_{H^1(B_R^+(x_0);t^b)}:=\left(\int_{B_R^+(x_0)} t^{b} |\nabla
U(x,t)|^2 dx\,dt+\int_{B_R^+(x_0)} t^{b} (U(x,t))^2 dx\,dt
\right)^{1/2} \, .
$$
We also define the space $H^1_0(\Sigma_R^+(x_0);t^b)$ as the
closure in $H^1(B_R^+(x_0);t^b)$ of $C^\infty_c(\Sigma_R^+(x_0))$.

In a completely similar way, we can introduce the Hilbert space
$H^1(Q_R^+(x_0);t^b)$ and its subspace
$H^1_0(\Gamma_R^+(x_0);t^b)$ defined as the closure in
$H^1(Q_R^+(x_0);t^b)$ of $C^\infty_c(\Gamma_R^+(x_0))$.

We observe that thanks to \eqref{eq:PA2-2} the spaces
$H^1_0(\Sigma_R^+(x_0);t^b)$ and $H^1_0(\Gamma_R^+(x_0);t^b)$ may
be endowed with the equivalent norms
\begin{equation*}
\|U\|_{H^1_0(\Sigma_R^+(x_0);t^b)}:=\left(\int_{B_R^+(x_0)} t^b
|\nabla U|^2 dx\,dt \right)^{\frac 12} \, , \quad
\|U\|_{H^1_0(\Gamma_R^+(x_0);t^b)}:=\left(\int_{Q_R^+(x_0)} t^b
|\nabla U|^2 dx\,dt\right)^{\frac 12} \, .
\end{equation*}
For any $x_0\in \Omega$ let $R>0$ be such that
$B_{2R}'(x_0)\subset \Omega$; here and in the sequel
$\Omega\subset \R^N$ is an open  domain.

Let us consider the eigenvalue problem
\begin{equation} \label{eq:eigen-prob}
\begin{cases}
-\Delta_b U=\lambda U  & \qquad \text{in } Q_{2R}^+(x_0) \, ,\\[5pt]
U=0 & \qquad \text{on } [\partial B_{2R}'(x_0)\times (0,2R)]\cup
[B_{2R}'(x_0) \times \{2R\}] \, , \\[5pt]
{\ds \lim_{t\to 0^+} t^b U_t(\cdot,t)\equiv 0 } & \qquad \text{on
} B_{2R}'(x_0) ,
\end{cases}
\end{equation}
in a weak sense, i.e.
\begin{equation} \label{eq:EPW}
\int_{Q_{2R}^+(x_0)}t^b\nabla U\nabla \varphi\,dx\,dt=\lambda
\int_{Q_{2R}^+(x_0)}t^b U\varphi\,dx\,dt,\quad \text{for all
}\varphi\in H^1_0(\Gamma_R^+(x_0);t^b).
\end{equation}
In the following proposition we construct a complete orthonormal
system for $L^2(Q_{2R}^+(x_0);t^b)$ consisting of eigenfunctions
of \eqref{eq:eigen-prob}.

\begin{Proposition} \label{p:eigenfunctions}
Let $b\in (-1,1)$, $x_0\in \Omega$ and let $R>0$ be such that
$B_{2R}'(x_0)\subset \Omega$. Define
\begin{equation*} e_{n,m}(x,t):=\gamma_m \, t^\alpha
J_{-\alpha}\left(\tfrac{j_{-\alpha,m}}{2R}\, t\right)e_n(x) \,
\qquad \text{for any } n,m\in \N\setminus\{0\}
\end{equation*}
and
\begin{equation*}
\lambda_{n,m}:=\mu_n+\tfrac{j_{-\alpha,m}^2}{4R^2} \, , \qquad
\text{for any } n,m\in \N\setminus\{0\}
\end{equation*}
where $\alpha:=\frac{1-b}2$, $J_{-\alpha}$ is the first kind
Bessel function with index $-\alpha$,
\begin{equation*}
0<j_{-\alpha,1}<j_{-\alpha,2}<\cdots <j_{-\alpha,m}<\cdots
\end{equation*}
are the zeros of $J_{-\alpha}$, $\gamma_m:=\Big\{\int_0^{2R} t
\big[J_{-\alpha}\big(\frac{j_{-\alpha,m}}{2R}\,
t\big)\big]^2 dt\Big\}^{-1/2}$, $\{e_n\}_{n\ge 1}$ denotes a
complete system, orthonormal in $L^2(B_{2R}'(x_0))$, of
eigenfunctions of $-\Delta$ in $B_{2R}'(x_0)$ with homogeneous
Dirichlet boundary conditions and $\mu_1<\mu_2\le \dots \le
\mu_n\le \dots$ the corresponding eigenvalues.

Then for any $n,m\in \N\setminus\{0\}$, $e_{n,m}$ is an
eigenfunction of \eqref{eq:eigen-prob} with corresponding
eigenvalue $\lambda_{n,m}$. Moreover the set $\{e_{n,m}:n,m\in
\N\setminus\{0\}\}$ is a complete orthonormal system for
$L^2(Q_{2R}^+(x_0);t^b)$.
\end{Proposition}

\begin{proof}
We look for nontrivial solutions of \eqref{eq:eigen-prob} in the
form
\begin{equation*}
U(x,t)=\sum_{n=1}^{+\infty} A_n(t)e_n(x) \, .
\end{equation*}
By \eqref{eq:EPW} with a choice of the test function of the type
$\varphi(x,t)=\eta(t)e_n(x)$, with $\eta=\eta(t)$ one variable
test function, we see that
\begin{equation} \label{eq:ode-An}
t^2 A_n''(t)+btA_n'(t)+(\lambda-\mu_n)t^2 A_n(t)=0
\end{equation}
and
\begin{equation} \label{eq:bc-ode} \lim_{t\to 0^+} t^b
A_n'(t)=0 \, , \qquad A_n(2R)=0 \, .
\end{equation}
If we put $z_n(t):=t^{-\alpha}A_n(t)$ then $z_n$ solves
\begin{equation} \label{eq:zn}
t^2 z_n''(t)+tz_n'(t)+[(\lambda-\mu_n)t^2-\alpha^2]z_n(t)=0 \, .
\end{equation}
When $\lambda-\mu_n=0$, \eqref{eq:zn} becomes an Euler equation.
Therefore, exploiting the explicit representation of solutions of
\eqref{eq:zn}, one can check that problem
\eqref{eq:ode-An}-\eqref{eq:bc-ode} does not admit any nontrivial
solution.

Let us assume that $\lambda-\mu_n\neq 0$. Depending on the sign of
$\lambda-\mu_n$ one may proceed with the following rescaling
\begin{equation*}
y_n(t):=z_n\Big(\tfrac{t}{\sqrt{|\lambda-\mu_n|}}\Big)
\end{equation*}
to obtain
\begin{equation} \label{eq:bessel-type}
t^2 y_n''(t)+ty_n'(t)+{\rm sign}(\lambda-\mu_n) t^2
y_n(t)-\alpha^2 y_n(t)=0 \, .
\end{equation}
If $\lambda-\mu_n>0$, \eqref{eq:bessel-type} becomes a Bessel
equation while if $\lambda-\mu_n<0$ it becomes a ``modified Bessel
equation'', see Section 4.5 and Section 4.12 in \cite{AAR}.

By \eqref{eq:bc-ode}, (4.12.2), (4.12.4), (4.6.1) and (4.6.2) in \cite{AAR} (and the analogue for modified Bessel functions), and
some tedious computations, one can check that the only possible
way to find nontrivial solutions of
\eqref{eq:ode-An}-\eqref{eq:bc-ode} is to assume that
$\lambda-\mu_n>0$ and to choose
\begin{equation*}
y_n(t)=c_n J_{-\alpha}(t) \, .
\end{equation*}
This implies that any nontrivial $A_n$ admits necessarily the
representation
\begin{equation} \label{eq:An}
A_n(t)=c_n t^\alpha J_{-\alpha}(\sqrt{\lambda-\mu_n}t)
\end{equation}
with $\lambda$ satisfying $J_{-\alpha}(2\sqrt{\lambda-\mu_n}\,
R)=0$ whenever $c_n\neq 0$.

From this we deduce that $\lambda$ necessarily satisfies
\begin{equation} \label{eq:(n,m)}
\lambda=\mu_n+\tfrac{j_{-\alpha,m}^2}{4R^2} \, , \qquad \text{for
some } n,m\in \N,\ n,m\geq 1 .
\end{equation}
This proves that the eigenvalues of $-\Delta_b$ are the numbers
which admits the representation \eqref{eq:(n,m)}.

\noindent For any number $\lambda>0$ we denote by $S(\lambda)$ the
possibly empty set defined by
$$
S(\lambda):=\{(n,m)\in
(\N\setminus\{0\})^2
:\eqref{eq:(n,m)} \ \text{holds true}\}
\, .
$$
For any $\lambda>0$, the set $S(\lambda)$ is finite since
$$
\lim_{n\to +\infty} \mu_n=+\infty  \qquad \text{and} \qquad
\lim_{m\to +\infty} j_{-\alpha,m}=+\infty \, .
$$
Hence if $\lambda$ is an eigenvalue then the corresponding
eigenfunctions $U$ are of the form
\begin{equation*}
U(x,t)=\sum_{(n,m)\in S(\lambda)} c_{n,m} \, t^\alpha
J_{-\alpha}\left(\tfrac{j_{-\alpha,m}}{2R}\, t\right)e_n(x) \, .
\end{equation*}
For any $(n,m)\in
(\N\setminus\{0\})^2$, it may be useful to define
\begin{equation} \label{eq:e-nm}
e_{n,m}(x,t):=\gamma_m \, t^\alpha
J_{-\alpha}\left(\tfrac{j_{-\alpha,m}}{2R}\, t\right)e_n(x)
\end{equation}
where $\gamma_m:=\Big\{\int_0^{2R} t
\big[J_{-\alpha}\big(\frac{j_{-\alpha,m}}{2R}\,
t\big)\big]^2 dt\Big\}^{\!-1/2}$.

With this choice we have that
$\|e_{n,m}\|_{L^2(Q_{2R}^+(x_0);t^b)}=1$. Moreover we also have
orthogonality in $L^2(Q_{2R}^+(x_0);t^b)$ of two distinct
eigenfunctions $e_{n_1,m_1},e_{n_2,m_2}$. If $n_1\neq n_2$, this
follows from Fubini-Tonelli Theorem and the orthogonality of
$e_{n_1}$ and $e_{n_2}$ in $L^2(B_{2R}'(x_0))$. If $n_1=n_2$ and
$m_1\neq m_2$ orthogonality between $e_{n_1,m_1}$ and
$e_{n_2,m_2}$ follows from Fubini-Tonelli Theorem and  the fact
that
\begin{align*}
 \int_0^{2R} t& J_{-\alpha}\left(\frac{j_{-\alpha,m_1}}{2R}\,
t\right) J_{-\alpha}\left(\frac{j_{-\alpha,m_2}}{2R}\, t\right) \,
dt =4R^2
\int_0^{1} t J_{-\alpha}(j_{-\alpha,m_1}t) J_{-\alpha}(j_{-\alpha,m_2}t) \, dt\\
& =\frac{4R^2}{j_{-\alpha,m_1}^2-j_{-\alpha,m_2}^2}
[J_{-\alpha}(j_{-\alpha,m_1})J_{-\alpha}'(j_{-\alpha,m_2})-J_{-\alpha}(j_{-\alpha,m_2})J_{-\alpha}'(j_{-\alpha,m_1})]=0
\end{align*}
where the last identity follows from \cite[Equation
(4.14.2)]{AAR}.

Therefore the set $\{e_{n,m}:n,m\in \N\setminus\{0\}\}$ is a
complete orthonormal system for $L^2(Q_{2R}^+(x_0);t^b)$ of
eigenfunctions of $-\Delta_b$ with eigenvalues
\begin{equation} \label{eq:eigenvalue}
\lambda_{n,m}:=\mu_n+\frac{j_{-\alpha,m}^2}{4R^2} \, .
\end{equation}
We observe that $\{e_{n,m}:n,m\in\N\setminus\{0\}\}$ is a complete
system thanks to compactness of the embedding
$H^1_0(\Gamma_{2R}^+(x_0);t^b)\subset L^2(Q_{2R}^+(x_0);t^b)$ (see
Proposition \ref{p:compactness-2}) and to the theory of compact
self-adjoint operators.
\end{proof}

In the next proposition we prove some estimates on the
eigenfunctions of \eqref{eq:eigen-prob}.

\begin{Proposition} Suppose that all the assumptions of
Proposition \ref{p:eigenfunctions} hold true. Let $e_{n,m}$ and
$\lambda_{n,m}$ be as in Proposition \ref{p:eigenfunctions}. Then
for any $n,m\in \N\setminus\{0\}$ and $k\ge 0$,
$e_{n,m}\in C^k\left(\overline{Q_{2R}^+(x_0)}\right)$ and, letting $\delta=[N/4]+[(k+1)/2]+1$,
with $[\cdot]$ denoting the integer part of a number, we have
\begin{equation*}
\|e_{n,m}\|_{C^k\left(\overline{Q_{2R}^+(x_0)}\right)}=
\begin{cases}
O\left(\lambda_{n,m}^{\frac k2+\delta+\frac{1}4}\right) & \text{if
} \alpha\in\left[\frac 12,1\right) \, , \\[5pt]
O\left(\lambda_{n,m}^{\frac{k-\alpha+1}2+\delta}\right) & \text{if
} \alpha\in\left(0,\frac 12\right) \, ,
\end{cases}
 \qquad \text{as
} |(n,m)|=\sqrt{n^2+m^2} \to +\infty \, .
\end{equation*}
Moreover we also have that
\begin{equation} \label{eq:enm-t}
\lim_{t\to 0^+} \partial_t e_{n,m}(\cdot,t)=0
\end{equation}
uniformly in $B_{2R}'(x_0)$.
\end{Proposition}

\begin{proof}
Combining elliptic estimates (see \cite[Chapter V]{ADN}) and
Sobolev embeddings with the fact that
$\|e_n\|_{L^2(B_{2R}'(x_0)))}=1$, we have that, for any $k\in
\N$, there exists a constant $C(N,R,k)$ depending only on
$N,R$ and $k$ such that
\begin{equation} \label{eq:ell-reg}
\|e_n\|_{C^k(B_{2R}'(x_0))}\le C(N,R,k) \mu_n^{\delta}
\end{equation}
with $\delta$ as in the statement of the lemma.

In order to obtain a similar estimate for the function $\gamma_m
\, t^\alpha J_{-\alpha}\left(\tfrac{j_{-\alpha,m}}{2R}\, t\right)$
we first observe that
\begin{align} \label{eq:gamma-m}
\gamma_m & =\left[\left(\frac{2R}{j_{-\alpha,m}}\right)^{2}
\int_0^{j_{-\alpha,m}} t (J_{-\alpha}(t))^2 dt \right]^{-1/2} \le
\left[4R^2 \int_0^{j_{-\alpha,1}} t (J_{-\alpha}(t))^2 dt
\right]^{-1/2} j_{-\alpha,m} \, .
\end{align}
By \eqref{eq:Bessel-infty} and \eqref{eq:iterative-Bessel-2} in
Subsection \ref{ss:7.3} and direct computation one may check that
the function $h(t):=t^\alpha J_{-\alpha}(t)\in
C^\infty([0,\infty))$ (see the series expansion of first kind
Bessel functions in \cite[Section 4.5]{AAR}) satisfies
\begin{equation} \label{eq:hk}
|h^{(k)}(t)|\le
\begin{cases}
C(\alpha,k) (1+t^{\alpha-1/2}) \qquad \text{for any } t\ge 0\, , &
\quad \text{if } \alpha\in \left(\frac 12,1\right) \, , \\[5pt]
C(\alpha,k) \qquad \text{for any } t\ge 0\, , & \quad \text{if }
\alpha\in \left(0,\frac 12\right] \, ,
\end{cases}
\end{equation}
for any $k\in \N$, where $C(\alpha,k)$ is a positive
constant depending only on $\alpha$ and $k$. Here and in the
sequel by derivative of order zero of a function we mean the
function itself.

By \eqref{eq:hk} we obtain
\begin{align} \label{eq:hk2}
& \left\|\frac{d^k}{dt^k} \left(t^\alpha
J_{-\alpha}\left(\frac{j_{-\alpha,m}}{2R}\, t\right)\right)
\right\|_{L^\infty(0,2R)}
=\left(\frac{j_{-\alpha,m}}{2R}\right)^{-\alpha}
\left\|\frac{d^k}{dt^k} \left(h\left(\frac{j_{-\alpha,m}}{2R}\,
t\right)
\right)\right\|_{L^\infty(0,2R)}\\[8pt]
\notag & =\left(\frac{j_{-\alpha,m}}{2R}\right)^{-\alpha+k}
\left\|h^{(k)}\left(\frac{j_{-\alpha,m}}{2R}\, t\right)
\right\|_{L^\infty(0,2R)} \le
\left(\frac{j_{-\alpha,m}}{2R}\right)^{-\alpha+k}
C(\alpha,k)[1+(j_{-\alpha,m})^{\alpha-1/2}]
\end{align}
in both the cases $\alpha\in\left(\frac 12,1\right)$ and
$\alpha\in \left(0,\frac 12\right]$.

By \eqref{eq:gamma-m} and \eqref{eq:hk2} we deduce that for any
$k\in \N$
\begin{equation*}
\left\|\frac{d^k}{dt^k}\left(\gamma_m \, t^\alpha
J_{-\alpha}\left(\tfrac{j_{-\alpha,m}}{2R}\,
t\right)\right)\right\|_{L^\infty(0,2R)}=
\begin{cases}
O((j_{-\alpha,m})^{k+1/2}) & \quad \text{if }
\alpha\in\left(\frac 12,1\right) \, , \\[5pt]
O((j_{-\alpha,m})^{k-\alpha+1}) & \quad \text{if } \alpha\in
\left(0,\frac 12\right] \, ,
\end{cases}
\qquad \text{as } m\to +\infty \, .
\end{equation*}
Combining this estimate with \eqref{eq:e-nm},
\eqref{eq:eigenvalue} and \eqref{eq:ell-reg}, we obtain
\begin{equation} \label{eq:stima-C^k}
\|e_{n,m}\|_{C^k\left(\overline{Q_{2R}^+(x_0)}\right)}=
\begin{cases}
O\left(\lambda_{n,m}^{\frac k2+\delta+\frac14}\right) & \quad \text{if
} \alpha\in\left(\frac 12,1\right) \, , \\[5pt]
O\left(\lambda_{n,m}^{\frac{k-\alpha+1}2+\delta}\right) & \quad \text{if
} \alpha\in\left(0,\frac 12\right] \, ,
\end{cases}
 \qquad \text{as
} |(n,m)|\to +\infty \, .
\end{equation}
This completes the first part of the proof of the proposition.

Finally, from the series expansion of first kind Bessel functions,
see \cite[Section 4.5]{AAR}, we infer that $\lim_{t\to 0^+}(
t^\alpha J_{-\alpha}(t))'=0$ which, together with \eqref{eq:e-nm},
implies $\lim_{t\to 0^+}
\partial_t e_{n,m}(\cdot,t)=0$ uniformly in $B_{2R}'(x_0)$. This
completes the proof of the proposition.
\end{proof}

Given a function $\psi\in C^\infty_c\left(Q_{2R}^+(x_0)\right)$,
consider the following Poisson equation
\begin{equation} \label{eq:Poisson}
\begin{cases}
-\Delta_b \varphi=\psi  & \qquad \text{in } Q_{2R}^+(x_0) \, ,\\[5pt]
\varphi=0 & \qquad \text{on } [\partial B_{2R}'(x_0)\times
(0,2R)]\cup
[B_{2R}'(x_0) \times \{2R\}] \, , \\[5pt]
{\ds \lim_{t\to 0^+} t^b \varphi_t(\cdot,t)=0 } & \qquad \text{on
} B_{2R}'(x_0) \times \{0\} \, .
\end{cases}
\end{equation}
We prove below the existence of a smooth solution to \eqref{eq:Poisson}.

\begin{Proposition} Let $b\in (-1,1)$, $x_0\in \Omega$ and let $R>0$ be such that $B_{2R}'(x_0)\subset
\Omega$. Then for any $\psi\in
C^\infty_c\left(Q_{2R}^+(x_0)\right)$, \eqref{eq:Poisson} admits a
unique solution $\varphi\in
C^\infty\left(\overline{Q_{2R}^+(x_0)}\right)$. Moreover $\varphi$
satisfies
\begin{equation*}
\lim_{t\to 0^+} \varphi_t(\cdot,t)=0
\end{equation*}
uniformly in $B_{2R}'(x_0)$.
\end{Proposition}

\begin{proof}
The datum $\psi$ can be written in the form
\begin{equation*}
\psi(x,t)=\sum_{n,m=1}^{+\infty} c_{n,m} \, e_{n,m}(x,t) \, .
\end{equation*}
Then the solution $\varphi$ of \eqref{eq:Poisson} is formally
given by
\begin{equation*}
\varphi(x,t)=\sum_{n,m=1}^{+\infty} \frac{c_{n,m}}{\lambda_{n,m}}
\, e_{n,m}(x,t)  \, .
\end{equation*}
We observe that by integration by parts and the fact that
$e_{n,m}$ is an eigenfunction of $-\Delta_b$ corresponding to the
eigenvalue $\lambda_{n,m}$, we have
\begin{align*}
c_{n,m}& =\int_{Q_{2R}^+(x_0)} t^b \psi e_{n,m} \, dx\,dt=\frac
1{\lambda_{n,m}}
\int_{Q_{2R}^+(x_0)} -t^b \psi \Delta_b e_{n,m}\, dx\,dt \\[8pt]
& =\frac 1{\lambda_{n,m}} \int_{Q_{2R}^+(x_0)} -t^b \Delta_b \psi
\, e_{n,m}\, dx\,dt \, .
\end{align*}
Iterating this procedure, we deduce that, for any $\ell \in \N$,
\begin{align} \label{eq:c-nm}
c_{n,m}=\frac 1{\lambda_{n,m}^\ell} \int_{Q_{2R}^+(x_0)} t^b
(-\Delta_b)^\ell \psi \, e_{n,m}\, dx\,dt=:\frac
1{\lambda_{n,m}^\ell} \, d_{n,m,\ell} \, .
\end{align}
Since $\psi\in C^\infty_c(Q_{2R}^+(x_0))$ then $(-\Delta_b)^\ell
\psi\in C^\infty_c(Q_{2R}^+(x_0))$ and hence $(-\Delta_b)^\ell
\psi\in L^2(Q_{2R}^+(x_0);t^b)$. This yields
$\sum_{n,m=1}^{+\infty} d_{n,m,\ell}^2<+\infty$ and, in turn,
$\lim_{|(n,m)|\to +\infty} d_{n,m,\ell}=0$. This, combined with
\eqref{eq:c-nm}, shows that for any $\ell\in \N$
\begin{equation} \label{eq:ell}
c_{n,m}=o(\lambda_{n,m}^{-\ell}) \qquad \text{as } |(n,m)|\to
+\infty \, .
\end{equation}
By \eqref{eq:stima-C^k} and \eqref{eq:ell}, we obtain as $|(n,m)|\to +\infty$
\begin{equation} \label{eq:stima-Ck-bis}
 \left\|\frac{c_{n,m}}{\lambda_{n,m}} \, e_{n,m}\right\|_{C^k\left(\overline{Q_{2R}^+(x_0)}\right)}
=\begin{cases} O(c_{n,m} \, \lambda_{n,m}^{\frac k2+\delta-\frac 34})=o(\lambda_{n,m}^{\frac k2+\delta-\frac 34-\ell}) &
 \text{if } \alpha\in \left(\frac 12,1\right) \, , \\[5pt]
O(c_{n,m} \,
\lambda_{n,m}^{\frac{k-\alpha-1}2+\delta})=o(\lambda_{n,m}^{\frac{k-\alpha-1}2+\delta-\ell})
&  \text{if } \alpha\in \left(0,\frac 12\right] \, .
\end{cases}
\end{equation}
We put $L:=\ell-\frac k2-\delta+\frac 34$ if $\alpha\in\left(\frac
12,1\right)$ and $L:=\ell-\frac{k-\alpha-1}2-\delta$ if $\alpha\in
\left(0,\frac 12\right]$. We may fix $\ell$ large enough such that
$L>N$ in both cases.

By \eqref{eq:eigenvalue}, \eqref{eq:zero-bessel} and Weyl Law
for the asymptotic behavior of eigenvalues of $-\Delta$ with
Dirichlet boundary conditions, we infer that there exists a
constant $C>0$ such that
\begin{equation*}
\lambda_{n,m}\ge C (n^{\frac 2N}+m^2)\ge C(n^{\frac 2N}+m^{\frac 2N})\ge C(n^2+m^2)^{\frac 1N} \qquad
\text{for any } n,m\ge 1 \, .
\end{equation*}
Combining this with \eqref{eq:stima-Ck-bis} we obtain
\begin{equation*}
 \left\|\frac{c_{n,m}}{\lambda_{n,m}} \, e_{n,m}\right\|_{C^k\left(\overline{Q_{2R}^+(x_0)}\right)}=o\Big((n^2+m^2)^{-\frac LN}\Big)
 \qquad \text{as } |(n,m)|\to +\infty \, .
\end{equation*}
Since $L>N$, this proves that
\begin{equation*}
\sum_{n,m=1}^{+\infty} \left\|\frac{c_{n,m}}{\lambda_{n,m}} \,
e_{n,m}\right\|_{C^k\left(\overline{Q_{2R}^+(x_0)}\right)}<+\infty
\end{equation*}
for any $k\in \N$ thus showing that $\varphi\in
C^\infty\left(\overline{Q_{2R}^+(x_0)}\right)$.

Finally, by \eqref{eq:enm-t} we also have
\begin{equation} \label{eq:phi-t}
\lim_{t\to 0^+} \varphi_t(\cdot,t)=0 \qquad \text{uniformly in }
B_{2R}'(x_0) \, .
\end{equation}
This completes the proof of the proposition.
\end{proof}

We are ready to prove the main result of this section.

\begin{Proposition} \label{p:A2}
Let $\Omega\subset \R^N$ be open. Let $s\in
(1,2)$ and $b=3-2s$. Let $g\in (\mathcal D^{s-1,2}(\R^N))^\star$, $f\in L^2_{{\rm loc}}(\R^{N+1}_+;t^b)$
and let $V\in L^2(\R^{N+1}_+;t^b)$ be a distributional solution of
the problem
\begin{equation} \label{eq:reg-H1}
\begin{cases}
\dive (t^b \nabla V)=t^bf & \qquad \text{in } \R^{N+1}_+, \\[5pt]
{\ds \lim_{t\to 0^+} t^b V_t(\cdot,t)=g} & \qquad \text{in }
\Omega,
\end{cases}
\end{equation}
namely
\begin{align*}
\int_{\R^{N+1}_+} V \dive(t^b \nabla \varphi) \, dx\,dt=\int_{\R^{N+1}_+} t^b f\varphi\, dx\,dt \qquad
\text{for any } \varphi\in C^\infty_c(\R^{N+1}_+)
\end{align*}
and
\begin{align} \label{eq:ultradebole}
\int_{\R^{N+1}_+} V \dive(t^b \nabla \varphi) \,
dx\,dt=\int_{\R^{N+1}_+} t^b f\varphi\, dx\,dt+\!\!\!\! \phantom{A}_{ (\mathcal D^{s-1,2}(\R^N))^\star}\Big\langle
g,\varphi(x,0)\Big\rangle_{\mathcal D^{s-1,2}(\R^N)}
\end{align}
for any $\varphi\in C^\infty_c\big(\overline{\R^{N+1}_+}\big)$ such that
${\rm supp}(\varphi(\cdot,0))\subset \Omega$ and ${\ds \lim_{t\to
0^+} \varphi_t(\cdot,t)\equiv 0}$ in $\R^N$.

Then $V\in H^1(Q_R^+(x_0);t^b)$ for any $x_0\in \Omega$ and $R>0$
satisfying $B_{2R}'(x_0)\subset \Omega$ and moreover there exists a positive constant $C$ depending
only on $N,b,x_0,R$ such that
\begin{equation}  \label{eq:cont-dep}
\|V\|_{H^1(Q_R^+(x_0);t^b)}\le C\Big(\|f\|_{L^2(Q_{2R}^+(x_0);t^b)}+\|g\|_{(\mathcal D^{s-1,2}(\R^N))^\star}+\|V\|_{L^2(\R^{N+1}_+;t^b)}\Big) \, .
\end{equation}
\end{Proposition}
 \begin{proof} Let $x_0\in \Omega$ and let $R>0$ be such that $B_{2R}'(x_0)\subset \Omega$.
Let $\eta_0\in C^\infty([0,\infty))$ be such that $0\le \eta_0\le
1$ in $[0,\infty)$, $\eta_0\equiv 1$ in $[0,R]$ and $\eta_0\equiv
0$ in $[2R,\infty)$. We now define $\eta:\R^{N+1}\to \R$ as
$\eta(x,t):=\eta_0(|x-x_0|)\eta_0(t)$ for any $(x,t)\in
\R^{N+1}_+$ and $W(x,t):=\eta(x,t)V(x,t)$ for any
$(x,t)\in \R^{N+1}_+$.

By \eqref{eq:ultradebole} and the fact that ${\ds \lim_{t\to
0^+} (\eta\varphi)_t(\cdot,t)\equiv 0}$ in $\R^N$ for any function
$\varphi\in C^\infty_c\big(\overline{\R^{N+1}_+}\big)$ satisfying ${\rm
supp}(\varphi(\cdot,0))\subset \Omega$ and ${\ds \lim_{t\to 0^+}
\varphi_t(\cdot,t)\equiv 0}$ in $\Omega$, it turns out that
\begin{multline}  \label{eq:ultradebole-2-0}
\int_{\R^{N+1}_+} W \dive(t^b \nabla \varphi) \, dx\,dt=\int_{\R^{N+1}_+} t^b f\eta\varphi\, dx\,dt+\!\!
\phantom{A}_{ (\mathcal D^{s-1,2}(\R^N))^\star}\Big\langle
g,\eta(x,0)\varphi(x,0)\Big\rangle_{\mathcal
D^{s-1,2}(\R^N)}\\
\qquad -\int_{\R^{N+1}_+} V\left[\dive(t^b \nabla
\eta)\varphi+2t^b \nabla\eta\nabla\varphi\right]\, dx\,dt,
\end{multline}
 where we exploited the identity
$\eta\dive(t^b\nabla\varphi)=\dive (t^b
\nabla(\eta\varphi))-2t^b\nabla\eta\nabla\varphi-\dive
(t^b\nabla\eta)\varphi$.

From this we can deduce that $W$ is a solution of the problem
\begin{align} \label{eq:problema-ultradebole}
\left\{\hskip-10pt
\begin{array}{ll}
&W\in L^2(Q_{2R}^+(x_0);t^b) , \\[8pt]
&{\ds \int_{Q_{2R}^+(x_0)} W \dive(t^b \nabla \varphi) \, dx\,dt=\int_{Q_{2R}^+(x_0)} t^b f\eta\varphi\, dx\,dt } \\[12pt]
&{\ds\,  +\!\!\!\phantom{A}_{ (\mathcal D^{s-1,2}(\R^N))^\star}\Big\langle
g,\eta(x,0)\varphi(x,0)\Big\rangle_{\mathcal
D^{s-1,2}(\R^N)}\!\!
-\!\int_{Q_{2R}^+(x_0)} V\left[\dive(t^b \nabla \eta)\varphi+2t^b \nabla\eta\nabla\varphi\right]\, dx\,dt } \\[15pt]
&\text{for any }\varphi\in C^\infty(\overline{Q_{2R}^+(x_0)})\text{
  such that }
\varphi\equiv 0  \text{ on } \partial Q_{2R}(x_0)\cap \R^{N+1}_+\\[3pt]
&\text{and }
 {\ds \lim_{t\to 0^+} \varphi_t(\cdot,0)\equiv 0}  \text{ in
} B_{2R}'(x_0),
\end{array}\right.
\end{align}
where the duality product has to be interpreted as applied to a
trivial extension of $\eta \varphi$.

We divide the remaining part of the proof into three steps.

{\bf Step 1.} We prove that given $V,g$ as in the statement and
$\eta$ as above, there exists a unique solution of
\eqref{eq:problema-ultradebole}.

 Suppose that $W_1,W_2$ are two of these functions and denote by $W$ their difference. Then we have that
$W\in L^2(Q_{2R}^+(x_0);t^b)$ and it satisfies
\begin{equation} \label{eq:ultradebole-2-omogeneo}
\int_{Q_{2R}^+(x_0)} W \dive(t^b \nabla \varphi) \, dx\,dt=0
\end{equation}
for any $\varphi\in C^\infty(\overline{Q_{2R}^+(x_0)})$ with
$\varphi\equiv 0$ on $\partial Q_{2R}(x_0)\cap \R^{N+1}_+$ and
${\ds \lim_{t\to 0^+} \varphi_t(\cdot,t)\equiv 0}$ in
$B_{2R}'(x_0)$.

Let $\psi\in C^\infty_c(Q_{2R}^+(x_0))$ and let $\varphi$ be the
unique solution of \eqref{eq:Poisson}. We have shown that such a
function $\varphi$ belongs to
$C^\infty(\overline{Q_{2R}^+(x_0)})$. This together with
\eqref{eq:phi-t} implies that $\varphi$ is an admissible test
function in \eqref{eq:ultradebole-2-omogeneo}. This yields
\begin{equation*}
\int_{Q_{2R}^+(x_0)} t^b W \psi \, dx\,dt=-\int_{Q_{2R}^+(x_0)} W
\dive(t^b \nabla \varphi) \, dx\,dt=0
\end{equation*}
for any $\psi\in C^\infty_c(Q_{2R}^+(x_0))$. This shows that
$W\equiv 0$ in $Q_{2R}^+(x_0)$ and completes the proof of Step 1.

{\bf Step 2.} In this step we prove that, for $V,g$ as in the
statement of the proposition and $\eta$ as above, there exists a
unique function $Z\in H^1_0(\Gamma_{2R}^+(x_0);t^b)$ such that
\begin{multline} \label{eq:uniqueness}
\int_{Q_{2R}^+(x_0)} t^b \nabla Z \nabla \varphi \, dx\,dt=-\int_{Q_{2R}^+(x_0)} t^b f\eta\varphi\, dx\,dt\\
-\!\! \phantom{A}_{ (\mathcal D^{s-1,2}(\R^N))^\star}\Big\langle
g,\eta(x,0)\varphi(x,0)\Big\rangle_{\mathcal D^{s-1,2}(\R^N)}
+\int_{Q_{2R}^+(x_0)} V\left[\dive(t^b \nabla \eta)\varphi+2t^b
\nabla\eta\nabla\varphi\right]\, dx\,dt
\end{multline}
for any $\varphi\in H^1_0(\Gamma_{2R}^+(x_0);t^b)$. We recall that
there exists a well-defined continuous trace embedding from
$\mathcal D^{1,2}(\R^{N+1}_+;t^b)$ into $\mathcal
D^{s-1,2}(\R^N)$, see \eqref{eq:C-b}. We observe that for any
$\varphi\in H^1_0(\Gamma_{2R}^+(x_0);t^b)$ the function
$\eta\varphi$, once it is trivially extended outside
$Q_{2R}^+(x_0)$, belongs to $\mathcal D^{1,2}(\R^{N+1}_+;t^b)$. We
denote the trace of $\eta\varphi$ simply by
$\eta(\cdot,0)\varphi(\cdot,0)\in \mathcal D^{s-1,2}(\R^N)$. We
have
\begin{align} \label{eq:Lax1}
\left|\phantom{A}_{ (\mathcal D^{s-1,2}(\R^N))^\star}\Big\langle
g,\eta(x,0)\varphi(x,0)\Big\rangle_{\mathcal D^{s-1,2}(\R^N)}
\right| &\leq \|g\|_{ (\mathcal
D^{s-1,2}(\R^N))^\star}\|\eta(\cdot,0)\varphi(\cdot,0)\|_{\mathcal
D^{s-1,2}(\R^N)}\\
\notag&\leq {\rm const\,}\|g\|_{ (\mathcal
D^{s-1,2}(\R^N))^\star}\|\eta\varphi\|_{\mathcal D^{1,2}(\R^{N+1}_+;t^b)}\\
\notag&\leq {\rm const\,}\|g\|_{ (\mathcal
D^{s-1,2}(\R^N))^\star}\|\varphi\|_{H^1_0(\Gamma_{2R}^+(x_0);t^b)}
\end{align}
for some ${\rm const\,}>0$ depending only on $N,R,b$ and $\eta$.

On the other hand, from the fact that $\eta_t(\cdot,0)\equiv 0$ in
$\Omega$ and by \eqref{eq:PA2-2}, we deduce that
\begin{align} \label{eq:Lax2}
& \left|\int_{Q_{2R}^+(x_0)} V \dive(t^b \nabla \eta)\varphi \,
dx\,dt\right| \le
\left(b\|\eta_t/t\|_{L^\infty(\R^{N+1}_+)}+\|\Delta
\eta\|_{L^\infty(\R^{N+1}_+)}\right) \int_{Q_{2R}^+(x_0)} t^b
|V|\,
|\varphi| \, dx\,dt \\
\notag & \le \left(b\|\eta_t/t\|_{L^\infty(\R^{N+1}_+)}+\|\Delta
\eta\|_{L^\infty(\R^{N+1}_+)}\right)
\|V\|_{L^2(Q_{2R}^+(x_0);t^b)} \tfrac{4\sqrt 2 R}{N+b-1}\,
\|\varphi\|_{H^1_0(\Gamma_{2R}^+(x_0);t^b)}
\end{align}
and
\begin{align} \label{eq:Lax2-bis}
& \left|\int_{Q_{2R}^+(x_0)}t^b f\eta\varphi\, dx\,dt \right|\le \tfrac{4\sqrt 2 R}{N+b-1} \, \|f\|_{L^2(Q_{2R}^+(x_0);t^b)}\, \|\varphi\|_{H^1_0(\Gamma_{2R}^+(x_0);t^b)}
\end{align}
for any $\varphi\in H^1_0(\Gamma_{2R}^+(x_0);t^b)$.

Finally we have
\begin{align} \label{eq:Lax3}
\left|\int_{Q_{2R}^+(x_0)} V t^b\nabla\eta \nabla\varphi \,
dx\,dt\right|\le \|\nabla \eta\|_{L^\infty(\R^{N+1}_+)}
\|V\|_{L^2(Q_{2R}^+(x_0);t^b)}\,
\|\varphi\|_{H^1_0(\Gamma_{2R}^+(x_0);t^b)} \, .
\end{align}
From \eqref{eq:Lax1}-\eqref{eq:Lax3} and the Lax-Milgram Theorem we
deduce that \eqref{eq:uniqueness} admits a unique solution $Z\in
H^1_0(\Gamma_{2R}^+(x_0);t^b)$. An integration by parts yields
\begin{equation*}
\int_{B_{2R}'(x_0)\times (\e,2R)} t^b \nabla Z \nabla \varphi \,
dx\,dt=-\int_{B_{2R}'(x_0)} \e^b Z(x,\e) \varphi_t(x,\e)\,
dx-\int_{B_{2R}'(x_0)\times (\e,2R)} Z \dive(t^b \nabla\varphi) \,
dx\,dt
\end{equation*}
for any $\varphi\in C^\infty(\overline{Q_{2R}^+(x_0)})\cap
H^1_0(\Gamma_{2R}^+(x_0);t^b)$ satisfying $\lim_{t\to 0^+}
\varphi_t(\cdot,t)=0$ uniformly in $B_{2R}'(x_0)$. Passing to the
limit as $\e\to 0^+$ we obtain
\begin{equation} \label{eq:int-parts}
\int_{Q_{2R}^+(x_0)} t^b \nabla Z \nabla \varphi \,
dx\,dt=-\int_{Q_{2R}^+(x_0)} Z \dive(t^b \nabla\varphi) \, dx\,dt \, .
\end{equation}
Actually, one has to prove first \eqref{eq:int-parts} for smooth
functions $Z$ and then, by a density argument, for all functions in
$H^1_0(\Gamma_{2R}^+(x_0);t^b)$. Combining
\eqref{eq:uniqueness} and \eqref{eq:int-parts} we obtain
\begin{align} \label{eq:uniqueness-2}
& \int_{Q_{2R}^+(x_0)} Z \dive(t^b \nabla\varphi) \, dx\,dt=\int_{Q_{2R}^+(x_0)} t^b f\eta\varphi \, dx\,dt \\[8pt]
\notag & \  +\!\!\phantom{A}_{ (\mathcal
D^{s-1,2}(\R^N))^\star}\Big\langle
g,\eta(x,0)\varphi(x,0)\Big\rangle_{\mathcal
D^{s-1,2}(\R^N)}-\int_{Q_{2R}^+(x_0)} V\left[\dive(t^b \nabla
\eta)\varphi+2t^b \nabla\eta\nabla\varphi\right]\, dx\,dt
\end{align}
for any $\varphi\in C^\infty(\overline{Q_{2R}^+(x_0)})\cap
H^1_0(\Gamma_{2R}^+(x_0);t^b)$ with ${\ds \lim_{t\to 0^+}
\varphi_t(\cdot,t)\equiv 0}$ in $B_{2R}'(x_0)$. From this we
deduce that $Z$ is a solution of \eqref{eq:problema-ultradebole}.

{\bf Step 3.} We conclude the proof of the proposition. We have shown
that \eqref{eq:problema-ultradebole} admits a unique solution, hence
$Z$ coincides in $Q_{2R}^+(x_0)$ with the function $W=\eta V$ defined
at the beginning of the proof. In particular
$\eta V\in H^1(Q_{2R}^+(x_0);t^b)$ and, in turn,
$V\in H^1(Q_{R}^+(x_0);t^b)$ being $\eta\equiv 1$ in $Q_{R}^+(x_0)$.
The proof of \eqref{eq:cont-dep} follows from the estimates of Step 2
and standard application of the continuous dependence from the data in
Lax-Milgram Theorem.
\end{proof}

\section{An alternative formulation of problem \eqref{eq:problema}} \label{s:alternative-formulation}

 Inspired by \cite{CS} and \cite{Y}, we introduce an alternative
formulation for problem \eqref{eq:problema}. For any $1<s<2$ as in
\eqref{eq:problema} we define $b:=3-2s\in (-1,1)$. Next we define
$\mathcal D_b$ as the completion of
\begin{equation} \label{eq:space}
\mathcal T=\left\{U\in C^\infty_c(\overline{\R^{N+1}_+}):U_t\equiv 0 \ \text{on} \
\R^N\times\{0\}
\right\}
\end{equation}
with respect to the norm
\[
\|U\|_{\mathcal D_b}=\bigg(\int_{\R^{N+1}_+}t^b |\Delta_b
U(x,t)|^2\,dx\,dt\bigg)^{\!\!1/2}.
\]
Let now $u\in \mathcal D^{s,2}(\R^N)$ be a solution of
\eqref{eq:problema} in the sense given in
\eqref{eq:problema-variazionale} and let $U\in \mathcal D_b$ be a solution of \eqref{eq:aux-prob}.

The existence of a solution for problem \eqref{eq:aux-prob} is
essentially contained in \cite{Y}. For completeness, we provide
here a rigorous formulation for \eqref{eq:aux-prob} and we prove
the existence of a solution for it.

In order to do that, we need to show that it is well defined and continuous the trace
map ${\rm Tr}:\mathcal D_b\to \mathcal D^{s,2}(\R^N)$ so that the
first boundary condition in \eqref{eq:aux-prob} can be interpreted
in the sense of traces. The construction of this trace operator is one of the main goals
of this section.

The second boundary condition in \eqref{eq:aux-prob} is a forced condition coming from the
functional space $\mathcal D_b$ and has the
following meaning: any function $U\in \mathcal D_b$ is the limit
with respect to the norm $\|\cdot\|_{\mathcal D_b}$ of a sequence
$\{U_n\}$ of smooth functions satisfying $\lim_{t\to 0^+} t^b
(U_n)_t(\cdot,t)\equiv 0$ in $\R^N$. In other words, the
boundary condition $\lim_{t \to 0^+ } t^b U_t(\cdot,0)\equiv 0$ on
$\R^N$ is equivalent to the validity of the following integration by
parts formula
\begin{equation}\label{eq:9}
\int_{\R^{N+1}_+} t^b \psi\Delta_b U \,dx\,dt= -\int_{\R^{N+1}_+}
t^b\nabla U \nabla\psi\,dx\,dt,\quad\text{for any }\psi\in
C^\infty_c(\overline{\R^{N+1}_+}).
\end{equation}
The previous arguments show that the minimization problem
\begin{equation} \label{eq:min-prob}
\min\left\{\int_{\R^{N+1}_+} t^b |\Delta_b V|^2 dx\,dt:V\in \mathcal
D_b \ \text{and} \ {\rm Tr}(V)=u \right\}
\end{equation}
is meaningful being the set $\{V\in \mathcal D_b:{\rm
Tr}(V)=u\}\neq \emptyset$ as one can deduce from
Lemma \ref{l:lemmone} and the density of $C^\infty_c(\R^N)$ in $\mathcal D^{s,2}(\R^N)$.
With
a standard procedure it is possible to verify that
\eqref{eq:min-prob} admits a minimizer $U\in \mathcal D_b$ which
is a weak solution of \eqref{eq:aux-prob}.

As mentioned above our main purpose now is to construct the trace map ${\rm Tr}:\mathcal D_b\to \mathcal D^{s,2}(\R^N)$.
We define the weighted Sobolev space $V(0,\infty;t^b)$ as the completion of
\begin{equation}\label{eq:1}
\left\{\varphi\in C^\infty_c([0,\infty)):\varphi'(0)=0 \right\}
\end{equation}
 with respect to the norm
\begin{equation}\label{eq:2}
\|\varphi\|_{V(0,\infty;t^b)}=\bigg(\int_0^{\infty}t^b\left[|\Delta_{b,t}\varphi|^2+
|\varphi'|^2+|\varphi|^2\right]dt\bigg)^{\!\!1/2}
\end{equation}
where $\Delta_{b,t}\varphi=\varphi''+\frac{b}{t}\varphi'$.

\begin{Lemma}\label{l:V0inft}
  Let $V(0,\infty;t^b)$ be the space defined in
  \eqref{eq:1}--\eqref{eq:2}. Then the following facts hold true:
  \begin{itemize}

  \item[(i)] $V(0,\infty;t^b)\subset C^1([0,\infty))$;

  \item[(ii)] $\varphi'',\frac{\varphi'}t\in L^2(0,\infty;t^b)$ and
    $\varphi'(0)=0$ for any $\varphi\in V(0,\infty;t^b)$;

  \item[(iii)] for any $\varphi\in V(0,\infty;t^b)$ there exists a constant $C>0$ independent of $t$ but possibly dependent on $\varphi$ such that
  \begin{equation} \label{eq:est-varphi}
  |\varphi(t)|\le C(1+t^{\frac{3-b}2})      \qquad \text{for any } t\ge 0 \, .
  \end{equation}

  \end{itemize}
\end{Lemma}

\begin{pf} We divide the proof of the lemma in several steps.

{\bf Step 1.} We prove that for any $\varphi$ as in \eqref{eq:1} we have
\begin{equation} \label{eq:id-1}
\int_0^{\infty}t^b\left[(\Delta_{b,t}\varphi)^2+
(\varphi')^2+\varphi^2\right]dt=\int_0^\infty t^b \left[|\varphi''(t)|^2+bt^{-2}|\varphi'(t)|^2+|\varphi'(t)|^2+|\varphi(t)|^2\right]dt.
\end{equation}
By direct computation we see that
\begin{equation}  \label{eq:id-2}
\int_0^\infty t^b \left(\varphi''(t)+bt^{-1}\varphi'(t)\right)^2 dt=
\int_0^\infty t^b\left[|\varphi''(t)|^2+b^2t^{-2} |\varphi'(t)|^2+2bt^{-1}\varphi'(t)\varphi''(t)\right]dt.
\end{equation}
By integration by parts and taking into account that $\lim_{t\to 0^+} t^{b-1} |\varphi'(t)|^2=0$ since $b>-1$, we obtain
\begin{equation} \label{eq:id-3}
\int_0^\infty t^{b-1} \varphi'(t)\varphi''(t)\, dt=-\frac{b-1}2\int_0^\infty t^{b-2} |\varphi'(t)|^2 dt \, .
\end{equation}
Combining \eqref{eq:id-2} and \eqref{eq:id-3} we obtain
\begin{equation*}
\int_0^\infty t^b \left(\varphi''(t)+bt^{-1}\varphi'(t)\right)^2 dt=\int_0^\infty t^b |\varphi''(t)|^2 dt+b\int_0^\infty t^{b-2} |\varphi'(t)|^2 dt
\end{equation*}
and the proof of Step 1 follows.

{\bf Step 2.} We prove that for any $\varphi$ as in \eqref{eq:1} we have
\begin{equation} \label{eq:id-5}
\frac{(b-1)^2}4 \int_0^\infty t^{b-2} |\varphi'(t)|^2 dt\le \int_0^\infty t^b |\varphi''(t)|^2 dt \, .
\end{equation}
Indeed, using \eqref{eq:id-3} we have
\begin{align*}
0 & \le \int_0^\infty \left(t^{\frac b2}\varphi''(t)+\frac{b-1}2 t^{\frac b2-1}\varphi'(t)\right)^2 dt \\
& =\int_0^\infty t^b |\varphi''(t)|^2 dt+\frac{(b-1)^2}4 \int_0^\infty t^{b-2}|\varphi'(t)|^2 dt+(b-1)\int_0^\infty t^{b-1}\varphi'(t)\varphi''(t)\, dt \\
& =\int_0^\infty t^b |\varphi''(t)|^2 dt-\frac{(b-1)^2}4 \int_0^\infty t^{b-2}|\varphi'(t)|^2 dt \, .
\end{align*}

{\bf Step 3.} We prove that the norm in \eqref{eq:2} and the norm
\begin{equation*}
\varphi\mapsto \left(\int_0^\infty t^b \left(|\varphi''(t)|^2+|\varphi'(t)|^2+|\varphi(t)|^2\right)dt\right)^{1/2}
\end{equation*}
are equivalent on the space defined in \eqref{eq:1}.

If $b\in [0,1)$ the equivalence of the two norms follows by \eqref{eq:id-1} and \eqref{eq:id-5}.

If $b\in (-1,0)$ one of the two estimate is trivial and for the other we proceed in this way:
\begin{align*}
\int_0^\infty t^b \left(|\varphi''(t)|^2+bt^{-2}|\varphi'(t)|^2\right)dt\ge \left(1+\frac{4b}{(b-1)^2}\right)\int_0^\infty t^b |\varphi''(t)|^2 dt
=\left(\frac{b+1}{b-1}\right)^2 \int_0^\infty t^b |\varphi''(t)|^2 dt
\end{align*}
where the above inequality follows from \eqref{eq:id-5} and the fact that $b<0$.

{\bf Step 4.} In this step we complete the proof of the lemma. From Step 2 and Step 3 and a density argument we deduce that
\begin{equation*}
\int_0^\infty t^b |\varphi''(t)|^2 dt\le C\|\varphi\|_{V(0,\infty;t^b)}^2 \quad \text{and} \quad \int_0^\infty t^{b-2} |\varphi'(t)|^2 dt\le C\|\varphi\|_{V(0,\infty;t^b)}^2
\end{equation*}
for any $\varphi\in V(0,\infty;t^b)$, where $C$ is a positive constant independent of $\varphi$. This proves the first two assertions in (ii).

For any $\varphi$ as in \eqref{eq:1} and $t>s>0$ we have, for some positive constant $C$ independent of $s$,$t$ and $\varphi$,
\begin{align} \label{eq:holder}
|\varphi(t)-\varphi(s)|&=\left|\int_s^t \tau^{\frac b2-1}\varphi'(\tau) \tau^{1-\frac b2}d\tau\right|\le
\left(\int_s^t \tau^{b-2}|\varphi'(\tau)|^2d\tau\right)^{1/2} \left(\int_s^t \tau^{2-b}d\tau\right)^{1/2} \\
\notag &\le C \|\varphi\|_{V(0,\infty;t^b)}\left|t^{3-b}-s^{3-b}\right|^{1/2}
\end{align}
where the last inequality follows from Step 2 and Step 3. By density we have that estimate \eqref{eq:holder} actually holds for any $\varphi\in V(0,\infty;t^b)$.
This proves that any $\varphi\in V(0,\infty;t^b)$ is continuous in $[0,+\infty)$ being $3-b>0$. Moreover if we put $s=0$ in \eqref{eq:holder} we obtain
\begin{equation} \label{eq:holder-2}
\left|\frac{\varphi(t)-\varphi(0)}t\right|\le C
\|\varphi\|_{V(0,\infty;t^b)} t^{\frac{1-b}2}
\qquad \text{and} \qquad |\varphi(t)|\le |\varphi(0)|+C \|\varphi\|_{V(0,\infty;t^b)} t^{\frac{3-b}2}.
 \end{equation}
Since $b<1$, from the first estimate in \eqref{eq:holder-2} we deduce that $\varphi$ is differentiable at $0$ and $\varphi'(0)=0$ so that the proof of (ii) is complete. The second estimate in \eqref{eq:holder-2} gives \eqref{eq:est-varphi} and proves (iii).

It remains to complete the proof of (i).
For any $\varphi$ as in \eqref{eq:1} and $t>s>0$ we have,
 for some positive constant $C$ independent of $s$,$t$ and $\varphi$,
\begin{align} \label{eq:holder-3}
|\varphi'(t)-\varphi'(s)|&=\left|\int_s^t \tau^{\frac b2}\varphi''(\tau) \tau^{-\frac b2}d\tau\right|\le
\left(\int_s^t \tau^{b}|\varphi''(\tau)|^2d\tau\right)^{1/2} \left(\int_s^t \tau^{-b}d\tau\right)^{1/2} \\
\notag &\le C \|\varphi\|_{V(0,\infty;t^b)}\left|t^{1-b}-s^{1-b}\right|^{1/2}
\end{align}
where the last inequality follows from Step 2 and Step 3. By density we have that estimate \eqref{eq:holder-3} actually holds for any $\varphi\in V(0,\infty;t^b)$. Since $b<1$, we deduce that $\varphi'$ is continuous in $[0,\infty)$ and this completes the proof of (i).
\end{pf}

Thanks to Lemma \ref{l:V0inft} we can now prove the existence of a classical solution of \eqref{eq:aux-prob} when the datum $u$ is sufficiently smooth.

\begin{Lemma} \label{l:lemmone} Let $u\in C^\infty_c(\R^N)$. Then
\eqref{eq:aux-prob} admits a classical solution $U\in C^2(\overline{\R^{N+1}_+})$. Moreover
$U\in \mathcal D_b$ and the following assertions hold true:
\begin{itemize}
\item[(i)] there exists a constant $C_b>0$ depending only on $b$
such that
\begin{equation} \label{eq:main-isometry}
\|U\|_{\mathcal D_b}=C_b \|u\|_{\mathcal D^{s,2}(\R^N)} \, ;
\end{equation}
\item[(ii)] for any $V\in C^\infty_c(\overline{\R^{N+1}_+})$ such that $V(\cdot,0)\equiv u$ and
$V_t(\cdot,0)\equiv 0$ in $\R^N$, we have
\begin{equation} \label{eq:minimale}
\|U\|_{\mathcal D_b}\le \|V\|_{\mathcal D_b}  \, .
\end{equation}
\end{itemize}
\end{Lemma}

\begin{proof}
Given a function $u\in C^\infty_c(\R^N)$ we aim to solve problem
\eqref{eq:aux-prob} by using the Fourier transform. Writing the
equation $\Delta_b^2 U=0$ as $\Delta_x^2 U+2\Delta_{b,t}\Delta_x
U+\Delta_{b,t}^2 U=0$ and applying the Fourier transform with
respect to the $x$ variable to both sides of the equation, we
formally obtain
\begin{equation} \label{eq:transformed-eq}
|\xi|^4\widehat{U}-2|\xi|^2\Delta_{b,t}\widehat{U}+\Delta_{b,t}^2\widehat{U}=0 \, .
\end{equation}
Following \cite{Y}, we look for a solution of \eqref{eq:transformed-eq} in the form $\widehat U(\xi,t)=\widehat u(\xi)\phi(|\xi|t)$ with $\phi(0)=1$ and $\phi'(0)=0$. From \eqref{eq:transformed-eq}, $\phi$ has to be a solution of the equation
\begin{equation} \label{eq:phi-equation}
\Delta_{b,t}^2\phi-2\Delta_{b,t}\phi+\phi=0 \, .
\end{equation}
We now divide the rest of the proof in several steps.

{\bf Step 1.} In this step we prove the existence of a solution to equation \eqref{eq:phi-equation} in $V(0,\infty;t^b)$.  We introduce the functional
$J:V(0,\infty;t^b)\to \R$ defined as
\[
J(\varphi)=\int_{0}^{\infty}t^b\left[\left(\Delta_{b,t} \varphi\right)^2+2(\varphi')^2+\varphi^2\right]\,dt=
\int_0^{\infty} t^b (\Delta_{b,t} \varphi-\varphi)^2 dt \, .
\]
Thanks to Lemma \ref{l:V0inft}, it is possible to consider the minimization problem
\begin{equation*}
\min\{J(\varphi):\varphi\in V(0,\infty;t^b),\ \varphi(0)=1\} \, .
\end{equation*}
Since the functional $J$ is clearly coercive with respect to the
norm of $V(0,\infty;t^b)$, the minimization problem admits a weak
solution $\phi$ which solves equation \eqref{eq:phi-equation} and
satisfies the initial conditions $\phi(0)=1$ and $\phi'(0)=0$.
In particular we have
\begin{equation} \label{eq:var-phi}
\int_0^{\infty} t^b [\Delta_{b,t}\phi(t)-\phi(t)] [\Delta_{b,t}
\psi(t)-\psi(t)] \, dt=0
\end{equation}
for any $\psi\in V(0,\infty;t^b)$ such that $\psi(0)=\psi'(0)=0$.

{\bf Step 2.} We prove that $\phi\in C^2([0,\infty))$. If we put $\zeta(t):=\Delta_{b,t}\phi(t)-\phi(t)\in L^2(0,\infty;t^b)$, by \eqref{eq:var-phi}, we see that $\zeta$ is a distributional solution of the equation
\begin{equation} \label{eq:zeta}
\Delta_{b,t} \zeta-\zeta=0 \qquad \text{in } (0,\infty) \, .
\end{equation}
We claim that $\zeta\in C^\infty(0,\infty)$ and it solves \eqref{eq:zeta} in a classical sense.

Indeed, if we put $F(t):=\int_1^t s^b \zeta(s)\, ds$ then $F\in
H^1_{{\rm loc}}(0,\infty)$ being $\zeta\in L^2(0,\infty;t^b)$ and
moreover $F'(t)=t^b\zeta(t)$ in the sense of distributions.

Hence, by \eqref{eq:zeta}, $(t^b\zeta'(t)-F(t))'=0$ in the sense of
distributions so that $t^b \zeta'(t)=F(t)+c$ in $(0,\infty)$. This
implies $\zeta'\in H^1_{{\rm loc}}(0,\infty)$ and in particular
$\zeta\in H^2_{{\rm loc}}(0,\infty)$. Now, with a bootstrap
procedure which makes use of \eqref{eq:zeta}, we conclude that
$\zeta\in C^\infty(0,\infty)$.

Now we claim that $\zeta\in C^0([0,\infty))$. For any $t>s>0$, by \eqref{eq:zeta}, we have
\begin{align} \label{eq:zeta-C0}
 & \left|t^b\zeta'(t)-s^b\zeta'(s)\right|=\left|\int_s^t (\tau^b \zeta(\tau))'d\tau\right|=\left|\int_s^t \tau^b \Delta_{b,\tau}\zeta(\tau)\, d\tau\right| \\
\notag & \quad \le \left(\int_s^t \tau^b |\Delta_{\tau,b} \zeta(\tau)|^2 d\tau\right)^{1/2}\left(\int_s^t \tau^b d\tau\right)^{1/2}
=\tfrac 1{\sqrt{b+1}}\left(\int_s^t \tau^b |\zeta(\tau)|^2 d\tau\right)^{1/2}\left|t^{b+1}-s^{b+1}\right|^{1/2} \\
\notag & \quad \le \tfrac 1{\sqrt{b+1}} \|\zeta\|_{L^2(0,\infty;t^b)} \left|t^{b+1}-s^{b+1}\right|^{1/2} \, .
\end{align}
Since $b>-1$, choosing $t=1$ in \eqref{eq:zeta-C0} and letting $s\to 0^+$, we infer that $s^b \zeta'(s)=O(1)$ as $s\to 0^+$ and, in turn, $\zeta'(s)=O(s^{-b})$ as $s\to 0^+$. This proves that $\zeta'$ is integrable in a right neighborhood of $0$ and hence $\zeta$ is continuous at $0$, thus proving the claim.

Next, we can proceed by completing the proof of Step 2. By
\begin{equation} \label{eq:equazione-phi}
(t^b\phi'(t))'=t^b[\phi(t)+\zeta(t)]
\end{equation}
we deduce that $\phi\in C^\infty(0,\infty)$. Moreover, integrating \eqref{eq:equazione-phi},
for any $0<s<t$, we obtain
\begin{equation} \label{eq:equazione-phi-2}
t^b \phi'(t)-s^b\phi'(s)=\int_s^t \tau^b[\phi(\tau)+\zeta(\tau)]\, d\tau \, .
\end{equation}
By Lemma \ref{l:V0inft} (i), the continuity of $\zeta$ and the fact that $b>-1$, it follows
\begin{equation*}
\lim_{s\to 0^+} s^b \phi'(s)=t^b \phi'(t)-\int_0^t \tau^b[\phi(\tau)+\zeta(\tau)]\, d\tau\in \R \, .
\end{equation*}
This means that there exists $L\in \R$ such that
$\lim_{t\to 0^+} t^b \phi'(t)=L$. We observe that $L=0$ since
otherwise we would have
\begin{equation*}
t^b \frac{(\phi'(t))^2}{t^2}\sim L^2 t^{-b-2} \qquad \text{as } t\to 0^+
\end{equation*}
and hence $t^b \frac{(\phi'(t))^2}{t^2}\not\in L^1(0,R)$ for any $R>0$, in contradiction with Lemma \ref{l:V0inft} (ii).

Therefore, letting $s\to 0^+$ in \eqref{eq:equazione-phi-2}, we infer that
\begin{equation} \label{eq:phi'}
\phi'(t)=t^{-b} \int_0^t \tau^b [\phi(\tau)+\zeta(\tau)]\, d\tau
\end{equation}
and, in turn, by de L'H\^opital rule, we obtain
\begin{equation*}
\lim_{t\to 0^+} \frac{\phi'(t)}t=\lim_{t\to 0^+} \frac{\int_0^t \tau^b [\phi(\tau)+\zeta(\tau)]\, d\tau}{t^{b+1}}=\frac{\phi(0)+\zeta(0)}{b+1} \, .
\end{equation*}
Finally, by \eqref{eq:equazione-phi}, we have that
\begin{equation*}
\lim_{t\to 0^+} \phi''(t)=\lim_{t\to 0^+} \left(-b\frac{\phi'(t)}t+\phi(t)+\zeta(t)\right)=\frac{1}{b+1}\, [\phi(0)+\zeta(0)] \, .
\end{equation*}
This completes the proof of Step 2.

{\bf Step 3.} We show that the function $U$, defined in such a way
that $\widehat U(\xi,t)=\widehat u(\xi)\phi(|\xi|t)$ with $\phi$
as in Step 1, satisfies $U\in C^2(\overline{\R^{N+1}_+})$,
$U_t(\cdot,0)\equiv 0$ in $\R^N$ and it solves \eqref{eq:aux-prob}
in a classical sense.

First, we observe that, by Lemma \ref{l:V0inft} (iii) and
\eqref{eq:zeta-C0}, $\phi,\zeta'$ and, in turn also $\zeta$, have at
most a polynomial growth at $+\infty$. Hence, by \eqref{eq:phi'} also
$\phi'$ has at most a polynomial growth at $+\infty$. Finally, from
the equation $\Delta_{b,t}\phi=\phi+\zeta$, we also deduce that
$\phi''$ has at most a polynomial growth at $+\infty$.

Therefore, since $\phi\in C^2([0,\infty))$ and $\widehat u\in \mathcal S(\R^N)$, with $\mathcal S(\R^N)$ the space of rapidly decreasing $C^\infty(\R^N)$ functions, by the Dominated Convergence Theorem, one can deduce that the map $t\mapsto \widehat u(\xi)\phi(|\xi|t)$ belongs to the space of vector valued functions $C^2([0,\infty);L^2_\C(\R^N;(1+|\xi|^2)^\gamma))$ for any $\gamma\ge 0$. Here $L^2_\C(\R^N;(1+|\xi|^2)^\gamma)$ denotes the weighted complex $L^2$-space. This proves that the map $t\mapsto U(x,t)$ belongs to the space
$C^2([0,\infty);H^\gamma(\R^N))$ for any $\gamma\ge 0$.
From this we deduce that $U\in C^2(\overline{\R^{N+1}_+})$. Since
$$
U_t(x,t)=\frac 1{(2\pi)^{N/2}} \int_{\R^N} e^{i\xi\cdot x} \, \widehat u(\xi) \, |\xi| \, \phi'(|\xi|t)\, d\xi
$$
and $\phi'(0)=0$ it follows that $U_t(x,0)=0$ for any $x\in \R^N$. By construction, we also have that $U$ is a classical solution of \eqref{eq:aux-prob}.

{\bf Step 4.} We prove that $\Delta_b U\in L^2(\R^{N+1}_+;t^b)$ and
\begin{equation} \label{eq:isometry}
\int_{\R^{N+1}_+} t^b |\Delta_b U(x,t)|^2 dx\,dt=J(\phi) \int_{\R^N} |\xi|^{2s} |\widehat u(\xi)|^2 d\xi  \, .
\end{equation}
By direct computation we see that
\begin{equation*}
|\Delta_{b,t} \widehat U(\xi,t)-|\xi|^2 \widehat U(\xi,t)|^2=|\xi|^4 |\widehat u(\xi)|^2 [\Delta_{b,t} \phi(|\xi|t)-\phi(|\xi|t)]^2 \, .
\end{equation*}
After integration, a change of variable with respect to $t$ and Fubini-Tonelli Theorem, we obtain
\begin{align} \label{eq:J-phi}
 \int_{\R^{N+1}_+} t^b |\Delta_{b,t} \widehat U(\xi,t)-|\xi|^2 \widehat U(\xi,t)|^2 d\xi dt
&=\int_{\R^{N+1}_+} |\xi|^{3-b} |\widehat u(\xi)|^2 t^b [\Delta_{b,t} \phi(t)-\phi(t)]^2 d\xi dt \\
\notag & =J(\phi) \int_{\R^N} |\xi|^{2s} |\widehat u(\xi)|^2 d\xi \, .
\end{align}
Since $\widehat u\in \mathcal S(\R^N)$, the last integral is finite
and hence, by Fubini-Tonelli Theorem, for almost every $t\in
(0,\infty)$ the map $\xi\mapsto  \Delta_{b,t} \widehat
U(\xi,t)-|\xi|^2 \widehat U(\xi,t)=\widehat{\Delta_b U}(\xi,t)$
belongs to the complex space $L^2_{\C}(\R^N)$. Hence by Plancherel Theorem also the
map $x\mapsto \Delta_b U(x,t)$ belongs to $L^2(\R^N)$ for almost
every $t\in (0,\infty)$. Moreover
\begin{equation*}
\int_{\R^N} |\Delta_b U(x,t)|^2 dx=\int_{\R^N} |\Delta_{b,t} \widehat U(\xi,t)-|\xi|^2 \widehat U(\xi,t)|^2 d\xi \qquad \text{for almost every } t\in (0,\infty) \, .
\end{equation*}
Multiplying this identity by $t^b$, integrating in $(0,\infty)$ with
respect to the variable $t$ and applying Fubini-Tonelli Theorem we
deduce that $\Delta_b U\in L^2(\R^{N+1}_+;t^b)$. Moreover
\eqref{eq:isometry} follows by exploiting \eqref{eq:J-phi}.

{\bf Step 5.} We prove that $U\in \mathcal D_b$.

We have to prove that $U$ can be approximated with functions in
$\mathcal T$ with respect to the norm $\|\cdot\|_{\mathcal D_b}$.
Here $\mathcal T$ is the space defined in \eqref{eq:space}.

Combining Plancherel Theorem with the fact that $\widehat u\in \mathcal S(\R^N)$ and $\phi\in V(0,\infty;t^b)$ one can verify that
$U\in L^2(\R^{N+1}_+;t^b)$ and $\nabla U\in L^2(\R^{N+1}_+;t^b)$. Therefore since $U\in C^2(\overline{\R^{N+1}_+})$ we also have that
\begin{equation} \label{eq:summability}
\frac{U}{|x|^2+t^2}\in L^2(\R^{N+1}_+;t^b) \qquad \text{and}
\qquad \frac{|\nabla U|}{\sqrt{|x|^2+t^2}}\in L^2(\R^{N+1}_+;t^b)
\, .
\end{equation}
Define $U_n(x,t)=\eta\left(\frac{|x|}n\right)\eta\left(\frac t n\right)U(x,t)$ where $\eta\in C^\infty([0,\infty))$, $\eta\equiv 1$ in $[0,1]$ and
$\eta\equiv 0$ in $[2,\infty)$.
We prove that
\begin{equation} \label{eq:pass-lim}
\int_{\R^{N+1}_+} t^b |\Delta_b(U_n-U)|^2 dx\,dt\to 0 \qquad \text{as } n\to +\infty \, .
\end{equation}
By direct computation one sees that
\begin{align} \label{eq:truncation}
\Delta_b U_n(x,t)&=\eta\left(\tfrac t n\right)\Theta\left(\tfrac xn\right)\Delta_b U(x,t)
+\eta\left(\tfrac tn\right)\left[\tfrac 1{n^2}\Delta_x\Theta\left(\tfrac xn\right)U(x,t)+\tfrac 2n\nabla_x\Theta\left(\tfrac xn\right)\nabla_x U(x,t)\right]\\
\notag & +\Theta\left(\tfrac xn\right)\left[\tfrac 1{n^2}\eta''\left(\tfrac tn\right)U(x,t)+\tfrac 2n \eta'\left(\tfrac tn\right)U_t(x,t)
+\tfrac bt\, \tfrac 1n \eta'\left(\tfrac tn\right)U(x,t)\right]
\end{align}
where we put $\Theta(x)=\eta(|x|)$. Then, we observe that there exists a positive constant $C$ independent of $x$, $t$ and $n$, such that
\begin{align} \label{eq:truncation-2}
& t^b \left|\eta\left(\tfrac t n\right)\Theta\left(\tfrac xn\right)\Delta_b U(x,t)\right|^2\le t^b \left|\Delta_b U(x,t)\right|^2 \, , \quad
 \tfrac{t^b}{n^4}\left|\eta\left(\tfrac tn\right)\Delta_x\Theta\left(\tfrac xn\right)U(x,t)\right|^2\le Ct^b \tfrac{U^2(z)}{|z|^4}\, , \\
\notag & \tfrac{4t^b}{n^2}\left|\eta\left(\tfrac t n\right)\nabla_x\Theta\left(\tfrac xn\right)\nabla_x U(x,t)\right|^2\le Ct^b \tfrac{|\nabla U(z)|^2}{|z|^2} \, , \quad
 \tfrac{t^b}{n^4}\left|\Theta\left(\tfrac xn\right)\eta''\left(\tfrac tn\right)U(x,t)\right|^2\le Ct^b \tfrac{U^2(z)}{|z|^4} \, , \\
\notag & \tfrac{4t^b}{n^2} \left|\Theta\left(\tfrac xn\right)\eta'\left(\tfrac tn\right)U_t(x,t)\right|^2 \le C t^b \tfrac{|\nabla U(z)|^2}{|z|^2} \, , \quad
 \tfrac{b^2 t^b}{n^4}\left|\Theta\left(\tfrac xn\right)\tfrac{\eta'(t/n)}{t/n}U(x,t)\right|^2\le Ct^b \tfrac{U^2(z)}{|z|^4} \, ,
\end{align}
since $|z|\le \sqrt 8 n$ for any $z\in {\rm supp}\left(\eta\left(\tfrac tn\right)\Theta\left(\tfrac xn\right)\right)$ where we put $z=(x,t)\in \R^{N+1}$.

By \eqref{eq:summability}, \eqref{eq:truncation}, \eqref{eq:truncation-2} and the Dominated Convergence Theorem, \eqref{eq:pass-lim} follows.

This shows that for any $\e>0$ there exists a function
$V\in C^2_c(\overline{\R^{N+1}_+})$ such that
$$
\int_{\R^{N+1}_+} t^b |\Delta_b(U-V)|^2 dx\,dt<\e \, .
$$
By Step 3 and the truncation argument introduced above, we deduce that we can choose $V$ in such a way that $V_t(\cdot,0)\equiv 0$ in $\R^N$.

A mollification argument allows us to approximate, with respect to
the norm $\|\cdot\|_{\mathcal D_b}$, the function $V$ found above,
with a $C^\infty$ compactly supported function $W$ satisfying
$W_t(\cdot,0)\equiv 0$ in $\R^N$. Indeed, one can introduce a
sequence of mollifiers $\{\rho_n\}$ and still denote by $V$ the
even extension with respect to the variable $t$ to the whole
$\R^{N+1}$. This extension satisfies
  $V\in C^2_c(\R^{N+1})$ since $V_t(x,0)=0$ for any $x\in \R^N$. We choose the functions $\rho_n$ even with respect to the $t$ variable. Then one can verify that
the functions $W_n:=\rho_n*V\in C^\infty_c(\R^{N+1})$ are even
with respect to $t$ and the functions $\partial_t W_n$ are odd
with respect to $t$; in particular $\partial_t W_n(\cdot,0)\equiv
0$ in $\R^N$. Exploiting the fact that for any $n\in \N$,
$\partial_t W_n$ is odd with respect to $t$, one can show that
$|\partial_t W_n(x,t)|\le C|t|$ for any $(x,t)\in \R^{N+1}$ and
$n\in \N$ where $C$ is a constant independent of $(x,t)\in
\R^{N+1}$ and $n\in \N$.

Combining this estimate with the fact that $V\in C^2_c(\R^{N+1})$,
by Dominated Convergence Theorem we obtain $\int_{\R^{N+1}} |t|^b
\, |\Delta_b(W_n-V)|^2 dx\,dt\to 0$ as $n\to +\infty$. We have just
shown that $U\in \mathcal D_b$.

{\bf Step 6.} In this step we complete the proof of the lemma. The proof of (i) follows from \eqref{eq:isometry} once we put $C_b:=J(\phi)$.

It remains to prove (ii). Let $\Phi\in C^\infty_c(\overline{\R^{N+1}_+})$
such that $\Phi(x,0)=\Phi_t(x,0)=0$ for any
$x\in\R^N$.
Recalling that $\widehat{\Delta_b U}(\xi,t)=|\xi|^2 \widehat u(\xi)[\Delta_{b,t}\phi(|\xi|t)-\phi(|\xi|t)]$, by Plancherel Theorem, Fubini-Tonelli Theorem and a change of variable, we have
\begin{align} \label{eq:U-VAR}
& \int_{\R^{N+1}_+} t^b \Delta_b U(x,t) \Delta_b \Phi(x,t) \, dx\,dt\\
\notag & \qquad =\int_{\R^N} \left(\int_0^{\infty} t^b |\xi|^2 \, \widehat
u(\xi) \, [\Delta_{b,t}
\phi(|\xi|t)-\phi(|\xi|t)]\left[\overline{\Delta_{b,t}\widehat
\Phi(\xi,t)-|\xi|^2\widehat \Phi(\xi,t)}\right] \, dt\right)d\xi\\
\notag & \qquad =\int_{\R^N} \left(\int_0^{\infty} t^b |\xi|^{1-b} \, \widehat
u(\xi) \, [\Delta_{b,t}
\phi(t)-\phi(t)]\left[\overline{\Delta_{b,t}\widehat \Phi\left(\xi,\tfrac
t{|\xi|}\right)-|\xi|^2\widehat \Phi\left(\xi,\tfrac t{|\xi|}\right)}\right] \,
dt\right)d\xi \\
\notag & \qquad =\int_{\R^N} |\xi|^{3-b} \, \widehat u(\xi) \, \left(\int_0^{\infty} t^b
[\Delta_{b,t} \phi(t)-\phi(t)]\left[\overline{\Delta_{b,t}
\left(\widehat\Phi\left(\xi,\tfrac t{|\xi|}\right)\right)-\widehat \Phi\left(\xi,\tfrac
t{|\xi|}\right)}\right] \, dt\right)d\xi=0
\end{align}
where the last identity follows from the fact that, for any
$\xi\neq 0$, the real part and the imaginary part of the map
$t\mapsto \widehat\Phi\left(\xi,\tfrac t{|\xi|}\right)$ are
admissible test functions in \eqref{eq:var-phi} since they belong
to $C^\infty_c([0,\infty))$ and they vanish at $t=0$ together with
their first derivatives. By a density argument combined with the
regularization procedure shown in Step 5, one can show that
\eqref{eq:U-VAR} actually holds for any $\Phi\in
C^2(\overline{\R^{N+1}_+})$ such that
\begin{align} \label{eq:condizioni}
& \Delta_b \Phi\in L^2(\R^{N+1}_+;t^b)\, , \qquad \frac{|\nabla \Phi|}{\sqrt{|x|^2+t^2}} \in L^2(\R^{N+1}_+;t^b)\, , \\
\notag & \frac{\Phi}{|x|^2+t^2} \in L^2(\R^{N+1}_+;t^{b}) \, ,
\qquad \Phi(\cdot,0)\equiv \Phi_t(\cdot,0)\equiv 0 \ \text{in }
\R^N \, .
\end{align}
Let $V$ be as in the statement of the lemma and put $\Phi:=V-U$ is such a way that $\Phi\in C^2(\overline{\R^{N+1}_+})$ and it satisfies \eqref{eq:condizioni}. By \eqref{eq:U-VAR} we then have
\begin{align*}
\|V\|_{\mathcal D_b}^2
=\|\Phi\|_{\mathcal D_b}^2+2\int_{\R^{N+1}_+} t^b \Delta_b U\Delta_b \Phi \, dx\,dt+\|U\|_{\mathcal D_b}^2
=\|\Phi\|_{\mathcal D_b}^2+\|U\|_{\mathcal D_b}^2\ge \|U\|_{\mathcal D_b}^2 \, .
\end{align*}
This completes the proof of the lemma.
\end{proof}

Thanks to Lemma \ref{l:lemmone}, in the next
proposition we construct a trace map ${\rm Tr}:\mathcal D_b\to \mathcal D^{s,2}(\R^N)$.

\begin{Proposition} \label{p:3.3} Let $s\in (1,2)$ and let $b=3-2s\in (-1,1)$. Then there exists a linear continuous map ${\rm Tr}:\mathcal D_b\to \mathcal D^{s,2}(\R^N)$
such that ${\rm Tr}(V)=V_{|\R^N\times \{0\}}$ for any $V\in
C^\infty_c(\overline{\R^{N+1}_+})$.
\end{Proposition}

\begin{proof} Let $V\in C^\infty_c(\overline{\R^{N+1}_+})$ be such that $V_t(\cdot,0)\equiv 0$ in $\R^N$ and put $u=V_{|\R^N\times \{0\}}\in C^\infty_c(\R^N)$. By Lemma \ref{l:lemmone}, we deduce that there exists $U\in C^2(\overline{\R^{N+1}_+})\cap \mathcal D_b$ such that
\begin{equation} \label{eq:id-imp}
U_{|\R^N\times\{0\}}=u \, , \quad \|U\|_{\mathcal D_b}=C_b \|u\|_{\mathcal D^{s,2}(\R^N)} \, , \quad  \|U\|_{\mathcal D_b}\le \|V\|_{\mathcal D_b} \, .
\end{equation}
Therefore, if we put ${\rm Tr}(V):=u$ we have $\|{\rm Tr}(V)\|_{\mathcal D^{s,2}(\R^N)}\le C_b^{-1} \|V\|_{\mathcal D_b}$.
The conclusion follows by completion.
\end{proof}

Let $u$ be a solution of \eqref{eq:problema} and let $U\in
\mathcal D_b$ be the corresponding solution to
\eqref{eq:aux-prob}. From Lemma \ref{l:lemmone} it follows that
$C_b^2\|u\|^2_{\mathcal D^{s,2}(\R^N)}=\|U\|^2_{\mathcal D_b}$.
Moreover by the proof of Proposition \ref{p:3.3}, for all
$\varphi\in \mathcal D_b$ satisfying ${\rm Tr\,}(\varphi)=u$, we
have that
\begin{equation} \label{eq:TRACE}
C_b^{2}\|{\rm Tr\,}(\varphi)\|^2_{\mathcal D^{s,2}(\R^N)}
=\|U\|^2_{\mathcal
  D_b}\leq \|\varphi\|^2_{\mathcal D_b},
\end{equation}
which is equivalent to say that $U\in \mathcal D_b$ is
a solution to the minimum problem
\begin{equation*}
\min_{\varphi\in\mathcal D_b, {\rm Tr\,}(\varphi)=u}\left\{\|\varphi\|^2_{\mathcal D_b}-C_b^{2}\|{\rm Tr\,}
(\varphi)\|^2_{\mathcal D^{s,2}(\R^N)}\right\} \, .
\end{equation*}
Therefore we have
\begin{equation} \label{eq:test-psi}
(U,\psi)_{\mathcal D_b}=0 \qquad \text{for any } \psi\in \mathcal
D_b \ \text{such that } {\rm Tr\,}(\psi)=0 \, .
\end{equation}
Now, for any $\varphi\in\mathcal D_b$ we denote by $\Phi\in \mathcal D_b$ the solution of \eqref{eq:aux-prob} corresponding to ${\rm Tr\,}(\varphi)$.
By \eqref{eq:id-imp} we have that
\begin{equation*}
\|U+\Phi\|_{\mathcal D_b}^2=C_b^2\|u+{\rm Tr\,}(\varphi)\|_{\mathcal D^{s,2}(\R^N)}^2 \quad
\text{and} \quad \|U-\Phi\|_{\mathcal D_b}^2=C_b^2\|u-{\rm Tr\,}(\varphi)\|_{\mathcal D^{s,2}(\R^N)}^2
\end{equation*}
and taking the difference we obtain
\begin{equation} \label{eq:test-Phi}
(U,\Phi)_{\mathcal D_b}=C_b^2(u,{\rm Tr\,}(\varphi))_{\mathcal
D^{s,2}(\R^N)}  \, .
\end{equation}
Since ${\rm Tr\,}(\varphi-\Phi)=0$, combining \eqref{eq:test-psi} and \eqref{eq:test-Phi} we obtain
\begin{equation} \label{eq:nuova}
(U,\varphi)_{\mathcal D_b}=(U,\Phi)_{\mathcal D_b}=C_b^2(u,{\rm Tr\,}(\varphi))_{\mathcal
  D^{s,2}(\R^N)}  \qquad \text{for any } \varphi\in\mathcal D_b.
\end{equation}
Hence $u\in \mathcal D^{s,2}(\R^N)$ solves
\eqref{eq:problema-variazionale} if and only if the corresponding
function $U\in \mathcal D_b$ solving
\eqref{eq:aux-prob}-\eqref{eq:min-prob} is a solution to
\begin{equation}\label{eq:6}
(U,\varphi)_{\mathcal D_b}=0 \qquad \text{for all }
\varphi\in \mathcal D_b\text{ s.t. }{\rm supp\,}({\rm
  Tr\,}(\varphi))\subset\Omega \, .
\end{equation}

\section{An Almgren type monotonicity formula} \label{s:Almgren}
Let us assume that $U\in \mathcal D_b$ is a solution to
\eqref{eq:6}. Let us set
\begin{equation}\label{eq:defV}
  V:=\Delta_b U\in L^2(\R^{N+1}_+;t^b),
\end{equation}
i.e., in view of \eqref{eq:9} and Proposition \ref{t:rellich},
\begin{equation}\label{eq:11}
\int_{\R^{N+1}_+} t^b V\varphi \, dz= -\int_{\R^{N+1}_+} t^b\nabla
U \nabla\varphi\, dz \, ,\qquad\text{for any }\varphi\in
C^\infty_c(\overline{\R^{N+1}_+}).
\end{equation}
Furthermore \eqref{eq:6} yields
\begin{equation}\label{eq:10}
\int_{\R^{N+1}_+} V \dive(t^b \nabla \varphi) \, dz=0
\end{equation}
for any $\varphi\in C^\infty_c\big(\overline{\R^{N+1}_+}\big)$ such
that
${\rm supp}(\varphi(\cdot,0))\subset \Omega$ and ${\ds \lim_{t\to
0^+} \varphi_t(\cdot,t)\equiv 0}$ in $\R^N$. Proposition \ref{p:A2}
then ensures that
\begin{equation} \label{eq:H^1}
V\in H^1(Q_R^+(x_0);t^b) \text{ for any $x_0\in \Omega$ and $R>0$
satisfying $B_{2R}'(x_0)\subset \Omega$}.
\end{equation}
Up to translation it is not restrictive to suppose that $x_0=0\in\Omega$.
Then we fix a radius $R>0$ satisfying \eqref{eq:H^1}. For
simplicity, the center $x_0$ of the sets introduced in
\eqref{eq:sets} will be omitted whenever $x_0=0$.

By \eqref{eq:10}-\eqref{eq:H^1} we obtain
\begin{equation}\label{eq:10-bis}
\int_{B_R^+} t^b \nabla V \nabla \varphi \, dz=0
\end{equation}
for any $\varphi\in C^\infty_c(\Sigma_R^+(0))$ such that
$\varphi_t(\cdot,0)\equiv 0$ in $B_R'$.

Actually \eqref{eq:10-bis} still holds true for any $\varphi\in
C^\infty_c(\Sigma_R^+(0))$ not necessarily satisfying
$\varphi_t(\cdot,0)\equiv 0$ in $B_R'$. Indeed, for any
$\varphi\in C^\infty_c(\Sigma_R^+(0))$, one can test
\eqref{eq:10-bis} with
$\varphi_k(x,t)=\varphi(x,t)-\varphi_t(x,0)\, t\, \eta(kt)$, $k\in
\N$, where $\eta\in C^\infty_c(\R)$, $0\le \eta\le 1$, $\eta(t)=1$
for any $t\in [-1,1]$ and $\eta(t)=0$ for any $t\in
(-\infty,-2]\cup [2,+\infty)$, and pass to the limit as $k\to
+\infty$.

By density we may conclude that
\begin{equation*} 
\int_{B_R^+} t^b \nabla V \nabla \varphi \, dz=0 \qquad \text{for
any } \varphi\in H^1_0(\Sigma_R^+;t^b) \, .
\end{equation*}
Hence, the couple $(U,V)\in \mathcal D_b\times
L^2(\R^{N+1}_+;t^b)$ is a weak solution to the system \eqref{eq:system}
in the sense that \eqref{eq:11} and \eqref{eq:10} hold together
with the forced boundary condition \eqref{eq:9}. Thanks to
Proposition \ref{t:rellich} and \eqref{eq:H^1}, we may define the
functions
\begin{equation} \label{eq:D(r)}
D(r)=r^{-N-b+1} \left[\int_{B_r^+} t^b\left(|\nabla U|^2+|\nabla
V|^2+UV\right)dz\right]
\end{equation}
and
\begin{equation} \label{eq:H(r)}
H(r)=r^{-N-b}\int_{S_r^+} t^b(U^2+V^2)\, dS \, .
\end{equation}
We observe that the function $H=H(r)$ is well defined for every
$r>0$ such that $B_{2r}'\subset \Omega$ since the trace operator
$$
{\rm Tr}_{S_r}:H^1(B_r^+;t^b)\to L^2(S_r^+;t^b)
$$
is well-defined and continuous being $b\in (-1,1)$, see
\cite[Subsection 2.2]{FF}.

We now prove a Pohozaev-type identity for system
\eqref{eq:system}.

\begin{Lemma} Let $U$ and $V$ be as in \eqref{eq:6} and
\eqref{eq:defV}. Then for a.e. $r>0$ such that $B_{2r}'\subset
\Omega$ we have
\begin{align} \label{Pohozaev-1}
\int_{B_r^+} t^b\left(|\nabla U|^2+|\nabla
V|^2+UV\right)dz&=\int_{S_r^+}t^b \left(\frac{\partial U}{\partial
\nu} U+\frac{\partial V}{\partial \nu}V\right)dS
\end{align}
and
\begin{align} \label{Pohozaev-2}
-\frac{N+b-1}2 & \int_{B_r^+} t^b \left(|\nabla U|^2+|\nabla
V|^2\right)dz+\int_{B_r^+} t^b V(z\cdot \nabla U)\, dz \\
\notag &+\frac r2
\int_{S_r^+} t^b \left(|\nabla U|^2+|\nabla V|^2\right)dS =r\int_{S_r^+} t^b \left(\left|\frac{\partial
U}{\partial\nu}\right|^2+\left|\frac{\partial
V}{\partial\nu}\right|^2\right)dS \, .
\end{align}
\end{Lemma}

\begin{proof} The proof of this lemma can be obtained proceeding exactly as in the proof of Theorem 3.7 in \cite{FF}. Hence here we omit the details and
we show only the main steps. Let us consider first
identity \eqref{Pohozaev-2}. Let $r$ be as in the statement of the
lemma. Similarly to \cite{FF}, for any $\delta>0$
we define the set
$$
O_\delta:=B_r^+\cap
\{(x,t):t>\delta\} \, .
$$
By \eqref{eq:system} and exploiting \cite[(51)]{FF} by replacing
their $1-2s$ with our $b=3-2s$, we obtain
\begin{align} \label{eq:PI-V}
\frac{N+b-1}2\int_{O_\delta} t^b |\nabla V|^2 dz= &-\frac 12 \, \delta^{b+1} \int_{B'_{\sqrt{r^2-\delta^2}}}
|\nabla V(x,\delta)|^2 dx+\delta^{b+1} \int_{B'_{\sqrt{r^2-\delta^2}}} |V_t(x,\delta)|^2 dx \\
\notag & +\delta^b \int_{B'_{\sqrt{r^2-\delta^2}}} (x\cdot \nabla_x V(x,\delta)) V_t(x,\delta) \, dx \\
\notag & +\frac r2 \int_{S_r^+\cap \{t>\delta\}} t^b |\nabla V|^2 dS-r \int_{S_r^+\cap \{t>\delta\}} t^b \left|\frac{\partial V}{\partial \nu}\right|^2 dS
\end{align}
and
\begin{align} \label{eq:PI-U}
\frac{N+b-1}2\int_{O_\delta} t^b |\nabla U|^2 dz& -\int_{O_\delta} t^b
                                                  V(z\cdot \nabla U)
                                                  \, dz=-\frac 12 \,
 \delta^{b+1} \int_{B'_{\sqrt{r^2-\delta^2}}}
|\nabla U(x,\delta)|^2 dx \\
\notag & +\delta^{b+1} \int_{B'_{\sqrt{r^2-\delta^2}}} |U_t(x,\delta)|^2 dx \\
\notag & +\delta^b \int_{B'_{\sqrt{r^2-\delta^2}}} (x\cdot \nabla_x U(x,\delta)) U_t(x,\delta) \, dx \\
\notag & +\frac r2 \int_{S_r^+\cap \{t>\delta\}} t^b |\nabla U|^2 dS-r \int_{S_r^+\cap \{t>\delta\}} t^b \left|\frac{\partial U}{\partial \nu}\right|^2 dS.
\end{align}
Now, arguing as in \cite{FF}, on can show that there exists a
sequence $\delta_n \downarrow 0$ such that
\begin{align} \label{eq:PI-INT}
&  \delta_n^{b+1} \int_{B'_{\sqrt{r^2-\delta_n^2}}} |\nabla V(x,\delta_n)|^2 dx\to 0 \,
, \qquad
\delta_n^{b+1} \int_{B'_{\sqrt{r^2-\delta_n^2}}} |V_t(x,\delta_n)|^2 dx\to 0\, , \\
\notag &  \delta_n^{b+1} \int_{B'_{\sqrt{r^2-\delta_n^2}}}
|\nabla U(x,\delta_n)|^2 dx\to 0 \, ,\qquad
\delta_n^{b+1} \int_{B'_{\sqrt{r^2-\delta_n^2}}} |U_t(x,\delta_n)|^2 dx\to 0 \, ,
\end{align}
as $n\to +\infty$.

By the local regularity estimates of Propositions \ref{p:6.5} and
\ref{p:6.4} we infer that
$U,V\in C^{0,\alpha}(\overline B_r^+)$,
$\nabla_x U,\nabla_x V\in C^{0,\alpha}(\overline B_r^+)$ and $t^b U_t, t^b V_t \in C^{0,\alpha}(\overline
B_r^+)$ for some $\alpha\in (0,1)$.

These regularity estimates on $U,V$ and their derivatives combined
with the Dominated Convergence Theorem imply that
\begin{align} \label{eq:PI-DC}
& \lim_{\delta\to 0^+} \delta^b
\int_{B'_{\sqrt{r^2-\delta^2}}} (x\cdot \nabla_x V(x,\delta))
V_t(x,\delta) \, dx=0 \, , \\
\notag & \lim_{\delta\to 0^+} \delta^b \int_{B'_{\sqrt{r^2-\delta^2}}} (x\cdot \nabla_x U(x,\delta)) U_t(x,\delta) \, dx
=0 \, .
\end{align}
Next, by \eqref{eq:PI-INT} and \eqref{eq:PI-DC}, one can pass to the
limit in \eqref{eq:PI-V}
 and \eqref{eq:PI-U} with $\delta=\delta_n$ as $n\to +\infty$, thus
 obtaining \eqref{Pohozaev-2}.

In order to prove \eqref{Pohozaev-1} it is sufficient to test the
equations in \eqref{eq:system} with $U$ and $V$ respectively.
\end{proof}

\begin{Lemma} \label{l:positivity} Let $U$ and $V$ be as in \eqref{eq:6} and
\eqref{eq:defV} and let $D=D(r)$ and $H=H(r)$ be the functions
defined in \eqref{eq:D(r)} and \eqref{eq:H(r)}. Suppose that
$(U,V)\not\equiv (0,0)$. Then there exists $r_0>0$ such that
$H(r)>0$ for any $r\in (0,r_0)$.
\end{Lemma}

\begin{proof} Suppose by contradiction that for any $r_0>0$ there exists $r\in
(0,r_0)$ such that $H(r)=0$. This means that $U$ and $V$ vanish on
$S_r^+$. In particular, by \eqref{eq:PA2-2} and
\eqref{Pohozaev-1}, we have
\begin{align} \label{eq:LE}
& 0=\int_{B_r^+} t^b(|\nabla U|^2+|\nabla V|^2+UV)\, dz \ge
\left(1-\frac{2r^2}{(N+b-1)^2} \right) \int_{B_r^+} t^b(|\nabla
U|^2+|\nabla V|^2)\, dz \, .
\end{align}
If $r_0$ is sufficiently small and $r\in (0,r_0)$, the parenthesis
appearing in the right hand side of \eqref{eq:LE} becomes
positive. This, in turn, implies $\int_{B_r^+} t^b(|\nabla
U|^2+|\nabla V|^2)\, dz=0$ which, combined with \eqref{eq:PA2-2},
implies $U\equiv 0$ and $V\equiv 0$ in $B_r^+$. Since $U$ and $V$
are weak solutions of the equations $\Delta_b U=V$ and $\Delta_b
V=0$ in $\R^{N+1}_+$, by classical unique continuation principles
for elliptic operators with smooth coefficients, we deduce that
$U$ and $V$ vanish in $\R^{N+1}$ thus contradicting the assumption
$(U,V)\not\equiv 0$.~\end{proof}

The statement of Lemma \ref{l:positivity} allows us to define the
Almgren type function $\mathcal N:(0,r_0)\to \R$ as
\begin{equation} \label{eq:def-N}
\mathcal N(r)=\frac{D(r)}{H(r)} \qquad \text{for any } r\in
(0,r_0) \, .
\end{equation}

\begin{Lemma} \label{l:D(r)ge}
Let $U$ and $V$ be as in \eqref{eq:6} and
\eqref{eq:defV} and let $R$ be as in \eqref{eq:H^1}. Let $D$, $H$,
$\mathcal N$ be the functions defined in \eqref{eq:D(r)},
\eqref{eq:H(r)} and \eqref{eq:def-N} respectively. Then there
exists $\tilde r\in (0,r_0)$ such that
\begin{align} \label{eq:prima}
D(r)\ge \frac{r^{-N-b+1}}2 \int_{B_r^+} t^b(|\nabla U|^2+|\nabla V|^2)\, dz-\frac{r^2}{N+b-1} H(r)
\end{align}
for any $r\in (0,\tilde r)$. In particular we have that
\begin{equation} \label{eq:seconda}
\mathcal N(r)\ge -\frac{r^2}{N+b-1}
\, .
\end{equation}
Moreover, there exist two positive constants $C_1,C_2$ independent of $r$ such that $D(r)+C_2 H(r)\ge 0$ for any $r\in (0,\tilde r)$ and
\begin{equation} \label{eq:terza}
\int_{B_r^+} t^b (U^2+V^2) \, dz\le C_1 r^{N+b+1}[D(r)+C_2 H(r)]
\qquad \text{for any } r\in (0,\tilde r) \, .
\end{equation}
\end{Lemma}

\begin{proof} By Young inequality and \eqref{eq:PA2-2}, we have
\begin{align} \label{eq:preliminare}
& \left|\int_{B_r^+} t^b UV\, dz\right| \le \frac 12 \int_{B_r^+} t^b (U^2+V^2) \, dz\\
\notag & \qquad \le \frac{2r^2}{(N+b-1)^2}  \left[\int_{B_r^+} t^b(|\nabla U|^2+|\nabla V|^2)\, dz+\frac{N+b-1}{2r} \int_{S_r^+} t^b(U^2+V^2)\, dS\right]
\end{align}
from which we obtain
\begin{align*}
& \int_{B_r^+} t^b(|\nabla U|^2+|\nabla V|^2+UV)\, dz \\[7pt]
& \qquad \ge \left(1-\frac{2r^2}{(N+b-1)^2}\right)\int_{B_r^+} t^b(|\nabla U|^2+|\nabla V|^2)\, dz-\frac{r}{N+b-1} \int_{S_r^+} t^b(U^2+V^2)\, dS
\end{align*}
for any $r\in (0,r_0)$.  The proof of \eqref{eq:prima} and
\eqref{eq:seconda} then follows from the definitions of $D$, $H$ and
$\mathcal N$, choosing $\tilde r\in (0,r_0)$ sufficiently
small. Combining \eqref{eq:preliminare} and \eqref{eq:prima} we also
obtain \eqref{eq:terza}.
\end{proof}

In order to prove the validity of an Almgren type monotonicity
formula we need to compute the derivative of $\mathcal N$. In
order to do that we first compute the derivatives of the functions
$D$ and $H$.

\begin{Lemma} \label{l:H'(r)}
Let $U$ and $V$ be as in \eqref{eq:6} and \eqref{eq:defV} and let
$R$ be as in \eqref{eq:H^1}. Let $H=H(r)$ be the function defined
in \eqref{eq:H(r)}. Then $H\in W^{1,1}_{{\rm loc}}(0,R)$ and
moreover we have
\begin{equation} \label{H'}
H'(r)=2r^{-N-b} \int_{S_r^+} t^b \left(U\frac{\partial U}{\partial
\nu}+V\frac{\partial V}{\partial \nu}\right)dS \quad \text{in a
distributional sense and a.e. } r\in (0,R) \, ,
\end{equation}
and
\begin{equation} \label{H'2}
H'(r)=\frac 2r \, D(r) \quad \text{in a distributional sense and
a.e. } r\in (0,R) \, .
\end{equation}
\end{Lemma}

\begin{proof} See the proof of \cite[Lemma 3.8]{FF}.
\end{proof}

\begin{Lemma} \label{l:D'(r)}
Let $U$ and $V$ be as in \eqref{eq:6} and \eqref{eq:defV} and let
$R$ be as in \eqref{eq:H^1}. Let $D=D(r)$ be the function defined
in \eqref{eq:D(r)}. Then $D\in W^{1,1}_{{\rm loc}}(0,R)$ and
moreover we have
\begin{align} \label{eq:D'}
D'(r)=&\frac{2}{r^{N+b-1}} \int_{S_r^+} t^b
          \left(\left|\frac{\partial U}{\partial \nu}\right|^2+
          \left|\frac{\partial V}{\partial \nu}\right|^2\right) dS
+\frac{1}{r^{N+b-1}} \int_{S_r^+} t^b UV\,dS\\
\notag &-\frac{2}{r^{N+b}} \int_{B_r^+}t^b V(z\cdot \nabla U) \,dz-\frac{N+b-1}{r^{N+b}}\int_{B_r^+} t^b UV\,dz
\end{align}
in a distributional sense and  a.e. $r\in (0,R)$.
\end{Lemma}

\begin{proof} The proof can be easily obtained by replacing \eqref{Pohozaev-2} into
$$
D'(r)=r^{-N-b}[(1-N-b)I(r)+rI'(r)]\, ,
$$
where $I(r)= \int_{B_r^+} t^b \left(|\nabla U|^2+|\nabla V|^2+UV\right)dz$.
\end{proof}

\begin{Lemma} \label{l:N'(r)}
Let $U$ and $V$ be as in \eqref{eq:6} and \eqref{eq:defV} and let
$R$ be as in \eqref{eq:H^1}. Let $\mathcal N=\mathcal N(r)$ and
$r_0$ be as in \eqref{eq:def-N}. Then $\mathcal N \in
W^{1,1}_{{\rm loc}}(0,r_0)$ and moreover we have
\begin{align} \label{eq:N'}
\mathcal N'(r)=\nu_1(r)+\nu_2(r)
\end{align}
in a distributional sense and for a.e. $r\in (0,r_0)$, where
\begin{align*}
  \nu_1(r)=\frac{2r\Big[
  \left(\int_{S_r^+} t^b
  \left(\left|\frac{\partial U}{\partial \nu}\right|^2+
  \left|\frac{\partial V}{\partial \nu}\right|^2\right) dS\right)
  \left(\int_{S_r^+} t^b
  (U^2+V^2)\,dS\right)-\left(
  \int_{S_r^+} t^b
  \left(U\frac{\partial U}{\partial \nu}
  +V\frac{\partial V}{\partial \nu}\right)\, dS\right)^{\!2} \Big]}
  {\left( \int_{S_r^+} t^b
  (U^2+V^2)\,dS\right)^2}
\end{align*}
and
\begin{align}\label{eq:16}
  \nu_2(r)= &\, \frac{r\int_{S_r^+}  t^b UV\,dS-2\int_{B_r^+}t^b V (z\cdot \nabla
              U) \,dz-(N+b-1)\int_{B_r^+} t^b UV\,dz}{\int_{S_r^+}
              t^b
              (U^2+V^2)\,dS} \, .
\end{align}
\end{Lemma}

\begin{proof} The proof follows immediately from \eqref{H'}, \eqref{H'2} and \eqref{eq:D'}.
\end{proof}

In the next result we obtain an estimate on the $\nu_2$ component
of the function $\mathcal N'$.

\begin{Lemma} \label{l:limitN} Under the same assumptions of Lemma \ref{l:N'(r)} we have that
\begin{equation} \label{eq:N-ex-lim}
\gamma:=\lim_{r\to 0^+} \mathcal N(r)
\end{equation}
exists, it is finite and moreover $\gamma\ge 0$.
\end{Lemma}

\begin{proof} Let $\nu_1$ and $\nu_2$ be the functions introduced in
  Lemma \ref{l:N'(r)}. By \eqref{eq:prima},
  \eqref{eq:terza} and \eqref{eq:preliminare}, for any
  $r\in (0,\tilde r)$, with $\tilde r$ as in Lemma \ref{l:D(r)ge}, we
  have
\begin{align} \label{eq:bound-nu2}
|\nu_2(r)| & \le \frac r2+\frac{r\int_{B_r^+} t^b V^2 dz+r\int_{B_r^+} t^b |\nabla U|^2 dz+(N+b-1)\left|\int_{B_r^+} t^b UV\, dz \right| }{\int_{S_r^+} t^b (U^2+V^2)\, dS} \\[7pt]
\notag & \le \frac r2+\frac{\widetilde C_1 r^{N+b} D(r)+\widetilde C_2 r^{N+b+1}H(r)}{r^{N+b} H(r)}=\widetilde C_1 \mathcal N(r)+\widetilde C_3 r
\end{align}
for some suitable constants $\widetilde C_1, \widetilde C_2, \widetilde C_2>0$ independent of $r$.

Therefore, since by Cauchy-Schwarz inequality we have that
$\nu_1\ge 0$, we obtain that
\begin{equation} \label{eq:st-basso-N'}
\mathcal N'(r)\ge -\widetilde C_1 \mathcal N(r)-\widetilde C_3r
\end{equation}
which yields
\begin{equation} \label{eq:bound-N}
\mathcal N(r)\le e^{-\widetilde C_1 r} \left[e^{\widetilde C_1 \tilde r} \mathcal N(\tilde r)+\widetilde C_3 \int_r^{\tilde r} \rho e^{\widetilde C_1 \rho} d\rho
\right]\le \widetilde C_4 \qquad \text{for any } r\in (0,\tilde r) \, .
\end{equation}
This, combined with \eqref{eq:bound-nu2}, yields boundedness of $\nu_2$ in $(0,\tilde r)$.

This means that $\mathcal N'(r)=\nu_1(r)+\nu_2(r)$ is the sum of a
nonnegative function and of a bounded function so that
\begin{equation*}
\gamma:=\lim_{r\to 0^+} \mathcal N(r)=\mathcal N(\tilde r)-\int_0^{\tilde r} \nu_2(\rho)\, d\rho-\lim_{r \to 0^+} \int_r^{\tilde r} \nu_1(\rho) \, d\rho
\end{equation*}
exists. Finally, by \eqref{eq:seconda} and \eqref{eq:bound-N} we
conclude that $\gamma$ is finite and nonnegative.
\end{proof}

A first consequence of the previous monotonicity argument is the
following estimate of the function $H$.

\begin{Lemma}\label{l:uppb}
Letting $\gamma$ be as in Lemma \ref{l:limitN}, we have that
\begin{equation}\label{eq:24}
H(r)=O(r^{2\gamma})\quad \text{as $r\to0^+$}.
\end{equation}
Furthermore, for any $\sigma>0$ there exist
$K(\sigma)>0$ and $r_\sigma\in (0,r_0)$ depending on $\sigma$ such that
\begin{equation} \label{2ndest}
H(r)\geq K(\sigma)\,
  r^{2\gamma+\sigma} \quad \text{for all } r\in (0, r_\sigma) \, .
\end{equation}
\end{Lemma}

\begin{proof} The proof is quite standard once we have proved \eqref{eq:N-ex-lim}, see the proof of \cite[Lemma 3.16]{FF} for the details.
\end{proof}

\section{A blow-up procedure} \label{s:blow-up}

In order to exploit the monotonicity formula obtained in Section
\ref{s:Almgren} and to obtain asymptotic estimates on solutions to
\eqref{eq:system}, we proceed with a blow-up argument.

\begin{Lemma} \label{l:blow-up}
Let $(U,V)\in
H^1(B_R^+;t^b)\times H^1(B_R^+;t^b)$ be a nontrivial solution to
\eqref{eq:system} in the sense of \eqref{eq:11}--\eqref{eq:10} and
\eqref{eq:9}.
Let $\mathcal N$ be the function defined in
\eqref{eq:def-N} and let $\gamma$ be as in Lemma \ref{l:limitN}.
Then the following statements hold true:

\begin{itemize}
\item[$(i)$] there exists $\ell\in \N$ such that
\begin{equation*}
\gamma=-\frac{N+b-1}2+\sqrt{\left(\frac{N+b-1}2\right)^2+\mu_\ell}
\end{equation*}
with $\mu_\ell$ as in Section \ref{intro};

\item[$(ii)$] for any sequence $\lambda_n \downarrow 0$ there
exists a subsequence $\{\lambda_{n_k}\}_{k\in \N}$ and $2M_\ell$
real constants $\beta_{\ell,m}, \beta'_{\ell,m}$,
$m=1,\dots,M_\ell$, such that $\sum_{m=1}^{M_\ell}
\Big[(\beta_{\ell,m})^2+(\beta'_{\ell,m})^2\Big]=1$ and
\begin{equation*}
\frac{U(\lambda_{n_k}z)}{\sqrt{H(\lambda_{n_k})}} \to |z|^\gamma
\sum_{m=1}^{M_\ell} \beta_{\ell,m}
Y_{\ell,m}\left(\frac{z}{|z|}\right) \, , \qquad
\frac{V(\lambda_{n_k}z)}{\sqrt{H(\lambda_{n_k})}} \to |z|^\gamma
\sum_{m=1}^{M_\ell} \beta_{\ell,m}'
Y_{\ell,m}\left(\frac{z}{|z|}\right)
\end{equation*}
weakly in $H^1(B_1^+;t^b)$ and strongly in $H^1(B_r^+;t^b)$ for
any $r\in (0,1)$, with $Y_{\ell,m}$ as in Section~\ref{intro}.
\end{itemize}
\end{Lemma}

\begin{proof} Let us define the following scaled functions
\begin{equation}\label{eq:4}
U_\lambda(z):=\frac{U(\lambda z)}{\sqrt{H(\lambda)}}\, , \qquad
V_\lambda(z):=\frac{V(\lambda z)}{\sqrt{H(\lambda)}},
\end{equation}
which satisfy
\begin{equation*}
\Delta_b U_\lambda=\lambda^2 V_\lambda \quad \text{and} \quad
\int_{S_1^+} t^b(U_\lambda^2+V_\lambda^2)\, dS=1 \, .
\end{equation*}
Using a change of variable, \eqref{eq:prima} and Lemma
\ref{l:limitN}, one sees that
\begin{equation*}
\int_{B_1^+} t^b(|\nabla U_\lambda|^2+|\nabla V_\lambda|^2) \,
dz\le 2\mathcal N(\lambda)+\frac{2\lambda^2}{N+b-1}=O(1) \qquad
\text{as } \lambda\to 0^+,
\end{equation*}
which combined with \eqref{eq:PA2-2} yields that
\begin{equation*}
\{U_\lambda\}_{\lambda\in (0,\tilde\lambda)} \quad \text{and}
\quad \{V_\lambda\}_{\lambda\in (0,\tilde\lambda)} \qquad
\text{are bounded in } H^1(B_1^+;t^b)
\end{equation*}
for some $\tilde\lambda$ small enough.
Hence, for any sequence $\lambda_n \downarrow 0$, there exists a
subsequence $\lambda_{n_k}\downarrow 0$ and two functions
$\widetilde U, \widetilde V\in H^1(B_1^+;t^b)$ such that
$U_{\lambda_{n_k}}\rightharpoonup \widetilde U$,
$V_{\lambda_{n_k}}\rightharpoonup \widetilde V$ weakly in
$H^1(B_1^+;t^b)$.

By compactness of the trace map  $H^1(B_1^+;t^b) \hookrightarrow L^2(S_1^+;t^b)$, see \cite[Section 2.2]{FF}, we obtain
\begin{equation} \label{eq:normalization}
\int_{S_1^+} t^b (\widetilde U^2+\widetilde V^2) \, dS=1 \, ,
\end{equation}
which implies that $(\widetilde U,\widetilde V)\not
\equiv (0,0)$. We observe that the couple $(U_\lambda,V_\lambda)$
weakly solves
\begin{equation*}
\begin{cases}
\Delta_b U_\lambda=\lambda^2 V_\lambda & \qquad \text{in } B_1^+ \, , \\
\Delta_b V_\lambda=0 & \qquad \text{in } B_1^+ \, , \\
\lim_{t\to 0^+} t^b \partial_t U_\lambda=\lim_{t\to 0^+} t^b \partial_t V_\lambda=0 & \qquad \text{on } B_1' \, .
\end{cases}
\end{equation*}
This means that
\begin{equation*}
\int_{B_1^+} t^b \nabla U_\lambda \nabla \varphi \, dz=-\lambda^2
\int_{B_1^+} t^b V_\lambda \varphi \, dz  \quad \text{and} \quad
\int_{B_1^+} t^b \nabla V_\lambda \nabla \varphi \, dz=0 \, ,
\end{equation*}
for any $\varphi \in H^1_0(\Sigma_1^+;t^b)$ with
$H^1_0(\Sigma_1^+;t^b)=H^1_0(\Sigma_1^+(0);t^b)$ as in Section
\ref{s:preliminary}.

From the weak convergences $U_{\lambda_{n_k}}\rightharpoonup \widetilde U$,  $V_{\lambda_{n_k}}\rightharpoonup \widetilde V$ in $H^1(B_1^+;t^b)$, we deduce that
\begin{equation*}
\int_{B_1^+} t^b \nabla \widetilde U \nabla \varphi \, dz=0 \, ,
\quad \text{and} \quad \int_{B_1^+} t^b \nabla \widetilde V \nabla
\varphi \, dz=0 \, , \quad \text{for any } \varphi \in
H^1_0(\Sigma_1^+;t^b) \, ,
\end{equation*}
which means that the couple $(\widetilde U,\widetilde V)$ weakly solves
\begin{equation} \label{eq:tilde-U-V}
\begin{cases}
\Delta_b \widetilde U=0 & \qquad \text{in } B_1^+ \, , \\
\Delta_b \widetilde V=0 & \qquad \text{in } B_1^+ \, , \\
\lim_{t\to 0^+} t^b \partial_t \widetilde U=\lim_{t\to 0^+} t^b \partial_t \widetilde V=0 & \qquad \text{on } B_1' \, .
\end{cases}
\end{equation}
By Propositions \ref{p:6.5}-\ref{p:6.4} we have that, for any
$r\in (0,1)$,
$$
\{\nabla_x U_\lambda\}_{\lambda\in (0,\tilde \lambda)}\, , \quad
\{\nabla_x V_\lambda\}_{\lambda\in (0,\tilde \lambda)}\, , \quad
\{t^b \partial_t U_\lambda\}_{\lambda\in (0,\tilde \lambda)}\, ,
\quad \{t^b \partial_t V_\lambda\}_{\lambda\in (0,\tilde \lambda)}
$$
are bounded in $C^{0,\beta}(\overline B_r^+)$ for some $\beta \in
(0,1)$; hence by the Ascoli-Arzel\`a Theorem we deduce that these
families of functions are uniformly convergent in $\overline
B_r^+$ up to subsequences. In particular, we have that
$U_{\lambda_{n_k}}\to \widetilde U$ and $V_{\lambda_{n_k}}\to
\widetilde V$ strongly in $H^1(B_r^+;t^b)$ for any $r\in (0,1)$.

Now, for any $k\in \N$ and $r\in (0,1)$ we define the functions
\begin{align*}
& D_k(r):=r^{-N-b+1} \int_{B_r^+} t^b \Big(|\nabla U_{\lambda_{n_k}}|^2+|\nabla V_{\lambda_{n_k}}|^2+\lambda_{n_k}^2 U_{\lambda_{n_k}} V_{\lambda_{n_k}}\Big) \, dz  \, , \\[7pt]
& H_k(r):=r^{-N-b} \int_{S_r^+} t^b \Big(U_{\lambda_{n_k}}^2+V_{\lambda_{n_k}}^2\Big) \, dS \, .
\end{align*}
We observe that
\begin{equation} \label{eq:scaling}
\mathcal N_k(r):=\frac{D_k(r)}{H_k(r)}=\frac{D(\lambda_{n_k}r)}{H(\lambda_{n_k}r)}=\mathcal N(\lambda_{n_k}r) \qquad \text{for any } r\in (0,1) \, .
\end{equation}
Next, if we define
\begin{align*}
& \widetilde D(r):=r^{-N-b+1} \int_{B_r^+} t^b \Big(|\nabla \widetilde U|^2+|\nabla \widetilde V|^2\Big) \, dz \, ,  \\[7pt]
& \widetilde H(r):=r^{-N-b} \int_{S_r^+} t^b \Big(\widetilde
U^2+\widetilde V^2\Big) \, dS \, ,
\end{align*}
the strong convergences $U_{\lambda_{n_k}}\to \widetilde U$ and
$V_{\lambda_{n_k}}\to \widetilde V$ in $H^1(B_r^+;t^b)$ yield
\begin{equation} \label{eq:strong-D-H}
D_k(r)\to \widetilde D(r) \quad \text{and} \quad H_k(r)\to \widetilde H(r) \qquad \text{for any } r\in (0,1) \, .
\end{equation}

We claim that $\widetilde H(r)>0$ for any $r\in (0,1)$. Indeed, if there exists $\bar r \in (0,1)$ such that $\widetilde H(\bar r)=0$ then by
\eqref{eq:tilde-U-V} and integration by parts we would have
\begin{align} \label{eq:nabla-const}
0=\int_{B_{\bar r}^+} {\rm div}(t^b\nabla \widetilde U)\widetilde U \, dz=-\int_{B_{\bar r}^+} t^b |\nabla \widetilde U|^2 dz  \, .
\end{align}
Since $\widetilde U\in H^1_0(\Sigma_{\bar r}^+;t^b)$, combining
\eqref{eq:nabla-const} with \eqref{eq:PA2-2}, we conclude that
$\widetilde U\equiv 0$ in $B_{\bar r}^+$ and, by the classical
unique continuation principle for uniformly elliptic operators
with regular coefficients, we conclude that $\widetilde U\equiv 0$
in $B_1^+$. With the same argument we also deduce that $\widetilde
V\equiv 0$ in $B_1^+$. We have shown that $(\widetilde
U,\widetilde V)\equiv (0,0)$ in $B_1^+$ thus contradicting
\eqref{eq:normalization}.

The validity of the preceding claim allows to define the function
$\widetilde{\mathcal N}(r):=\frac{\widetilde D(r)}{\widetilde
H(r)}$ for any $r\in (0,1)$.

By \eqref{eq:scaling}, \eqref{eq:strong-D-H} and Lemma \ref{l:limitN}, we infer
\begin{equation} \label{eq:tilde-N}
\widetilde{\mathcal N}(r)=\lim_{k\to \infty} \mathcal N_k(r)=\lim_{k\to +\infty} \mathcal N(\lambda_{n_k} r)=\gamma \, .
\end{equation}
This shows that $\widetilde{\mathcal N}$ is constant in $(0,1)$ so
that $\widetilde{\mathcal N}'(r)=0$ for any $r\in (0,1)$.
Therefore, adapting Lemma \ref{l:N'(r)} to the couple $(\widetilde
U,\widetilde V)$, we infer that
\begin{equation*}
  \int_{S_r^+} t^b
  \left(\left|\frac{\partial \widetilde U}{\partial \nu}\right|^2+\left|\frac{\partial \widetilde V}{\partial \nu}\right|^2\right) dS \cdot
    \int_{S_r^+} t^b
  (\widetilde U^2+\widetilde V^2)\,dS-\left[
\int_{S_r^+} t^b
  \left(\widetilde U\frac{\partial \widetilde U}{\partial \nu}+\widetilde V\frac{\partial \widetilde V}{\partial \nu}\right)\, dS\right]^{\!2}
=0
\end{equation*}
for any $r\in (0,1)$. This represents an equality in the Cauchy-Schwarz inequality in the Hilbert space $L^2(S_r^+;t^b)\times L^2(S_r^+;t^b)$ thus showing that
$(\widetilde U,\widetilde V)$ and $\left(\frac{\partial \widetilde U}{\partial \nu},\frac{\partial \widetilde V}{\partial \nu}\right)$ are parallel vectors
in $L^2(S_r^+;t^b)\times L^2(S_r^+;t^b)$. Hence, there exists a function $\eta=\eta(r)$ defined for any $r\in (0,1)$ such that
 $\left(\frac{\partial \widetilde U}{\partial \nu}(r\theta),\frac{\partial \widetilde V}{\partial \nu}(r\theta)\right)=\eta(r)(\widetilde U(r\theta),\widetilde U(r\theta))$ for any $r\in (0,1)$ and $\theta\in \SF$.
By integration we obtain
\begin{align}
\label{separate}
&\widetilde U(r\theta)=e^{\int_1^r \eta(s)ds} \widetilde U(\theta)
=\varphi(r) \Psi_1(\theta), \quad  r\in(0,1), \ \theta\in \SF,\\
\label{separate2}
&\widetilde V(r\theta)=e^{\int_1^r \eta(s)ds} \widetilde V(\theta)
=\varphi(r) \Psi_2(\theta), \quad  r\in(0,1), \ \theta\in \SF,
\end{align}
where $\varphi(r)=e^{\int_1^r \eta(s)ds}$ and
$\Psi_1=\widetilde U\big|_{\SF}$, $\Psi_2(\theta)=\widetilde V\big|_{\SF}$.
From \eqref{eq:tilde-U-V}, \eqref{separate} and \eqref{separate2}, it follows that
\begin{equation}\label{eq:28}
\begin{cases}
r^{-N}\big(r^{N+b}\varphi'(r)\big)'\theta_{N+1}^b \Psi_1(\theta)+
r^{b-2}\varphi(r){\rm div}_{\SF}(\theta_{N+1}^b \nabla_{\SF} \Psi_1(\theta))
=0 & \qquad \text{in } \SF, \\[6pt]
{\ds \lim_{\theta_{N+1}\to 0^+} \theta_{N+1}^b \nabla_{\SF} \Psi_1(\theta)\cdot {\bf e}_{N+1}=0} \, ,
\end{cases}
\end{equation}
and
\begin{equation}\label{eq:29}
\begin{cases}
r^{-N}\big(r^{N+b}\varphi'(r)\big)'\theta_{N+1}^b \Psi_2(\theta)+
r^{b-2}\varphi(r){\rm div}_{\SF}(\theta_{N+1}^b \nabla_{\SF} \Psi_2(\theta))
=0 & \qquad \text{in } \SF, \\[6pt]
{\ds \lim_{\theta_{N+1}\to 0^+} \theta_{N+1}^b \nabla_{\SF} \Psi_2(\theta)\cdot {\bf e}_{N+1}=0} \, .
\end{cases}
\end{equation}
Taking $r$ fixed, we deduce that $\Psi_1,\Psi_2$ are either zero or
eigenfunctions of \eqref{eq:Eigen-Weight} associated to the same eigenvalue.
Therefore there exist $\ell\in\N$, $\{\beta_{\ell,m},\beta'_{\ell,m}\}_{m=1}^{M_\ell}\subset \R$ such that
$$
\begin{cases}
-{\rm div}_{\SF}(\theta_{N+1}^b \nabla_{\SF} \Psi_1)=\mu_\ell \theta_{N+1}^b \Psi_1 & \qquad  \text{in } \SF, \\[6pt]
{\ds \lim_{\theta_{N+1}\to 0^+} \theta_{N+1}^b \nabla_{\SF} \Psi_1(\theta)\cdot {\bf e}_{N+1}=0}  \, ,
\end{cases}
$$

$$
\begin{cases}
-{\rm div}_{\SF}(\theta_{N+1}^b \nabla_{\SF} \Psi_2)=\mu_\ell \theta_{N+1}^b \Psi_2  &  \qquad \text{in } \SF, \\[6pt]
{\ds \lim_{\theta_{N+1}\to 0^+} \theta_{N+1}^b \nabla_{\SF} \Psi_2(\theta)\cdot {\bf e}_{N+1}=0} \, ,
\end{cases}
$$
and
\[
\Psi_1=\sum_{m=1}^{M_\ell}\beta_{\ell,m}Y_{\ell,m},\quad
\Psi_2=\sum_{m=1}^{M_\ell}\beta'_{\ell,m}Y_{\ell,m}.
\]
In view of \eqref{eq:normalization} we have that $\int_{{\mathbb S}^{N}_+} \theta_{N+1}^b (\Psi_1^2+\Psi_2^2)\,dS=1$ and hence
\[
\sum_{m=1}^{M_\ell}[(\beta_{\ell,m})^2+(\beta'_{\ell,m})^2]=1.
\]
Since $\Psi_1$ and $\Psi_2$ are not both identically zero, from
\eqref{eq:28} and \eqref{eq:29} it follows that $\varphi(r)$ solves the equation
\[
\varphi''(r)+\frac{N+b}r\, \varphi'(r)-\frac{\mu_\ell}{r^2}\, \varphi(r)=0
\]
and hence $\varphi(r)=c_1 r^{\sigma_\ell^+}+c_2 r^{\sigma_\ell^-}$
for some $c_1,c_2\in\R$ where
\begin{align} \label{eq:sigma+-sigma-}
& \sigma_\ell^+=-\tfrac{N+b-1}2+\sqrt{\left(\tfrac{N+b-1}2\right)^2+\mu_\ell} \, , \qquad  \sigma_\ell^-=-\tfrac{N+b-1}2-\sqrt{\left(\tfrac{N+b-1}2\right)^2+\mu_\ell} \, .
\end{align}
Since either $|z|^{\sigma_\ell^-}\Psi_1(\frac{z}{|z|}) \notin
H^1(B_1^+;t^b)$ or $|z|^{\sigma_\ell^-} \Psi_2(\frac{z}{|z|})
\notin H^1(B_1^+;t^b)$ as one can deduce by
\eqref{eq:Sobolev-traccia}, (we recall that
$(\Psi_1,\Psi_2)\not\equiv(0,0)$), we have that $c_2=0$ and
$\varphi(r)=c_1 r^{\sigma_\ell^+}$. Moreover, from $\varphi(1)=1$
we deduce  that $c_1=1$. Therefore
\begin{equation} \label{expw}
\widetilde U(r\theta)=r^{\sigma_\ell^+} \Psi_1(\theta),\quad
\widetilde V(r\theta)=r^{\sigma_\ell^+} \Psi_2(\theta),  \quad
\text{for all }r\in (0,1)\text{ and }\theta\in \SF.
\end{equation}
From \eqref{expw} and the fact that
\[
\int_{\SF} \theta_{N+1}^b
(\Psi_1^2+\Psi_2^2)\,dS=1\quad\text{and}\quad \int_{\SF}
\theta_{N+1}^b (|\nabla_{\SF} \Psi_1|^2+|\nabla_{\SF}
\Psi_2|^2)\,dS=\mu_\ell
\]
we infer
\begin{align*}
 \widetilde D (r)=\sigma_\ell^+  \, r^{2\sigma_\ell^+} \quad \text{and} \quad
  \widetilde  H(r)=r^{2 \sigma_\ell^+} \, .
\end{align*}
By \eqref{eq:tilde-N} we then have
$\gamma=\widetilde{\mathcal N}(r)=\frac{\widetilde
D(r)}{\widetilde H(r)}=\sigma_\ell^+$. The proof of the lemma is
thereby
complete.
\end{proof}

\begin{Lemma} \label{l:5.2} Suppose that all the assumptions of Lemma
\ref{l:blow-up} hold true and let $\gamma$ be as in Lemma
\ref{l:limitN}. Then the limit
\begin{equation*}
\lim_{r\to 0^+} r^{-2\gamma} H(r)
\end{equation*}
exists and it is finite.
\end{Lemma}

\begin{proof} Thanks to Lemma \ref{l:uppb}, it is sufficient to
show that the limit exists.

By \eqref{H'2} and Lemma \ref{l:limitN} we have
\begin{align} \label{eq:D-H-r2gamma}
 \frac{d}{dr} \frac{H(r)}{r^{2\gamma}}
=2r^{-2\gamma-1} H(r)[\mathcal N(r)-\gamma]
 =2r^{-2\gamma-1} H(r) \int_0^r \mathcal N'(\rho)\, d\rho
\, .
\end{align}
Since $\mathcal N$ is bounded in a right neighborhood of $0$, by
\eqref{eq:st-basso-N'} we deduce that $\mathcal N'$ is bounded
from below in a right neighborhood of $0$. Hence there exist a
constant $C>0$ and a nonnegative function $\omega \in
L^1_{{\rm loc}}(0,r_0)$ such that $\mathcal N'(r)=-C+\omega(r)$
for any $r\in (0,r_0)$.

Therefore, integrating \eqref{eq:D-H-r2gamma} in $(r,r_0)$, we
obtain
\begin{align*}
\frac{H(r_0)}{r_0^{2\gamma}}-\frac{H(r)}{r^{2\gamma}}
&=\int_r^{r_0} 2\rho^{-2\gamma-1} H(\rho) \left(\int_0^\rho
\omega(\tau)\, d\tau\right)d\rho-2C\int_r^{r_0} \rho^{-2\gamma}
H(\rho) \, d\rho \, .
\end{align*}
Since $\omega\ge 0$ then $\lim_{r\to 0^+} \int_r^{r_0}
2\rho^{-2\gamma-1} H(\rho) \left(\int_0^\rho \omega(\tau)\,
d\tau\right)d\rho$ exists. Moreover we also have that $\lim_{r\to
0^+} \int_r^{r_0} \rho^{-2\gamma} H(\rho) \, d\rho=\int_0^{r_0}
\rho^{-2\gamma} H(\rho) \, d\rho$ exists and it is finite being $\rho^{-2\gamma}
H(\rho)\in L^1(0,r_0)$ thanks to \eqref{eq:24}. This completes the
proof of the lemma.
\end{proof}

Let us expand $U$ and $V$  as
\begin{equation}\label{eq:8}
U(z)=U(\lambda
\theta)=\sum_{k=0}^\infty\sum_{m=1}^{M_k}\varphi_{k,m}(\lambda)Y_{k,m}(\theta),\quad
V(z)=V(\lambda \theta)=\sum_{k=0}^\infty\sum_{m=1}^{M_k}\widetilde\varphi_{k,m}(\lambda)Y_{k,m}(\theta)
\end{equation}
where $\lambda=|z|\in(0,r_0)$, $\theta=z/|z|\in{{\mathbb S}^{N}_+}$, and
\begin{equation}\label{eq:37}
  \varphi_{k,m}(\lambda)=\int_{{\mathbb S}^{N}_+} \theta_{N+1}^b U(\lambda\,\theta) Y_{k,m}(\theta)\,dS(\theta),\quad
\widetilde\varphi_{k,m}(\lambda)=\int_{{\mathbb S}^{N}_+} \theta_{N+1}^b  V(\lambda\,\theta)
    Y_{k,m}(\theta)\,dS(\theta).
\end{equation}

\begin{Lemma} \label{l:5.3} Suppose that all the assumptions of Lemma
\ref{l:blow-up} hold true. Let $\ell$ be as in Lemma \ref{l:blow-up} and let $\varphi_{\ell,m}$ and $\widetilde \varphi_{\ell,m}$, $m=1,\dots,M_\ell$, be as in
\eqref{eq:37}. Then for any $1\le m\le M_\ell$ we have
\begin{align}\label{eq:7}
& \varphi_{\ell,m}(\lambda)=c_1^{\ell,m}
  \lambda^{\sigma_\ell^+}+\tfrac{d_1^{\ell,m}}{K(N,b,\ell)} \,
  \lambda^{\sigma_\ell^+ +2} \quad \text{and}\quad \widetilde \varphi_{\ell,m}(\lambda)=d_1^{\ell,m} \lambda^{\sigma_\ell^+},
\end{align}
where $K(N,b,\ell):=(\sigma_\ell^++2)(\sigma_\ell^++1)+(N+b)(\sigma_\ell^++2)-\mu_\ell$,
\begin{align*}
& d_1^{\ell,m}\!=\!R^{-\sigma_\ell^+}\!\!\! \int_{\SF}\!\! \theta_{N+1}^b V(R\theta)Y_{\ell,m}(\theta) \, dS(\theta)\, ,
\ c_1^{\ell,m}\!=\!R^{-\sigma_\ell^+} \!\!\!\int_{\SF} \theta_{N+1}^b U(R\theta)Y_{\ell,m}(\theta) \, dS(\theta)-\tfrac{d_1^{\ell,m}}{K(N,b,\ell)} \, R^2
\end{align*}
with $\sigma_\ell^+$ as in \eqref{eq:sigma+-sigma-}. Furthermore $\varphi_{k,m}\equiv
\widetilde\varphi_{k,m}\equiv 0$ for any $1\le k<\ell$ and $1\le m\le M_k$.
\end{Lemma}

\begin{proof} From the Parseval identity it follows that
\begin{equation}\label{eq:17bis}
H(\lambda)=\sum_{k=0}^\infty\sum_{m=1}^{M_k}\big(\varphi_{k,m}^2(\lambda)+\widetilde\varphi_{k,m}^2(\lambda)\big),
\qquad \text{for any }0<\lambda\leq R.
\end{equation}By \eqref{eq:system} we have that for any $m=1,\dots, M_\ell$
\begin{equation} \label{eq:syst-ode}
\begin{cases}
\varphi_{\ell,m}''(\lambda)+\frac{N+b}\lambda \, \varphi_{\ell,m}'(\lambda)-\frac{\mu_\ell}{\lambda^2} \, \varphi_{\ell,m}(\lambda)=\widetilde \varphi_{\ell,m}(\lambda) \, ,\\[6pt]
\widetilde \varphi_{\ell,m}''(\lambda)+\frac{N+b}\lambda \, \widetilde \varphi_{\ell,m}'(\lambda)-\frac{\mu_\ell}{\lambda^2} \, \widetilde \varphi_{\ell,m}(\lambda)=0 \, .
\end{cases}
\end{equation}
By direct calculation we obtain
\begin{equation*}
\widetilde \varphi_{\ell,m}(\lambda)=d_1^{\ell,m} \lambda^{\sigma_\ell^+}+d_2^{\ell,m} \lambda^{\sigma_\ell^-}
\end{equation*}
for some constants $d_1^{\ell,m}, d_2^{\ell,m}$ where $\sigma_\ell^+$ and $\sigma_\ell^-$ are defined in \eqref{eq:sigma+-sigma-}.

Now, by \eqref{eq:17bis}, \eqref{eq:24} and the fact that
$\gamma=\sigma_\ell^+$, we infer $d_2^{\ell,m}=0$ so that
$\widetilde \varphi_{\ell,m}(\lambda)=d_1^{\ell,m}
\lambda^{\sigma_\ell^+}$.

In particular, \eqref{eq:syst-ode} and direct calculation yield
\begin{equation*}
\varphi_{\ell,m}(\lambda)=c_1^{\ell,m}
\lambda^{\sigma_\ell^+}+c_2^{\ell,m} \lambda^{\sigma_\ell^-}+
\tfrac{d_1^{\ell,m}}{(\sigma_\ell^++2)(\sigma_\ell^++1)+(N+b)(\sigma_\ell^++2)-\mu_\ell} \, \lambda^{\sigma_\ell^+ +2}
\end{equation*}
for some constants $c_1^{\ell,m}, c_2^{\ell,m}$. Exploiting again
\eqref{eq:17bis}, \eqref{eq:24} and the fact that
$\gamma=\sigma_\ell^+$ we deduce that $c_2^{\ell,m}=0$. The proof
of the first part of the lemma now easily follows. In order to
prove the second part of the lemma one can proceed exactly as
above replacing $\ell$ with $k$ in \eqref{eq:syst-ode} and solving
the corresponding equation. The conclusion now follows from
\eqref{eq:17bis} and \eqref{eq:24}.
\end{proof}

\begin{remark}\label{rem:tutti_k}
  We observe that the representation formula \eqref{eq:7} actually
  holds for  $\varphi_{k,m}$ and $\widetilde \varphi_{k,m}$ also for
  $k\neq\ell$; in this case to prove that $d_2^{k,m}=c_2^{k,m}=0$ we
  can use the fact that $U,V\in H^1(B_R^+;t^b)$.
\end{remark}

\begin{Lemma} Suppose that all the assumptions of Lemma \ref{l:blow-up} hold true. Then we have
\begin{equation} \label{eq:lim-pos}
\lim_{r\to 0^+} r^{-2\gamma} H(r)>0 \, .
\end{equation}
\end{Lemma}

\begin{proof} By Lemma \ref{l:5.2} we know that the limit in \eqref{eq:lim-pos} exists and it is nonnegative and finite. Suppose by contradiction that
$\lim_{\lambda\to 0^+} \lambda^{-2 \gamma} H(\lambda)=0$. Then by \eqref{eq:17bis} we deduce that for any $1\le m\le M_\ell$, with $\ell$ as in Lemma \ref{l:blow-up},
\begin{equation*}
\lim_{\lambda\to 0^+} \lambda^{-\gamma} \varphi_{\ell,m}(\lambda)=0 \qquad \text{and} \qquad
\lim_{\lambda\to 0^+} \lambda^{-\gamma} \widetilde \varphi_{\ell,m}(\lambda)=0 \, .
\end{equation*}
We recall that by Lemma \ref{l:blow-up} we have $\gamma=\sigma_\ell^+$ and hence by Lemma \ref{l:5.3} we infer $c_1^{\ell,m}=d_1^{\ell,m}=0$ so that
\begin{equation} \label{eq:annullamento}
\varphi_{\ell,m}(\lambda)=\widetilde \varphi_{\ell,m}(\lambda)=0 \qquad \text{for any } \lambda\in (0,R) \text{ and } 1\le m\le M_\ell \, .
\end{equation}
From Lemma \ref{l:blow-up}, for every sequence $\lambda_n\to0^+$, there exist a subsequence
$\{\lambda_{n_k}\}_{k\in\N}$ and $2M_\ell$ real constants
$\beta_{\ell,m},\beta'_{\ell,m}$, $m=1,2,\dots,M_\ell$,  such that
\begin{equation}\label{eq:35}
\sum_{m=1}^{M_\ell}((\beta_{\ell,m})^2+(\beta'_{\ell,m})^2)=1
\end{equation}
and
\[
U_{\lambda_{n_k}}\to |z|^{\sigma_\ell^+}
\sum_{m=1}^{M_\ell}\beta_{\ell,m} Y_{\ell,m}\Big(\frac
z{|z|}\Big) \, , \qquad
V_{\lambda_{n_k}}\to
|z|^{\sigma_\ell^+}
\sum_{m=1}^{M_\ell}\beta'_{\ell,m}Y_{\ell,m}\Big(\frac
z{|z|}\Big), \quad\text{as }k\to+\infty,
\]
 weakly in $H^1(B_1^+;t^b)$ and hence strongly in $L^2(S_1^+;t^b)$,
 where $U_\lambda,V_\lambda$ have been defined in \eqref{eq:4}. Combining this with \eqref{eq:annullamento}, it follows that, for any $m=1,2,\dots,M_\ell$,
\begin{align*}
& \beta_{\ell,m}=\lim_{k\to+\infty}(U_{\lambda_{n_k}},Y_{\ell,m})_{L^2(\SF;\theta_{N+1}^b)}
=\lim_{k\to+\infty} \frac{\varphi_{\ell,m}(\lambda_{n_k})}{\sqrt{H(\lambda_{n_k})}}=0  \, , \\[7pt]
& \beta'_{\ell,m}=\lim_{k\to+\infty}(V_{\lambda_{n_k}},Y_{\ell,m})_{L^2(\SF;\theta_{N+1}^b)}
=\lim_{k\to+\infty} \frac{\widetilde \varphi_{\ell,m}(\lambda_{n_k})}{\sqrt{H(\lambda_{n_k})}}=0  \, ,
\end{align*}
thus contradicting \eqref{eq:35}.
\end{proof}
Then we prove the following lemma. 
\begin{Lemma} \label{t:5.5} Let $(U,V)\in H^1(B_R^+;t^b)\times
  H^1(B_R^+;t^b)$ be  a weak solution to system
\eqref{eq:system} such that $(U,V)\not=(0,0)$. For any $k\in \N$ define
\begin{equation} \label{eq:sigma-k}
\sigma_k^+:=-\frac{N+b-1}2+\sqrt{\left(\frac{N+b-1}2\right)^2+\mu_k} \, .
\end{equation}
with $\mu_k$ as in Section \ref{intro}. Then there exists
$\ell\in\N$ such that
\[
\lambda^{-\sigma_\ell^+}U(\lambda z)\to
|z|^{\sigma_\ell^+}
\sum_{m=1}^{M_\ell}\alpha_{\ell,m}Y_{\ell,m}\Big(\frac
z{|z|}\Big),\quad
\lambda^{-\sigma_\ell^+}V(\lambda z)\to
|z|^{\sigma_\ell^+}
\sum_{m=1}^{M_\ell}\alpha'_{\ell,m}Y_{\ell,m}\Big(\frac
z{|z|}\Big),
\]
strongly in $H^1(B_1^+;t^b)$ as $\lambda\to 0^+$, where
\begin{align} \label{eq:form-alpha}
  \alpha_{\ell,m}&=R^{-\sigma_\ell^+} \int_{\SF} \theta_{N+1}^b U(R\,\theta)
  Y_{\ell,m}(\theta)\,dS(\theta)-\tfrac{R^{2-\sigma_\ell^+}}{K(N,b,\ell)}
\int_{\SF} \theta_{N+1}^b V(R\theta)
    Y_{\ell,m}(\theta)\,dS(\theta) \, , \\
\notag  \alpha_{\ell,m}'& =R^{-\sigma_\ell^+}\int_{\SF}
\theta_{N+1}^b V(R\,\theta)
  Y_{\ell,m}(\theta)\,dS(\theta)
\end{align}
with $K(N,b,\ell)$ as in Lemma \ref{l:5.3} and
\begin{equation} \label{eq:neq0}
\sum_{m=1}^{M_\ell}((\alpha_{\ell,m})^2+(\alpha'_{\ell,m})^2)\neq 0 \, .
\end{equation}
Moreover for any $1\le m\le M_\ell$ we have
\begin{equation} \label{eq:id-phi}
\varphi_{\ell,m}(\lambda)=\alpha_{\ell,m} \lambda^{\sigma_\ell^+}+\frac{\alpha_{\ell,m}'}{K(N,b,\ell)} \lambda^{\sigma_\ell^++2} \, ,
\qquad  \widetilde\varphi_{\ell,m}(\lambda)=\alpha_{\ell,m}' \lambda^{\sigma_\ell^+} \, .
\end{equation}
\end{Lemma}

\nero
\begin{proof} From Lemma \ref{l:blow-up} and \eqref{eq:lim-pos} there exist
  $\ell\in \N$ such that, for every sequence $\lambda_n\to0^+$, there exist a subsequence
$\{\lambda_{n_k}\}_{k\in\N}$ and $2M_\ell$ real constants
$\alpha_{\ell,m},\alpha '_{\ell,m}$, $m=1,2,\dots,M_\ell$,  such that
$\sum_{m=1}^{M_\ell}((\alpha_{\ell,m})^2+(\alpha'_{\ell,m})^2)\neq0$
and
\begin{equation}\label{eq:40}
\lambda_{n_k}^{-\sigma_\ell^+}U(\lambda_{n_k}z)\to
|z|^{\sigma_\ell^+}
\sum_{m=1}^{M_\ell}\alpha_{\ell,m}Y_{\ell,m}\Big(\frac
z{|z|}\Big),\quad
\lambda_{n_k}^{-\sigma_\ell^+ }V(\lambda_{n_k}z)\to
|z|^{\sigma_\ell^+}
\sum_{m=1}^{M_\ell}\alpha'_{\ell,m}Y_{\ell,m}\Big(\frac
z{|z|}\Big),
\end{equation}
strongly in $H^1(B_r^+;t^b)$ for all $r\in(0,1)$, and then, by
homogeneity, strongly in $H^1(B_1^+;t^b)$.

By \eqref{eq:37}, \eqref{eq:40} and Lemma \ref{l:5.3} we deduce that
\begin{align*}
\alpha_{\ell,m}&=\lim_{k\to\infty}\lambda_{n_k}^{-\sigma_\ell^+}
\int_{\SF} \theta_{N+1}^b U(\lambda_{n_k}\,\theta) Y_{\ell,m}(\theta)\,dS(\theta) \\
&=\lim_{k\to\infty}\lambda_{n_k}^{-\sigma_\ell^+}\varphi_{\ell,m}(\lambda_{n_k})=
c_1^{\ell,m} \\
&=R^{-\sigma_\ell^+} \int_{\SF} \theta_{N+1}^b U(R\,\theta)
  Y_{\ell,m}(\theta)\,dS(\theta)-\tfrac{R^{2-\sigma_\ell^+}}{K(N,b,\ell)} \int_{\SF} \theta_{N+1}^b V(R\theta)
    Y_{\ell,m}(\theta)\,dS(\theta)
\end{align*}
and
\begin{align*}
\alpha'_{\ell,m}&=\lim_{k\to\infty}\lambda_{n_k}^{-\sigma_\ell^+}
\int_{\SF} \theta_{N+1}^b V(\lambda_{n_k}\,\theta) Y_{\ell,m}(\theta)\,dS(\theta)\\
&=\lim_{k\to\infty}\lambda_{n_k}^{-\sigma_\ell^+} \widetilde\varphi_{\ell,m}(\lambda_{n_k})=
d_1^{\ell,m}=R^{-\sigma_\ell^+}\int_{\SF} \theta_{N+1}^b V(R\,\theta)
  Y_{\ell,m}(\theta)\,dS(\theta)  \, .
\end{align*}
We observe that the coefficients $\alpha_{\ell,m},\alpha'_{\ell,m}$ depend neither on the sequence
$\{\lambda_n\}_{n\in\N}$ nor on its subsequence
$\{\lambda_{n_k}\}_{k\in\N}$. Hence the convergences in \eqref{eq:40}  hold as $\lambda\to 0^+$
and the lemma   is proved.~\end{proof}

We now state and prove the following theorem.

\begin{Theorem} \label{p:5.6} Let $(U,V)\in H^1(B_R^+;t^b)\times
  H^1(B_R^+;t^b)$ be  a weak solution to system
\eqref{eq:system} such that $(U,V)\not=(0,0)$.
Then there exists $\delta_1>0$ and a linear combination $\Psi_1\not\equiv 0$ of eigenfunctions of \eqref{eq:Eigen-Weight}, possibly corresponding to different eigenvalues, such that
\begin{equation} \label{eq:est-aux}
\lambda^{-\delta_1} \, U(\lambda z)\to |z|^{\delta_1} \, \Psi_1\left(\tfrac z{|z|}\right)
\end{equation}
strongly in $H^1(B_1^+;t^b)$ as $\lambda\to 0^+$.
Furthermore, if $V\not\equiv0$, there exists $\delta_2>0$ and a
linear combination $\Psi_2\not\equiv 0$ of eigenfunctions of
\eqref{eq:Eigen-Weight}, possibly corresponding to different
eigenvalues, such that
\begin{equation} \label{eq:est-aux-bis}
\lambda^{-\delta_2} \, V(\lambda z)\to |z|^{\delta_2} \, \Psi_2\left(\tfrac z{|z|}\right)
\end{equation}
strongly in $H^1(B_1^+;t^b)$ as $\lambda\to 0^+$.
\end{Theorem}

\begin{proof} We treat separately the proofs of \eqref{eq:est-aux} and \eqref{eq:est-aux-bis}.

\medskip
{\it Proof of \eqref{eq:est-aux}.} Let $\ell$ be as in Lemma \ref{t:5.5}. If at least one of the numbers $\alpha_{\ell,1},\dots,\alpha_{\ell,M_\ell}$
introduced in Lemma \ref{t:5.5} is different from zero then the proof of \eqref{eq:est-aux} follows immediately with $\delta_1=\sigma_\ell^+$ and
\begin{equation*}
\Psi_1(\theta)=\sum_{m=1}^{M_\ell} \alpha_{\ell,m} Y_{\ell,m}(\theta) \, .
\end{equation*}
Suppose now that $\alpha_{\ell,1}=\, \dots \, =\alpha_{\ell,M_\ell}=0$.
Let $k>\ell$ be such that
\begin{equation*}
\sigma_k^+>\sigma_\ell^++2 \quad \text{and} \quad \sigma_{k-1}^+\le \sigma_\ell^++2
\end{equation*}
with $\sigma_\ell^+,\dots,\sigma_k^+$ as in \eqref{eq:sigma-k}, and let
\begin{equation*}
\Sigma:=\{j\in \{\ell+1,\dots,k-1\}:\alpha_{j,m}\neq 0 \text{ for at least one } m\in \{1,\dots,M_j\}\}
\end{equation*}
with $\Sigma$ being possibly empty. Here $\alpha_{j,m}$
is defined as in \eqref{eq:form-alpha} replacing $\ell$  with $j$. When $\Sigma\neq \emptyset$ we put $J=\min \Sigma$.

We distinguish the two cases $\Sigma\neq \emptyset$ and $\Sigma=\emptyset$.

{\bf The case $\Sigma\neq \emptyset$.} We put
\begin{align*}
\omega(z)& :=U(z)-\sum_{j=1}^{k-1}
\sum_{m=1}^{M_j} \varphi_{j,m}(|z|) Y_{j,m}\left(\tfrac
z{|z|}\right) \\
& =U(z)-\sum_{j\in \Sigma}
\sum_{m=1}^{M_j} \alpha_{j,m}|z|^{\sigma_j^+} \, Y_{j,m}\left(\tfrac
z{|z|}\right)-\sum_{j=\ell}^{k-1}
\sum_{m=1}^{M_j} \frac{\alpha_{j,m}'}{K(N,b,j)}\, |z|^{\sigma_j^++2}\, Y_{j,m}\left(\tfrac
z{|z|}\right)
\end{align*}
for any $z\in B_R^+$, with $K(N,b,j):=(\sigma_j^++2)(\sigma_j^++1)+(N+b)(\sigma_j^++2)-\mu_j$. The last identity follows from the second
part of Lemma \ref{l:5.3} and Remark \ref{rem:tutti_k}.

It is not restrictive to assume that $\omega\not\equiv 0$, otherwise the conclusion is trivial.
We observe that $\omega$ is in the same position as the function $U$ in
Lemma \ref{t:5.5} so that applying that result to $\omega$ we deduce that
there exists $\widetilde \ell\ge 0$ such that
\begin{align} \label{eq:conv-H1}
& \lambda^{-\sigma_{\widetilde \ell}^+}\omega(\lambda z)\to
|z|^{\sigma_{\widetilde \ell}^+} \sum_{m=1}^{M_{\widetilde\ell}}
\widetilde \alpha_m \, Y_{\widetilde
\ell,m}\left(\frac{z}{|z|}\right)\, , \qquad
\lambda^{-\sigma_{\widetilde \ell}^+}\Delta_b \omega(\lambda z)\to
|z|^{\sigma_{\widetilde \ell}^+} \sum_{m=1}^{M_{\widetilde\ell}}
\widetilde \alpha'_m \, Y_{\widetilde
\ell,m}\left(\frac{z}{|z|}\right)
\end{align}
in $H^1(B_1^+;t^b)$ as $\lambda\to 0^+$, where $\widetilde
\alpha_m$ and $\widetilde \alpha'_m$ satisfy
\eqref{eq:form-alpha} and \eqref{eq:neq0} in which the roles of $U$
and $V$ in Lemma \ref{t:5.5} are replaced by $\omega$ and
$\Delta_b\omega$ respectively.

We claim that $\widetilde \ell\ge k$. We first observe that the Fourier coefficients $\varphi_{j,m},\widetilde\varphi_{j,m}$ corresponding to $\omega$ are all zero for any $1\le j\le k-1$ and $1\le m\le M_j$. On the other hand, by \eqref{eq:id-phi} we deduce that at least one of the functions
$\varphi_{\widetilde\ell,m}, \widetilde\varphi_{\widetilde\ell,m}$, $1\le m\le M_{\widetilde\ell}$, corresponding to $\omega$ is not the null function. This proves the validity of the claim.

Note that since $\widetilde\ell\ge k$, by \eqref{eq:form-alpha} and the orthogonality of $\{Y_{j,m}\}_{j\ge 0, 1 \le m\le M_j}$ in $L^2(\SF;\theta_{N+1}^b)$, we also deduce that
$\widetilde\alpha_m=\alpha_{\widetilde\ell,m}$ and
$\widetilde\alpha_m'=\alpha_{\widetilde\ell,m}'$ for any $1\le
m\le M_{\widetilde\ell}$.

By \eqref{eq:conv-H1} and the fact that $\widetilde\ell\ge k$,
$\lambda^{-\sigma_k^+}\omega(\lambda z)$ and
$\lambda^{-\sigma_k^+}\Delta_b \omega(\lambda z)$ remain uniformly
bounded in $H^1(B_1^+;t^b)$ as $\lambda\to 0^+$.

We observe that from the definitions of $\omega$, $\Sigma$ and $J$ we have
$\sigma_j^++2\ge \sigma_\ell^++2\ge \sigma_{k-1}^+\ge \sigma_i^+$ for any $\ell+1\le i,j\le k-1$.

Therefore, if $\sigma_J^+<\sigma_\ell^++2$ the proof of \eqref{eq:est-aux} then follows with $\delta_1=\sigma_J^+$ and
$$
\Psi_1(\theta)=\sum_{m=1}^{M_J} \alpha_{J,m} \, Y_{J,m}(\theta) \, , \qquad \theta\in \SF \, .
$$
Suppose now that $\sigma_J^+=\sigma_\ell^++2$. In this case we necessarily have $J=k-1$ so that \eqref{eq:est-aux} follows with $\delta_1=\sigma_{k-1}^+=\sigma_\ell^++2$ and
$$
\Psi_1(\theta)= \sum_{m=1}^{M_{k-1}} \alpha_{k-1,m} \, Y_{k-1,m}(\theta)+\sum_{m=1}^{M_\ell} \frac{\alpha_{\ell,m}'}{K(N,b,\ell)} \, Y_{\ell,m}(\theta) \, ,
\qquad \theta\in \SF \, .
$$

{\bf The case $\Sigma=\emptyset$.} As in the previous case we define
\begin{align*}
\omega(z)& :=U(z)-\sum_{j=1}^{k-1}
\sum_{m=1}^{M_j} \varphi_{j,m}(|z|) Y_{j,m}\left(\tfrac
z{|z|}\right)=U(z)-\sum_{j=\ell}^{k-1}\sum_{m=1}^{M_j} \frac{\alpha_{j,m}'}{K(N,b,j)}\, |z|^{\sigma_j^++2} Y_{j,m}\left(\tfrac
z{|z|}\right) ,
\end{align*}
for any $z\in B_R^+$, where the last identity follows from the second
part of Lemma \ref{l:5.3}, Remark 
\ref{rem:tutti_k}, and the fact that $\Sigma=\emptyset$. Proceeding as in this case $\Sigma\neq \emptyset$ we find $\widetilde\ell\ge k$ such that
\eqref{eq:conv-H1} holds with $\widetilde\alpha_m=\alpha_{\widetilde\ell,m}$ and $\widetilde\alpha_m'=\alpha_{\widetilde\ell,m}'$ for any $1\le m\le M_{\widetilde \ell}$. Again we have that $\lambda^{-\sigma_k^+}\omega(\lambda z)$ and
$\lambda^{-\sigma_k^+}\Delta_b \omega(\lambda z)$ remain uniformly
bounded in $H^1(B_1^+;t^b)$ as $\lambda\to 0^+$. Since $\sigma_k^+>\sigma_\ell^++2$ and since $\alpha_{\ell,m}'\neq 0$ for at least one $1\le m\le M_\ell$, the \eqref{eq:est-aux} follows as well with $\delta_1=\sigma_\ell^++2$ and
$$
\Psi_1(\theta)=\sum_{m=1}^{M_\ell} \frac{\alpha_{\ell,m}'}{K(N,b,\ell)} \, Y_{\ell,m}(\theta) \, , \qquad \theta\in \SF\, .
$$

\medskip

{\it Proof of \eqref{eq:est-aux-bis}.} If at least one of the numbers $\alpha_{\ell,1}',\dots,\alpha_{\ell,M_\ell}'$
introduced in Lemma \ref{t:5.5} is different from zero then the proof of \eqref{eq:est-aux-bis} follows immediately with $\delta_2=\sigma_\ell^+$ and
\begin{equation*}
\Psi_2(\theta)=\sum_{m=1}^{M_\ell} \alpha_{\ell,m}' Y_{\ell,m}(\theta) \, .
\end{equation*}
Suppose now that $\alpha_{\ell,1}'=\, \dots \,
=\alpha_{\ell,M_\ell}'=0$. Let $k>\ell$ be the first integer for which at least one of the numbers
$\alpha'_{k,1},\dots,\alpha'_{k,M_k}$ is different from
zero (such $k$ exists if $V\not\equiv 0$ in view of
  \eqref{eq:8}, Lemma \ref{l:5.3}, and Remark \ref{rem:tutti_k}) and put
\begin{align*}
\omega(z)& :=U(z)-\sum_{j=1}^{k}
\sum_{m=1}^{M_j} \varphi_{j,m}(|z|) Y_{j,m}\left(\tfrac
z{|z|}\right) \\
& =U(z)-\sum_{j=\ell}^k
\sum_{m=1}^{M_j} \alpha_{j,m}|z|^{\sigma_j^+} \, Y_{j,m}\left(\tfrac
z{|z|}\right)-\sum_{m=1}^{M_k} \frac{\alpha_{k,m}'}{K(N,b,k)}\, |z|^{\sigma_k^++2}\, Y_{k,m}\left(\tfrac
z{|z|}\right)
\end{align*}
for any $z\in B_R^+$. The last identity follows from the second
part of Lemma \ref{l:5.3} and Remark \ref{rem:tutti_k}. Applying Lemma \ref{t:5.5} to $\omega$ and proceeding as in the proof of \eqref{eq:est-aux}, one can show that
$\lambda^{-\sigma_k^+} \omega(\lambda z)\to 0$ and $\lambda^{-\sigma_k^+} \Delta_b \omega(\lambda z)\to 0$ in $H^1(B_1^+;t^B)$ as $\lambda\to 0^+$. The proof of
\eqref{eq:est-aux-bis} now follows with $\delta_2=\sigma_k^+$ and
\begin{equation*}
\Psi_2(\theta)=\sum_{m=1}^{M_k}\alpha_{k,m}' \, Y_{k,m}(\theta)
\end{equation*}
being $\Delta_b \omega(z)=V(z)- |z|^{\sigma_k^+}\sum_{m=1}^{M_k}\alpha_{k,m}' \, Y_{k,m}\left(\tfrac
z{|z|}\right)$.
\end{proof}

\section{Proof of the main results} \label{s:main-results}

We start with the proof of Theorem \ref{t:asym-est} since the proofs of Theorems \ref{t:SUCP-1}--\ref{t:SUCP-2} are related to the asymptotic estimates stated in Theorem \ref{t:asym-est}.

\subsection{Proof of Theorem \ref{t:asym-est}} Up to translation it is not restrictive to assume that $x_0=0$. The proof now follows from Theorem \ref{p:5.6} and the regularity estimates of Proposition \ref{p:6.4}.

\smallskip
Once we have proved Theorem \ref{t:asym-est}, we can proceed with the proofs of Theorems \ref{t:SUCP-1}--\ref{t:SUCP-2}.

\subsection{Proof of Theorem \ref{t:SUCP-1}} Let $u$ be as in the statement of the theorem and let $U\in \mathcal D_b$ be the corresponding solution
of \eqref{eq:aux-prob}. According with Section \ref{s:Almgren} we also put $V=\Delta_b U$. Following the argument introduced at the beginning of Section \ref{s:Almgren}, by assuming up to translation that $x_0=0$, we see that the couple $(U,V)\in H^1(B_R^+;t^b)\times H^1(B_R^+;t^b)$ is a nontrivial solution
of \eqref{eq:system} with $R$ as in \eqref{eq:H^1}. Since $(-\Delta)^s u\in (\mathcal D^{s-1,2}(\R^N))^\star$, by \eqref{eq:C-b} we deduce that the map
\begin{equation*}
W\mapsto \phantom{a}_{(\mathcal D^{s-1,2}(\R^N))^\star}
\left\langle (-\Delta)^{s} u, \mathop{\rm Tr}(W)\right\rangle_{\mathcal D^{s-1,2}(\R^N)} \, , \qquad W\in \mathcal D^{1,2}(\R^{N+1}_+;t^b)
\end{equation*}
belongs to $(\mathcal D^{1,2}(\R^{N+1}_+;t^b))^\star$.

Then, by classical minimization methods, we have that the minimum
\begin{equation*}
\min_{W\in \mathcal D^{1,2}(\R^{N+1}_+;t^b)}
\left[\frac 12 \int_{\R^{N+1}_+} t^b|\nabla W|^2\, dz+C_b^2 \phantom{a}_{(\mathcal D^{s-1,2}(\R^N))^\star}
\left\langle (-\Delta)^{s} u, \mathop{\rm Tr}(W)\right\rangle_{\mathcal D^{s-1,2}(\R^N)}\right]
\end{equation*}
is attained by some $\widetilde V\in  \mathcal D^{1,2}(\R^{N+1}_+;t^b)$ weakly solving
\begin{align}\label{eq:50-duplicata}
-\int_{\R^{N+1}_+} t^b \nabla\widetilde V \nabla\Phi \, dz &=C_b^2
\phantom{a}_{(\mathcal D^{s-1,2}(\R^N))^\star}
\left\langle
  (-\Delta)^{s} u,
\mathop{\rm Tr}(\Phi)\right\rangle_{\mathcal D^{s-1,2}(\R^N)}
\end{align}
for any $\Phi\in \mathcal D^{1,2}(\R^{N+1};t^b)$. In particular we have
\begin{align}\label{eq:50-duplicata-bis}
-\int_{\R^{N+1}_+} t^b \nabla\widetilde V \nabla\Phi \, dz
=C_b^2 \int_{\R^N}|\xi|^{2s} \widehat u\,\overline{\widehat{\mathop{\rm Tr}(\Phi)}}\,d\xi
\qquad \text{for any } \Phi\in C^\infty_c(\overline{\R^{N+1}_+}) \, .
\end{align}
Combining \eqref{eq:50-duplicata-bis} and \eqref{eq:nuova} we obtain
\begin{align} \label{eq:non-reg}
-\int_{\R^{N+1}_+} & t^b \nabla\widetilde V \nabla\Phi \, dz=C_b^2 (u,\mathop{\rm Tr}(\Phi))_{\mathcal D^{s,2}(\R^N)}\\
\notag & =(U,\Phi)_{\mathcal D_b}=\int_{\R^{N+1}_+} t^b\Delta_b U\Delta_b \Phi \, dz=\int_{\R^{N+1}_+} t^b V\Delta_b \Phi \, dz
\quad \text{for any } \Phi\in \mathcal T
\end{align}
with $\mathcal T$ as in \eqref{eq:space}.

Since $u\in\mathcal D^{s,2}(\R^N)$ and $(-\Delta)^s u\in (\mathcal
D^{s-1,2}(\R^N))^\star$, with a mollification argument, it is possible
to construct an approximating sequence of functions $\{u_n\}\subset \mathcal D^{s,2}(\R^N)$ such that $u_n\to u$ in $\mathcal D^{s,2}(\R^N)$, $(-\Delta)^s u_n\in C^\infty(\R^N)$, $(-\Delta)^s u_n\to (-\Delta)^s u$ weakly in $(\mathcal D^{s-1,2}(\R^N))^\star$.

Then we can construct the corresponding functions $U_n, V_n$ and $\widetilde V_n$. First we observe that $U_n\to U$ in $\mathcal D_b$ and in particular $V_n\to V$ in $L^2(\R^{N+1}_+;t^b)$. Moreover $\widetilde V_n\rightharpoonup \widetilde V$ weakly in $\mathcal D^{1,2}(\R^{N+1};t^b)$.

Now we observe that for the functions $V_n$ we have
\begin{align} \label{eq:strong-conv-1}
\int_{\R^{N+1}_+} t^b V_n\, \Delta_b \Phi \, dz=C_b^2 (u_n,\mathop{\rm Tr}(\Phi))_{\mathcal D^{s,2}(\R^N)}
=C_b^2 \int_{\R^N}  (-\Delta)^s u_n \mathop{\rm Tr}(\Phi) \, dx \quad \text{for any } \Phi\in \mathcal T,
\end{align}
and hence, since
$(-\Delta)^s u_n\in (\mathcal D^{s-1,2}(\R^N))^\star$, by Proposition
\ref{p:A2} one can show that, for any $r>0$, $V_n\in H^1(Q_r^+;t^b)$.

Combining \eqref{eq:strong-conv-1} with \eqref{eq:6} we obtain
\begin{align*} 
\int_{\R^{N+1}_+} t^b (V_n-V)\, \Delta_b \Phi \, dz=C_b^2  \phantom{a}_{(\mathcal D^{s-1,2}(\R^N))^\star}\left\langle (-\Delta)^{s} u_n-(-\Delta)^s u, \mathop{\rm Tr}(\Phi)\right\rangle_{\mathcal D^{s-1,2}(\R^N)}
\end{align*}
for any $\Phi\in \mathcal T$ such that ${\rm supp}(\Phi(\cdot,0))\subset \Omega$. Hence, by \eqref{eq:cont-dep} we deduce that $V_n \rightharpoonup V$ weakly in $H^1(Q_R^+;t^b)$ and by Lemma \ref{l:sobolev} we also have
\begin{equation} \label{eq:Traccia}
\mathop{\rm Tr}(V_n) \rightharpoonup \mathop{\rm Tr}(V) \qquad \text{weakly in } L^{2^*(N,s-1)}(B_R') \, .
\end{equation}
The fact that $V_n\in H^1(Q_r^+;t^b)$ implies
\begin{align*}
\int_{\R^{N+1}_+} t^b V_n\, \Delta_b \Phi \, dz=-\int_{\R^{N+1}_+} t^b \nabla V_n \nabla \Phi \, dz \qquad \text{for any } \Phi\in \mathcal T
\end{align*}
and by \eqref{eq:non-reg} applied to $V_n$ and $\widetilde V_n$ we obtain
\begin{equation} \label{eq:harmonic}
\int_{\R^{N+1}_+} t^b \nabla(V_n-\widetilde V_n)\nabla \Phi\, dz=0 \qquad \text{for any } \Phi\in \mathcal T \, .
\end{equation}
Actually we can prove that \eqref{eq:harmonic} still holds true for any $\Phi\in
C^\infty_c(\overline{\R^{N+1}_+})$ not necessarily satisfying
$\Phi_t(\cdot,0)\equiv 0$ in $\R^N\times\{0\}$, arguing as we did for \eqref{eq:10-bis}.
If we define
\begin{equation*}
\widetilde W_n(x,t)=
\begin{cases}
V_n(x,t)-\widetilde V_n(x,t) & \qquad \text{if } t\ge 0\, , \\[7pt]
V_n(x,-t)-\widetilde V_n(x,-t) & \qquad \text{if } t<0 \, ,
\end{cases}
\end{equation*}
by \eqref{eq:harmonic} we obtain
\begin{equation} \label{eq:harmonic-bis}
\int_{\R^{N+1}} |t|^b \nabla\widetilde W_n \nabla \Phi\, dz=0
\end{equation}
for any $\Phi\in C^\infty_c(\R^{N+1})$. Choosing a suitable sequence of test functions in \eqref{eq:harmonic-bis} and passing to the limit, it is possible to prove that for any $x_0\in \R^N$ and $r>0$
\begin{equation*}
\int_{\partial B_r(x_0,0)} |t|^b \frac{\partial \widetilde W_n}{\partial \nu} \, dS=0 \, .
\end{equation*}
From this identity, proceeding similarly to the proof of the mean value theorem for harmonic functions, see \cite[Theorem 2.1]{GT}), and taking into account the H\"older regularity results stated in Proposition \ref{p:reg-1}, one can prove that
\begin{equation*}
\widetilde W_n(x_0,0)=\frac 1{\omega_{N,b} \, r^{N+b+1}} \int_{B_r(x_0,0)} |t|^b \, \widetilde W_n \, dz \qquad \text{for any } x_0\in \R^N \ \text{and} \ r>0
\end{equation*}
where $\omega_{N,b}=(N+b+1)^{-1} \int_{\partial B_1(0,0)} |t|^b \,
dS$, see also \cite[Lemma A.1]{wang_wei} and \cite[Lemma 2.6]{STT}.
Hence we have
\begin{align*}
& |\widetilde W_n(x_0,0)|\le \frac 2{\omega_{N,b} \, r^{N+b+1}} \left(\int_{B_r^+(x_0)} |t|^b |V_n| \, dz+\int_{B_r^+(x_0)} |t|^b |\widetilde V_n| \, dz\right)\\
& \le \!  \frac 2{\omega_{N,b} \, r^{N+b+1}} \!\!\left[r^{\frac{N+b+1}2} \! \!\left(\tfrac{|B_1'|}{b+1}\right)^{\!\!\frac 12} \! \|V_n\|_{L^2(\R^{N+1}_+;t^b)}
\! +\! r^{\frac{(N+b+1)(2^{**}(b)-1)}{2^{**}(b)}} \! \! \left(\tfrac{|B_1'|}{b+1}\right)^{\frac{2^{**}(b)-1}{2^{**}(b)}}\!\! \|\widetilde V_n\|_{L^{2^{**}(b)}(\R^{N+1}_+;t^b)}\right]\!\! .
\end{align*}
Letting $r\to +\infty$, we have that the right hand side of the previous inequality tends to zero, from which we deduce that $\widetilde W_n\equiv 0$
on $\R^N\times \{0\}$ and in particular that $V_n\equiv \widetilde V_n$ on $\R^N\times \{0\}$. But from the fact that $\widetilde V_n \rightharpoonup \widetilde V$ weakly in $\mathcal D^{1,2}(\R^{N+1}_+;t^b)$ and \eqref{eq:Sobolev-frazionario-bis} we have that $\mathop{\rm Tr}(\widetilde V_n) \rightharpoonup \mathop{\rm Tr}(\widetilde V)$ weakly in $L^{2^*(N,s-1)}(\R^N)$. Combining this with \eqref{eq:Traccia} we deduce that $\mathop{\rm Tr}(V)=\mathop{\rm Tr}(\widetilde V)$ on $B_R'$.

Letting $\widetilde v:=\mathop{\rm Tr}(\widetilde V)$, by \cite{BCdPS,CS} and \eqref{eq:50-duplicata-bis} we deduce that there exists a positive constant $\kappa_{N,b}$ depending only on $N$ and $b$ such that
\begin{equation*}
-(\widetilde v,\varphi)_{\mathcal D^{s-1,2}(\R^N)}=\kappa_{N,b} (u,\varphi)_{\mathcal D^{s,2}(\R^N)} \qquad \text{for any } \varphi\in C^\infty_c(\R^N)
\end{equation*}
which means that $\widehat{\widetilde v}(\xi)=-\kappa_{N,b}|\xi|^2 \, \widehat u(\xi)$ in $\R^N$ and hence $\widetilde v=\kappa_{N,b} \Delta u$ in $\R^N$.

Finally we have that
$\mathop{\rm Tr}(V)=\widetilde v=\kappa_{N,b} \Delta u$ in $B_R'$. In
the rest of the proof we denote by $v$ the trace of $V$ on $B_R'$.

Let us assume, by contradiction, that $u\not\equiv0$. Then the couple $(U,V)\neq(0,0)$
is a weak solution to \eqref{eq:system} in $H^1(B_R^+;t^b)\times H^1(B_R^+;t^b)$ for some $R>0$.

From Lemma \ref{t:5.5} and the fact that any eigenfunction of \eqref{eq:Eigen-Weight} cannot vanish on $\partial\SF$, as observed in Remark \ref{rem:nuovo}, it follows that either $u$ or $v$ (which are the traces of $U$ and $V$
respectively)  vanish of some order $\gamma\ge 0$ at $0$. Since by assumption, $u$ satisfies
\begin{equation} \label{eq:infinite-order}
u(x)=O(|x|^k) \qquad \text{as } x\to 0 \quad \text{for any } k\in \N\, ,
\end{equation}
we have that necessarily $V$ vanishes of order $\gamma$, i.e. 
there exists $\Psi:{\mathbb S}^{N}_+\to\R$, eigenfunction of
\eqref{eq:Eigen-Weight}, such that
\[
\lambda^{-\gamma}V(\lambda z)\to
|z|^{\gamma}\Psi\Big(\frac
z{|z|}\Big) \text{ as }\lambda\to 0\text{ strongly in $H^1(B_1^+;t^b)$} \, .
\]
In particular by \eqref{eq:Sobolev-traccia} we also have
\[
\lambda^{-\gamma}v(\lambda x)\to
|x|^{\gamma}\Psi\Big(\frac
x{|x|},0\Big) \text{ as }\lambda\to 0 \text{ strongly in $L^{2^*(N,s-1)}(B_1')$  }.
\]
Let us denote
\[
v_\lambda(x)=\lambda^{-\gamma}v(\lambda x)\quad\text{and}\quad
\widetilde u_\lambda(x)=\lambda^{-2-\gamma}u(\lambda x),
\]
so that
\begin{equation}\label{eq:55}
v_\lambda\to |x|^{\gamma}\Psi\Big(\frac
x{|x|},0\Big) \text{ as }\lambda\to 0\text{ strongly in
  $L^{2^*(N,s-1)}(B_1')$}
\end{equation}
and
\[
\kappa_{N,b} \Delta \widetilde u_\lambda=v_\lambda\quad\text{in
} B_{R/\lambda}'.
\]
For every $\varphi\in C^{\infty}_{\rm c}(B_1')$ we have that, for $\lambda$ small enough,
\begin{equation}\label{eq:54}
-\kappa_{N,b} \int_{\R^N} \widetilde u_\lambda(-\Delta\varphi)\,dx=-\kappa_{N,b}
\int_{\R^N} \varphi(-\Delta \widetilde u_\lambda)\,dx
=\int_{\R^N} \varphi v_\lambda \,dx.
\end{equation}
From one hand, assumption \eqref{eq:infinite-order} implies that
\[
\lim_{\lambda\to0^+}\int_{\R^N} \widetilde u_\lambda(-\Delta\varphi)\,dx
=0
\]
whereas convergence \eqref{eq:55} yields
\[
\lim_{\lambda\to0^+}\int_{\R^N} \varphi v_\lambda \,dx=
\int_{\R^N} |x|^{\gamma}\Psi\Big(\frac
x{|x|},0\Big) \varphi(x)\,dx.
\]
Hence passing to the limit in \eqref{eq:54} we obtain that
\[
\int_{\R^N} |x|^{\gamma}\Psi\Big(\frac
x{|x|},0\Big) \varphi(x)\,dx=0\quad\text{for every }\varphi\in
C^{\infty}_{\rm c}(B_1'),
\]
thus contradicting the fact that $|x|^{\gamma}\Psi\Big(\frac
x{|x|},0\Big) \not\equiv0$.

\subsection{Proof of Theorem \ref{t:SUCP-2}} Let us assume by contradiction, that $u\not\equiv0$ in
$\Omega$ and $u(x)=0$ a.e. in a measurable set $E\subset \Omega$ of positive measure.

Let $U$ and $V$ be defined as in the proof of Theorem
\ref{t:SUCP-1}. As we explained in the proof of Theorem
\ref{t:SUCP-1}, for any $x\in \Omega$ we have that
$(U,V)\in H^1(B_R^+(x);t^b)\times H^1(B_R^+(x);t^B)$ for any $R>0$ as
in \eqref{eq:H^1}.

Hence, by Lebesgue's density Theorem (i.e. almost every point of $E$ is a density point of $E$), there exists a point $y_0\in E$ and $R>0$ such that
$B_{2R}'(y_0)\subset \Omega$, $|B_R'(y_0)\cap E|_N>0$ and $(U,V)\in H^1(B_R^+(y_0);t^b)\times H^1(B_R^+(y_0);t^b)$ where $|\cdot|_N$ denotes the $N$-dimensional Lebesgue measure. With choice of $y_0$ and $R>0$, proceeding as in the proof of Theorem \ref{t:SUCP-1}, we deduce that $v=\kappa_{N,b} \Delta u$ in $B_R'(y_0)$ with $v=\mathop{\rm Tr}(V)$.

Since $\kappa_{N,b}\Delta u=v$ and by Lemma \ref{l:sobolev} $v\in L^{2^*(N,s-1)}(B_R'(y_0))$, by classical regularity theory we have that
$u\in H^2_{\rm loc}(B_R'(y_0))$. Since $u(x)=0$ for any $x\in E$, we
have that $\nabla u(x)=0$ for a.e. $x\in E\cap B_R'(y_0)$ and hence, since
$\frac{\partial u}{\partial x_i}\in  H^1_{\rm loc}(B_R'(y_0))$ for every
$i$, $\Delta u=0$ a.e. in $E\cap B_R'(y_0)$. In particular $u(x)=v(x)=0$ for a.e. $x\in E':=E\cap B_R'(y_0)$.

Let  $x_0$ be a density point of $E'$. Hence, for all $\e>0$ there exists $r_0=r_0(\e)\in(0,1)$
such that, for all $r\in(0,r_0)$,
\begin{equation}\label{eq:56}
\frac{|(\R^N\setminus E')\cap B'_r(x_0)|_N}{|B'_r(x_0)|_N}<\e \, .
\end{equation}
Lemma \ref{t:5.5} implies that there exist $\gamma\ge 0$,
$\Psi_1,\Psi_2:{\mathbb S}^{N}_+\to\R$ solving \eqref{eq:Eigen-Weight}
such that either $\Psi_1\not\equiv0$ or $\Psi_2\not\equiv0$ (and hence
$\Psi_1\not\equiv0$ or $\Psi_2\not\equiv0$ on $\partial\SF$
respectively as observed in Remark \ref{rem:nuovo}), and
\begin{equation}\label{eq:57}
\lambda^{-\gamma} u(x_0+\lambda (x-x_0) )\to
|x-x_0|^{\gamma}\Psi_1\Big(\frac{x-x_0}{|x-x_0|},0\Big)
\end{equation}
and
\begin{equation}\label{eq:58}
\lambda^{-\gamma}v(x_0+\lambda (x-x_0) )\to
|x-x_0|^{\gamma}\Psi_2\Big(\frac{x-x_0}{|x-x_0|},0\Big)
\end{equation}
as $\lambda\to 0$ strongly in $L^{2^*(N,s-1)}(B_1'(x_0))$.

Since $u\equiv v\equiv 0$ a.e. in $E'$,
 by \eqref{eq:56} we have
\begin{align*}
\int_{B'_r(x_0)}u^2(x)\,dx&=\int_{(\R^N\setminus E')\cap
  B'_r(x_0)}u^2(x)\,dx\\
&\leq \bigg(\int_{(\R^N\setminus E')\cap
  B'_r(x_0)}|u(x)|^{2^*(N,s-1)}dx\bigg)^{\frac 2{2^*(N,s-1)}}|(\R^N\setminus E') \cap
B'_r(x_0)|_N^{\frac{2^*(N,s-1)-2}{2^*(N,s-1)}}\\
&<\e^{\frac{2^*(N,s-1)-2}{2^*(N,s-1)}}|B'_r(x_0)|_N^{\frac{2^*(N,s-1)-2}{2^*(N,s-1)}}
\bigg(\int_{(\R^N\setminus E')\cap
  B'_r(x_0)}|u(x)|^{2^*(N,s-1)}dx\bigg)^{\frac 2{2^*(N,s-1)}}
\end{align*}
and similarly
\begin{align*}
\int_{B'_r(x_0)}v^2(x)\,dx<\e^{\frac{2^*(N,s-1)-2}{2^*(N,s-1)}}|B'_r(x_0)|_N^{\frac{2^*(N,s-1)-2}{2^*(N,s-1)}}
\bigg(\int_{(\R^N\setminus E')\cap
  B'_r(x_0)}|v(x)|^{2^*(N,s-1)}dx\bigg)^{\frac 2{2^*(N,s-1)}}
\end{align*}
for all $r\in(0,r_0)$. Then, letting
$u^r(x):=r^{-\gamma}u(x_0+r(x-x_0))$ and $v^r(x):=r^{-\gamma}v(x_0+r(x-x_0))$,
 \begin{align*}
& \int_{B'_1(x_0)}|u^r(x)|^2dx<\left(\tfrac{\omega_{N-1}}N\right)^{\frac{2^*(N,s-1)-2}{2^*(N,s-1)}}
\e^{\frac{2^*(N,s-1)-2}{2^*(N,s-1)}}\bigg(\int_{B'_1(x_0)}|u^r(x)|^{2^*(N,s-1)}dx\bigg)^{\frac 2{2^*(N,s-1)}} \, , \\
& \int_{B'_1(x_0)}|u^r(x)|^2dx<\left(\tfrac{\omega_{N-1}}N\right)^{\frac{2^*(N,s-1)-2}{2^*(N,s-1)}}
\e^{\frac{2^*(N,s-1)-2}{2^*(N,s-1)}}\bigg(\int_{B'_1(x_0)}|u^r(x)|^{2^*(N,s-1)}dx\bigg)^{\frac 2{2^*(N,s-1)}} \, ,
\end{align*}
for all $r\in(0,r_0)$, where $\omega_{N-1}=\int_{{\mathbb S}^{N-1}}1\,dS'$. Letting $r\to 0^+$, from \eqref{eq:57} and \eqref{eq:58}  we have that
 \begin{multline*}
\int_{B'_1(x_0)}
|x-x_0|^{2\gamma}\Psi_i^2\Big(\tfrac{x-x_0}{|x-x_0|},0\Big)\,dx\\
\leq
\left(\tfrac{\omega_{N-1}}N\right)^{\frac{2^*(N,s-1)-2}{2^*(N,s-1)}}
\e^{\frac{2^*(N,s-1)-2}{2^*(N,s-1)}} \bigg(\int_{B'_1(x_0)}|x-x_0|^{\gamma\cdot 2^*(N,s-1)}\left|
\Psi_i\Big(\tfrac{x-x_0}{|x-x_0|},0\Big)\right|^{2^*(N,s-1)}dx\bigg)^{\frac 2{2^*(N,s-1)}}
\end{multline*}
for $i=1,2$ which yields a contradiction as $\e\to 0^+$, since
either $\Psi_1\not\equiv0$ or $\Psi_2\not\equiv0$ on $\partial\SF$.

\section{Appendix} \label{s:Appendix}

\subsection{Inequalities involving weighted Sobolev spaces}
\label{ss:Sobolev}

Throughout this section, we will assume that $s\in (1,2)$, $N>2s$
and $b=3-2s\in (-1,1)$. For simplicity, the center $x_0$ of the
sets introduced in \eqref{eq:sets} will be omitted whenever
$x_0=0$.

Next we state the following Hardy-Sobolev inequality taken from
\cite[Lemma 2.4]{FF}. For any $R>0$ and $U\in H^1(B_R^+;t^b)$ we
have
\begin{align*} 
\left(\frac{N+b-1}2\right)^2 \int_{B_R^+} t^b \frac{U^2}{|z|^2}\,
dz\le \int_{B_R^+} t^b |\nabla U|^2 dz+\frac{N+b-1}{2R}
\int_{S_R^+} t^b U^2 dS.
\end{align*}
In particular, for any $x_0\in \R^N$ and $U\in
H^1(B_R^+(x_0);t^b)$, we have
\begin{align} \label{eq:PA2-2}
\left(\frac{N+b-1}{2R} \right)^2 \int_{B_R^+(x_0)} t^b U^2 \,
dz\le \int_{B_R^+(x_0)} t^b |\nabla U|^2 dz+\frac{N+b-1}{2R}
\int_{S_R^+(x_0)} t^b U^2 dS.
\end{align}
Now we state a Sobolev inequality involving a suitable critical
Sobolev exponent. Let
\begin{equation*}
2^{**}(b)=
\begin{cases}
\frac{2(N+b+1)}{N+b-1} & \qquad \text{if } 0<b<1 \, , \\[7pt]
\frac{2(N+1)}{N-1} & \qquad \text{if } -1<b\le 0 \, .
\end{cases}
\end{equation*}
By \cite[Theorem 19.10]{OpicKufner} we have
\begin{equation} \label{eq:Opic-Kufner}
\overline S(N,b) \left(\int_{B_1^+} t^b |U|^{2^{**}(b)} dz
\right)^{\!\!\frac2{2^{**}(b)}} \le \int_{B_1^+} t^b |\nabla U|^2
dz+\int_{B_1^+} t^b U^2 dz \quad \text{for any } U\in
H^1(B_1^+;t^b) ,
\end{equation}
for some constant $\overline S(N,b)$ depending only on $N$ and
$b$. The corresponding inequality in the half ball $B_R^+(x_0)$
can be obtained by \eqref{eq:Opic-Kufner} after scaling and
translation.

Next we show that the embedding $H^1_0(\Gamma_R^+(x_0);t^b)\subset
L^2(Q_R^+(x_0);t^b)$ is compact.

\begin{Proposition} \label{p:compactness-2}
Let $x_0\in\R^N$, $b\in (-1,1)$ and $R>0$. Then the embedding
\[
H^1_0(\Gamma_R^+(x_0);t^b)\subset L^2(Q_R^+(x_0);t^b)
\]
is
compact.
\end{Proposition}

\begin{proof} Let us define the function
$d:Q_{3R}^+(x_0)\to [0,\infty)$ where
$$
d(z):={\rm dist}(z,\partial Q_{3R}^+(x_0)) \qquad \text{for
any } z\in Q_{3R}^+(x_0) \, .
$$
We immediately see that if $z=(x,t)\in Q_R^+(x_0)$ then $d(x,t)=t$.
Let $\{U_n\}\subset H^1_0(\Gamma_R^+(x_0);t^b)$ be a sequence
bounded in $H^1_0(\Gamma_R^+(x_0);t^b)$. For any $n$ let us still
denote by $U_n$ the trivial extension to $Q_{3R}^+(x_0)$ so that
$U_n\in H^1_0(\Gamma_{3R}^+(x_0);t^b)$. We observe that
\begin{align*}
& \int_{Q_{3R}^+(x_0)} (d(z))^b \, |\nabla U_n|^2 dz
=\int_{Q_{R}^+(x_0)} (d(z))^b \, |\nabla U_n|^2 dz
=\int_{Q_{R}^+(x_0)} t^b \, |\nabla U_n|^2 dz, \\
& \int_{Q_{3R}^+(x_0)} (d(z))^b \, U_n^2 \, dz
=\int_{Q_{R}^+(x_0)} (d(z))^b \, U_n^2\, dz
=\int_{Q_{R}^+(x_0)} t^b U_n^2\, dz,
\end{align*}
thus showing that $\{U_n\}$ is bounded in the weighted Sobolev
space $W^{1,2}(Q_{3R}^+(x_0);d^b,d^b)$ where we used the notation
of \cite[Theorem 19.7]{OpicKufner}. By the same theorem in
\cite{OpicKufner} we deduce that $\{U_n\}$ is, up to subsequences,
strongly convergent in $L^2(Q_{3R}^+(x_0);d^b)$. But the functions
$U_n$ are supported in $Q_R^+(x_0)$ so that $\{U_n\}$ is strongly
convergent in $L^2(Q_R^+(x_0);t^b)$. This completes the proof of
the proposition.
\end{proof}

\noindent Now we state a Hardy-Rellich type inequality for functions in
$\mathcal D_b$.

\begin{Proposition}\label{t:rellich}
For every $U\in \mathcal D_b$, we have that $\frac{U}{|z|^2}\in
L^2(\R^{N+1}_+;t^b)$ and $\frac{\nabla U}{|z|}\in
L^2(\R^{N+1}_+;t^b)$. Furthermore
\begin{equation}\label{eq:rellich}
(N-2s)^2\int_{\R^{N+1}_+}t^b\frac{U^2}{|z|^4}\, dz+
2(N-2s)\int_{\R^{N+1}_+}t^b\frac{|\nabla U|^2}{|z|^2}\, dz
\leq \int_{\R^{N+1}_+}t^b |\Delta_b U|^2\, dz
\end{equation}
  for every $U\in \mathcal D_b$.
\end{Proposition}

\begin{proof}
  By definition of $\mathcal D_b$, it is enough to prove inequality
  \eqref{eq:rellich} for every $U\in
  C^\infty_c(\overline{\R^{N+1}_+})$ such that $U_t\equiv 0$ on
$\R^N\times\{0\}$. Arguing as in \cite{PPL}, we have that, for every
$\e>0$ and $\lambda\in\R$,
\begin{align*}
  0&\leq \left\|t^{b/2}\frac{z}{|z|}\Delta_b U+\lambda t^{b/2}
     U\frac{z}{|z|^3}\right\|^2_{L^2(\R^{N+1}_+\setminus
B_\e,\R^{N+1})}\\
&=\int_{\R^{N+1}_+\setminus
B_\e}t^b |\Delta_b U|^2\,dz+\lambda^2\int_{\R^{N+1}_+\setminus
B_\e}t^b\frac{U^2(z)}{|z|^4}\,dz+2\lambda
\int_{\R^{N+1}_+\setminus
B_\e}t^b\frac{U \Delta_b U}{|z|^2}\,dz,
\end{align*}
where $z=(x,t)$ and $B_\e=\{z\in\R^{N+1}:|z|<\e\}$. Integration by
parts yields
\begin{align*}
  &\int_{\R^{N+1}_+\setminus
B_\e}t^b\frac{U \Delta_b
  U}{|z|^2}\,dz=\int_{\R^{N+1}_+\setminus
B_\e}\frac{U}{|z|^2}\dive (t^b\nabla
                 U)\,dz\\
&=-\int_{\{x\in\R^N:|x|>\e\}}\frac{U(x,0)}{|x|^2}\Big(\lim_{t\to 0^+}t^bU_t(x,t)\Big)\,dx
-\int_{\R^{N+1}_+\cap \partial B_\e}t^b\frac{U}{|z|^2}\nabla
  U(z)\cdot\frac{z}{|z|}\,dS\\
&\qquad-
\int_{\R^{N+1}_+\setminus
B_\e}t^b\nabla U\cdot\nabla\bigg(\frac{U}{|z|^2}\bigg)\,dz\\
&=0+O(\e^{b+N-2})-\int_{\R^{N+1}_+\setminus
B_\e}t^b\frac{|\nabla U|^2}{|z|^2}\,dz+\int_{\R^{N+1}_+\setminus
B_\e}t^b \frac{\nabla (U^2)\cdot z}{|z|^4}\,dz
\end{align*}
and
\begin{align*}
  \int_{\R^{N+1}_+\setminus
B_\e}t^b \frac{\nabla (U^2)\cdot z}{|z|^4}\,dz&=
-\int_{\R^{N+1}_+\cap \partial B_\e}t^b\frac{U^2}{|z|^3}\,dS-
\int_{\R^{N+1}_+\setminus
B_\e}U^2\dive\bigg(t^b\frac{z}{|z|^4}\bigg)\,dz\\
&=O(\e^{b+N-3})-(N+b-3)\int_{\R^{N+1}_+\setminus
B_\e}t^b \frac{U^2(z)}{|z|^4}\,dz.
\end{align*}
Combining the previous estimates we obtain that
\begin{align*}
  0&\leq \int_{\R^{N+1}_+\setminus
B_\e}t^b |\Delta_b U|^2\,dz+\lambda^2\int_{\R^{N+1}_+\setminus
B_\e}t^b\frac{U^2(z)}{|z|^4}\,dz\\
&\quad
-2\lambda \int_{\R^{N+1}_+\setminus
B_\e}t^b\frac{|\nabla U|^2}{|z|^2}\,dz-2\lambda (N-2s)\int_{\R^{N+1}_+\setminus
B_\e}t^b \frac{U^2(z)}{|z|^4}\,dz+O(\e^{N-2s}).
\end{align*}
Choosing $\lambda=N-2s$ and letting $\e\to0^+$ we obtain that
\begin{align*}
(N-2s)^2\int_{\R^{N+1}_+}t^b\frac{U^2(z)}{|z|^4}\,dz +
2(N-2s)\int_{\R^{N+1}_+}t^b\frac{|\nabla U|^2}{|z|^2}\,dz
 \leq \int_{\R^{N+1}_+}t^b |\Delta_b U|^2\,dz
\end{align*}
thus completing the proof.
\end{proof}

If $N>2\gamma$, the Sobolev embedding implies that there exists a
positive constant $S(N,\gamma)$ depending only on $N$ and
$\gamma$, such that
\begin{equation} \label{eq:Sobolev-frazionario}
S(N,\gamma) \|u\|_{L^{2^*(N,\gamma)}(\R^N)}^2\le \|u\|_{\mathcal
D^{\gamma,2}(\R^{N})}^2 \qquad \text{for any } u\in \mathcal
D^{\gamma,2}(\R^{N})
\end{equation}
where $2^*(N,\gamma)=2N/(N-2\gamma)$, see e.g. \cite{cotsiolis}.

According with \cite{BCdPS}, we define $\mathcal
D^{1,2}(\R^{N+1}_+;t^b)$ as the completion of the space
$C^\infty_c(\overline{\R^{N+1}_+})$ with respect to the norm
$$
\|U\|_{\mathcal D^{1,2}(\R^{N+1}_+;t^b)}:=\left(\int_{\R^{N+1}_+}
t^b|\nabla U|^2 dz\right)^{1/2} .
$$
Arguing as in \cite{BCdPS}, we have that there exists a constant
$K_b$ depending only on $b\in (-1,1)$ such that
\begin{equation} \label{eq:C-b}
K_b \|{\rm Tr\, }(U)\|_{\mathcal D^{s-1,2}(\R^N)}\le
\|U\|_{\mathcal D^{1,2}(\R^{N+1}_+;t^b)} \qquad \text{for any }
U\in \mathcal D^{1,2}(\R^{N+1}_+;t^b) \, .
\end{equation}
Combining this with \eqref{eq:Sobolev-frazionario}, we infer
\begin{equation} \label{eq:Sobolev-frazionario-bis}
S(N,s-1)K_b^2 \, \|{\rm Tr\, }(U)\|_{L^{2^*(N,s-1)}(\R^N)}^2\le
\|U\|_{\mathcal D^{1,2}(\R^{N+1}_+;t^b)}^2 \qquad \text{for any }
U\in \mathcal D^{1,2}(\R^{N+1}_+;t^b) \, .
\end{equation}

\begin{Lemma} \label{l:sobolev} For any $r>0$ and any $U\in H^1(B_r^+;t^b)$ we have
\begin{equation} \label{eq:Sobolev-traccia}
\widetilde S(N,b) \left(\int_{B_r'}
|u|^{2^*(N,s-1)}dx\right)^{\frac{2}{2^*(N,s-1)}} \le \int_{B_r^+}
t^b |\nabla U|^2 dz+\frac{N+b-1}{2r} \int_{S_r^+} t^b U^2 dS
\end{equation}
where $u={\rm Tr\, }(U)$ and $\widetilde S(N,b)$ is a positive
constant depending only on $N$ and $b$.
\end{Lemma}

\begin{proof} See the proof of \cite[Lemma 2.6]{FF}.
\end{proof}

\subsection{H\"older regularity of solutions}
This subsection is devoted to some results about H\"older
regularity of solutions to systems of weighted elliptic equations
in divergence form. Throughout this subsection, we will assume that
$s\in (1,2)$, $N>2s$ and $b=3-2s\in (-1,1)$. As in Subsection
\ref{ss:Sobolev} the center $x_0\in \R^N$ of the sets introduced
in \eqref{eq:sets} will be omitted whenever $x_0=0$.

We start with the following proposition which is a restatement, adapted
to our setting, of some regularity results contained in \cite{FF2},
see also \cite{JLX1}.

\begin{Proposition} \label{p:reg-1} (Propositions 3-4 in \cite{FF2}) Let $A,B\in L^{q_1}(B_1')$ for some $q_1>\frac N{1-b}$ and let $D\in L^{q_2}(B_1^+;t^b)$  for some $q_2>\frac{N+b+1}2$. Let $W\in H^1(B_1^+;t^b)$ be a weak solution of
\begin{equation} \label{eq:W}
\begin{cases}
-{\rm div}(t^b\nabla W)=t^b D(z)  & \qquad \text{in } B_1^+  \, , \\[4pt]
-\lim_{t\to 0^+} t^b W_t=A(x)W+B(x) & \qquad \text{on } B_1' \, .
\end{cases}
\end{equation}
Then the following statements hold true:

\begin{itemize}
\item[$(i)$] $W\in C^{0,\alpha}(\overline{B_{1/2}^+})$ and in addition
$$
\|W\|_{C^{0,\alpha}(\overline{B_{1/2}^+})}\le C\left(\|W\|_{L^2(B_1^+;t^b)}+\|B\|_{L^{q_1}(B_1')}+\|D\|_{L^{q_2}(B_1^+;t^b)}\right)
$$
for some $C>0$ and $\alpha\in (0,1)$ depending only on $N,b$ and $\|A\|_{L^{q_1}(B_1')}$;

\item[$(ii)$] if in addition to the previous assumptions we also suppose that $A,B\in W^{1,\infty}(B_1')$ and $D, \nabla_x D\in L^\infty(B_1^+)$ then we also have $\nabla_x W\in C^{0,\alpha}(\overline{B_{1/2}^+})$ and
\begin{align*}
& \|W\|_{C^{0,\alpha}(\overline{B_{1/2}^+})}+\|\nabla_x W\|_{C^{0,\alpha}(\overline{B_{1/2}^+})} \\
& \qquad \le C\left(\|W\|_{L^2(B_1^+;t^b)}+\|A\|_{W^{1,\infty}(B_1')}+\|B\|_{W^{1,\infty}(B_1')}+\|D\|_{L^\infty(B_1^+)}+\|\nabla_x D\|_{L^\infty(B_1^+)} \right)
\end{align*}
for some $C>0$ and $\alpha\in (0,1)$ depending only on $N,b$ and $\|A\|_{L^\infty(B_1')}$.
\end{itemize}
\end{Proposition}

In order to obtain a H\"older estimate for the $t$-derivative of a solution of \eqref{eq:W} we need to adapt to our context some results from \cite{CaSi, FF2, FF}.

\begin{Proposition} \label{p:reg-2} Let $t^b D_t\in L^\infty(B_1^+)$ and let $W\in H^1(B_1^+;t^b)$ be a weak solution of \eqref{eq:W} with $A\equiv 0$ and $B\equiv 0$.
Then $t^b W_t\in C^{0,\alpha}(\overline{B_{1/4}^+})$ and
\begin{equation*} 
\|t^b W_t\|_{C^{0,\alpha}(\overline{B_{1/4}^+})}\le C\left(\|W\|_{H^1(B_1^+;t^b)}+\|t^b D_t\|_{L^\infty(B_1^+)} \right)
\end{equation*}
for some $C>0$ and $\alpha\in (0,1)$ depending only on $N$ and $b$.
\end{Proposition}

\begin{proof} Since $W$ is a weak solution of the problem
\begin{equation*}
\begin{cases}
-{\rm div}(t^b\nabla W)=t^b D(z)  & \qquad \text{in } B_1^+  \, , \\[4pt]
\lim_{t\to 0^+} t^b W_t=0 & \qquad \text{on } B_1' \, ,
\end{cases}
\end{equation*}
it is clear that the even reflection of $W$ with respect to $t$, which we denote by $\widetilde W$, belongs to $H^1(B_1;|t|^b)$ and it is a weak solution of
\begin{equation*} 
-{\rm div}(|t|^b\nabla \widetilde W)=|t|^b \widetilde D(z)  \qquad \text{in } B_1  \, ,
\end{equation*}
where we denote by $\widetilde D$ the even reflection of $D$. In other words
\begin{equation} \label{eq:W-ter}
\int_{B_1} |t|^b \nabla \widetilde W\nabla \varphi \, dz=\int_{B_1} |t|^b \widetilde D(z)\varphi\, dz \qquad \text{for any } \varphi\in H^1_0(B_1;|t|^b) \, .
\end{equation}

Let now $\psi$ be a function in $C^\infty_c(B_1)$ such that
$\psi_t(x,0)=0$ for any $x\in B_1'$. Then the function
$\varphi(x,t)=|t|^{-b}\psi_t(x,t)$ belongs to $H^1(B_1;|t|^b)$.
Since ${\rm supp}(\varphi)\subset B_1$ then $\varphi\in
H^1_0(B_1;|t|^b)$ as one can deduce from \cite[Theorem
2.5]{kilpelainen} and a standard truncation argument.

With this particular choice of $\varphi$ in \eqref{eq:W-ter} we obtain
\begin{align*}
& \int_{B_1} \widetilde D(z) \psi_t(z)\, dz=\int_{B_1} \nabla \widetilde W\nabla (\psi_t)\, dz-\int_{B_1} \frac bt \, \widetilde W_t\psi_t \, dz
=-\int_{B_1} \widetilde W (\Delta \psi)_t\, dz-\int_{B_1} \frac bt \, \widetilde W_t\psi_t \, dz \\
& \qquad  =\int_{B_1} \widetilde W_t \left(\Delta \psi-\frac bt\,  \psi_t\right)\, dz=\int_{B_1} |t|^b \widetilde W_t \, {\rm div}(|t|^{-b}\nabla \psi)\, dz  \, .
\end{align*}
This proves that the function $\Psi(x,t):=|t|^b \widetilde W_t(x,t)\in L^2(B_1;|t|^{-b})$ satisfies
\begin{equation} \label{eq:var-Psi}
-\int_{B_1} \Psi \, {\rm div}(|t|^{-b}\nabla \psi)\, dz=\int_{B_1} \widetilde D_t(z)\psi \, dz
\end{equation}
for any $\psi\in C^\infty_c(B_1)$ such that $\psi_t(x,0)=0$ for all $x\in B_1'$.

By Proposition \ref{p:A2} we deduce that $\Psi\in H^1(B_{1/2};|t|^{-b})$ being $|t|^b \widetilde D_t\in L^2(B_1;|t|^{-b})$.

In particular, by \eqref{eq:var-Psi} we have that
\begin{equation} \label{eq:var-Psi-bis}
\int_{B_{1/2}} |t|^{-b} \nabla\Psi \nabla \psi \, dz=\int_{B_{1/2}} \widetilde D_t(z)\psi \, dz
\end{equation}
for any $\psi\in C^\infty_c(B_{1/2})$ such that $\psi_t(x,0)=0$ for all $x\in B_{1/2}'$.

In order to remove the condition $\psi_t(\cdot,0)\equiv 0$ on $B_{1/2}'$, it is enough to test \eqref{eq:var-Psi-bis} with
$$
\psi_k(x,t)=\psi(x,t)-\psi_t(x,0)\, t\, \eta(kt)\, , \quad k\in \N \, , \quad \text{for any } \psi\in C^\infty_c(B_{1/2}) \, ,
$$
where $\eta\in C^\infty_c(\R)$, $0\le \eta\le 1$, $\eta(t)=0$ for any
$t\in (-\infty,-2]\cup [2,+\infty)$ and $\eta(t)=1$ for any
$t\in [-1,1]$, and to pass to the limit as $k\to +\infty$.

In other words, we have shown that $\Psi \in H^1(B_{1/2};|t|^{-b})$ is a weak solution in the usual sense of the equation
$$
-{\rm div}(|t|^{-b}\nabla \Psi)=\widetilde D_t(z) \qquad \text{in } B_{1/2} \, .
$$

Since by assumption $t^b D_t\in L^\infty(B_1^+)$ then $|t|^b
\widetilde D_t \in L^\infty(B_1)$ and hence $\widetilde D_t
/|t|^{-b}\in L^p(B_{1/2};|t|^{-b})$ for any $1\le p<\infty$. In
particular $\widetilde D_t/|t|^{-b}\in M_\sigma(B_{1/2},|t|^{-b})$
for some $\sigma>0$ (see Definition 2.4 and Remark 2.6 in
\cite{Zamboni}). Recalling that the weight $|t|^{-b}$ belongs to
the Muckenhoupt class $A_2$, by Theorem 5.2 in \cite{Zamboni} we
deduce that $\Psi\in C^{0,\alpha}(\overline{B_{1/4}})$ for some
$\alpha\in (0,1)$ and there exists a constant $C>0$ such that
\begin{equation*}
\|\Psi\|_{C^{0,\alpha}(\overline{B_{1/4}})}\le C\left(\|\Psi\|_{L^2(B_{1/2};|t|^{-b})}+\|\, |t|^b \widetilde D_t\|_{L^\infty(B_{1/2})}\right)
\le 2C\left(\|W\|_{H^1(B_1^+;t^b)}+\|t^b D_t\|_{L^\infty(B_1^+)}\right) \, .
\end{equation*}
The proof of the theorem now follows from the definition of $\Psi$.
\end{proof}

In order to apply the last two propositions to system
\eqref{eq:system}, we prove the following Brezis-Kato type result for a system of two equations with
a potential in the boundary conditions and forcing terms both in the equation and in the boundary conditions.

\begin{Proposition} \label{p:Brezis-Kato} Let $A,B\in L^{\frac N{2(s-1)}}(B_1')$.
Suppose that $U,V\in H^1(B_1^+;t^b)$ weakly solve the system
\begin{equation} \label{eq:system-BK}
\begin{cases}
{\rm div}(t^b\nabla U)=t^bV & \qquad \text{in } B_1^+ \, , \\
{\rm div}(t^b\nabla V)=0 & \qquad \text{in } B_1^+ \, , \\
\lim_{t\to 0^+} t^b U_t=0 & \qquad \text{on } B_1' \, ,\\
-\lim_{t\to 0^+} t^b V_t=A(x)U+B(x) & \qquad \text{on } B_1' \, .
\end{cases}
\end{equation}
Then $U,V\in L^q(B_{1/2}^+;t^b)$, $U(\cdot,0),V(\cdot,0)\in L^q(B_{1/2}')$ for any $1\le
q<\infty$ and moreover there exists a constant $K_1$
depending only on $N,b,q,\|A\|_{L^{\frac N{2(s-1)}}(B_1')}$ and $\|B\|_{L^{\frac N{2(s-1)}}(B_1')}$ such that
\begin{align*}
& \|U\|_{L^q(B_{1/2}^+;t^b)} \le K_1 \Big(1+\|U\|_{L^{2^{**}(b)}(B_1^+;t^b)}+\|V\|_{L^{2^{**}(b)}(B_1^+;t^b)}\Big) \, ,  \\
& \|U(\cdot,0)\|_{L^q(B_{1/2}')} \le  K_1 \Big(1+\|U\|_{L^{2^{**}(b)}(B_1^+;t^b)}+\|V\|_{L^{2^{**}(b)}(B_1^+;t^b)}\Big) \, ,\\
& \|V\|_{L^q(B_{1/2}^+;t^b)} \le  K_1 \Big(1+\|U\|_{L^{2^{**}(b)}(B_1^+;t^b)}+\|V\|_{L^{2^{**}(b)}(B_1^+;t^b)}\Big) \, ,
\\
& \|V(\cdot,0)\|_{L^q(B_{1/2}')} \le  K_1 \Big(1+\|U\|_{L^{2^{**}(b)}(B_1^+;t^b)}+\|V\|_{L^{2^{**}(b)}(B_1^+;t^b)}\Big)  \, .
\end{align*}
\end{Proposition}

\begin{proof} The proof is quite standard and it is based on a Moser-Trudinger
iteration scheme inspired by the paper of Brezis-Kato \cite{BrezisKato}.

If we combine \eqref{eq:Opic-Kufner} with \eqref{eq:PA2-2} we
obtain
\begin{equation} \label{eq:Sobolev-pesata}
\overline C(N,b) \left(\int_{B_1^+} t^b |W|^{2^{**}(b)} dz
\right)^{2/2^{**}(b)} \le \int_{B_1^+} t^b |\nabla W|^2 dz \qquad
\text{for any } W\in H^1_0(\Sigma_1^+;t^b) \, ,
\end{equation}
where $\overline C(N,b)=\overline S(N,b)\cdot \left[1+\left(\frac
2{N+b-1}\right)^2\right]^{-1}$.

Let $\frac 12<r_U<1$ and let $\eta_U \in C^\infty_c(\overline{\R^{N+1}_+})$ be a cut-off function
such that ${\rm supp}(\eta_U)\subset \Sigma_1^+$ and $\eta_U \equiv 1$ in
$\Sigma_{r_U}^+$. For any $n\in \N$, set $U^n:=\min\{|U|,n\}$,
$V^n:=\min\{|V|,n\}$. Put $\alpha_0=2^{**}(b)$. Testing the first
equation in \eqref{eq:system-BK} with $\eta_U^2 (U^n)^{\alpha_0-2} U$
and exploiting the respective boundary condition, we obtain
\begin{align} \label{eq:IDE}
& \int_{B_1^+} t^b \nabla U
\nabla\left(\eta_U^2(U^n)^{\alpha_0-2}U\right)\, dz=-\int_{B_1^+} t^b
V \eta_U^2 (U^n)^{\alpha_0-2} U \, dz \, .
\end{align}
By direct computation (see the proof of Lemma 9.1 in \cite{FFT1}
for more details), one can verify that if we put $C(q)=\min\{\frac
14,\frac 4{q+4}\}$ for any $q\ge 1$, we have
\begin{align*} & C(\alpha_0)
\int_{B_1^+} t^b \left|\nabla\Big(\eta_U(U^n)^{\frac{\alpha_0-2}2} U
\Big)\right|^2 dz \\
&  \le \int_{B_1^+} t^b \nabla U \nabla \Big(\eta_U^2
(U^n)^{\alpha_0-2}U
\Big)dz+\left(2+C(\alpha_0)\frac{\alpha_0+2}2\right)\int_{B_1^+}
t^b (U^n)^{\alpha_0-2}|U|^2 |\nabla \eta_U|^2 dz \, .
\end{align*}
Combing this with \eqref{eq:IDE}, using Young inequality and the fact that $U^n\le |U|$, we obtain
\begin{align*}
& C(\alpha_0) \int_{B_1^+} t^b
\left|\nabla\Big(\eta_U(U^n)^{\frac{\alpha_0-2}2} U \Big)\right|^2 dz \\
& \le -\int_{B_1^+} t^b V \eta_U^2 (U^n)^{\alpha_0-2} U \, dz
+\left(2+C(\alpha_0)\frac{\alpha_0+2}2\right)\int_{B_1^+} t^b
(U^n)^{\alpha_0-2}|U|^2 |\nabla \eta_U|^2 dz \\
& \le \frac 1{\alpha_0} \int_{B_1^+} t^b \eta_U^2 |V|^{\alpha_0}
dz+\frac{\alpha_0-1}{\alpha_0} \int_{B_1^+} t^b \eta_U^2
(U^n)^{\frac{\alpha_0(\alpha_0-2)}{\alpha_0-1}}
|U|^{\frac{\alpha_0}{\alpha_0-1}} dz
\\
& \qquad +\left(2+C(\alpha_0)\frac{\alpha_0+2}2\right)\int_{B_1^+}
t^b
(U^n)^{\alpha_0-2}|U|^2 |\nabla \eta_U|^2 dz \\
& \le \frac 1{\alpha_0} \int_{B_1^+} t^b \eta_U^2 |V|^{\alpha_0}
dz+\frac{\alpha_0-1}{\alpha_0} \int_{B_1^+} t^b \eta_U^2
(U^n)^{\alpha_0-2} |U|^{2} dz
\\
& \qquad +\left(2+C(\alpha_0)\frac{\alpha_0+2}2\right)\int_{B_1^+}
t^b (U^n)^{\alpha_0-2}|U|^2 |\nabla \eta_U|^2 dz \, .
\end{align*}
Since $U,V\in H^1(B_1^+;t^b)\subset L^{\alpha_0}(B_1^+;t^b)$, letting $n\to +\infty$, by Fatou Lemma, we deduce that $\nabla\Big(\eta_U
|U|^{\frac{\alpha_0-2}2}U\Big)\in L^2(B_1^+;t^b)$. Moreover, since $\eta_U |U|^{\frac{\alpha_0-2}2}U\in
L^2(B_1^+;t^b)$, being $U\in L^{\alpha_0}(B_1^+;t^b)$, then $\eta_U
|U|^{\frac{\alpha_0-2}2}U\in H^1(B_1^+;t^b)$.

In the rest of this proof, in order to simplify the notation, we will
denote the critical exponent $2^*(N,s-1)=\frac{2N}{N-2s+2}$ by $2^*$.
By Lemma \ref{l:sobolev} and \eqref{eq:Sobolev-pesata}, we have that
$\eta_U |U|^{\frac{\alpha_0-2}2}U\in
L^{2^{**}(b)}(B_1^+;t^b)$ and
 $\eta_U (\cdot,0)
|U(\cdot,0)|^{\frac{\alpha_0-2}2}U(\cdot,0)\in
L^{2^*}(B_1')$.
This implies that
\begin{equation} \label{eq:primo-passo}
U(\cdot,0) \in L^{\frac{\alpha_0 \cdot 2^*}2}(B_{r_U}')
\qquad \text{and} \qquad U\in L^{\frac{\alpha_0 \cdot
2^{**}(b)}2}(B_{r_U}^+;t^b) \, .
\end{equation}
Now, let $\frac 12<r_V<r_U$ and let $\eta_V\in C^\infty_c(\overline{\R^{N+1}_+})$ be such that
${\rm supp}(\eta_V)\subset \Sigma_{r_U}^+$ and $\eta_V \equiv 1$ in $\Sigma_{r_V}^+$.

Testing the second equation in \eqref{eq:system-BK} with $\eta_V^2
(V^n)^{\beta_0-2}V$, being $\beta_0=\frac{2^*\cdot
\alpha_0}{2(2^*-1)}\in (2,2^{**}(b))$, and exploiting the corresponding
boundary condition, we obtain
\begin{align*}
& \int_{B_1^+} t^b \nabla V\nabla\Big(\eta_V^2 (V^n)^{\beta_0-2}
V\Big)
dz \\
& \qquad =\int_{B_1'} [A(x)U(x,0)+B(x)]\,
\eta_V^2(x,0)(V^n(x,0))^{\beta_0-2}V(x,0)\, dx \, .
\end{align*}
Proceeding as above we infer
\begin{align*}
& C(\beta_0) \int_{B_1^+} t^b
\left|\nabla\Big(\eta_V(V^n)^{\frac{\beta_0-2}2} V \Big)\right|^2 dz \\
& \le\int_{B_1'} [A(x)U(x,0)+B(x)]\,
\eta_V^2(x,0)(V^n(x,0))^{\beta_0-2}V(x,0)\, dx \\
& \qquad + \left(2+C(\beta_0)\tfrac{\beta_0+2}2\right)\int_{B_1^+}
t^b
(V^n)^{\beta_0-2}|V|^2 |\nabla \eta_V|^2 dz  \, ,
\end{align*}
\begin{align*}
& \text{(by Young inequality)} \le \tfrac{\lambda^{1-\beta_0}}{\beta_0} \int_{B_1'} \eta_V^2(x,0) |A(x)|\, |U(x,0)|^{\beta_0}
dx \\
&\  +\tfrac{\lambda(\beta_0-1)}{\beta_0} \int_{B_1'} |A(x)|
\eta_V^2(x,0) (V^n(x,0))^{\frac{\beta_0(\beta_0-2)}{\beta_0-1}}
|V(x,0)|^{\frac{\beta_0}{\beta_0-1}}dx \\
& \
+\tfrac{\lambda^{1-\beta_0}}{\beta_0} \int_{B_1'} \eta_V^2(x,0) |B(x)|\, dx+\tfrac{\lambda(\beta_0-1)}{\beta_0} \int_{B_1'} |B(x)|
\eta_V^2(x,0) (V^n(x,0))^{\frac{\beta_0(\beta_0-2)}{\beta_0-1}}
|V(x,0)|^{\frac{\beta_0}{\beta_0-1}}dx \\
& \ +\left(2+C(\beta_0)\tfrac{\beta_0+2}2\right)\int_{B_1^+}
t^b (V^n)^{\beta_0-2}|V|^2 |\nabla \eta_V|^2 dz \, ,
\end{align*}

\begin{align*}
& (\text{by H\"older inequality and the fact that } V^n\le |V|) \\
& \le \tfrac{\lambda^{1-\beta_0}}{\beta_0} \|A\|_{L^{\frac{N}{2(s-1)}}(B_1')}|B_1'|^{\frac{2(s-1)[N-2(s-1)]}{N[N+2(s-1)]}} \left(\int_{B_1'} \eta_V^2(x,0)|U(x,0)|^{\beta_0(2^*-1)} dx \right)^{\frac 1{2^*-1}}\\
& \ +\tfrac{\lambda(\beta_0-1)}{\beta_0} \int_{B_1'} |A(x)|\left(\eta_V(x,0)(V^n(x,0))^{\frac{\beta_0-2}2}|V(x,0)|\right)^2 dx \\
& \ +\tfrac{\lambda^{1-\beta_0}}{\beta_0}\, \|B\|_{L^{\frac N{2(s-1)}}(B_1')}|B_1'|^{\frac{N-2(s-1)}N}
+\tfrac{\lambda(\beta_0-1)}{\beta_0}\,  \int_{B_1'}|B(x)|
\left(\eta_V(x,0)(V^n(x,0))^{\frac{\beta_0-2}2}|V(x,0)| \right)^2 \, dx  \\
& \
+\left(2+C(\beta_0)\tfrac{\beta_0+2}2\right)\int_{B_1^+} t^b
(V^n)^{\beta_0-2}|V|^2 |\nabla \eta_V|^2 dz
\end{align*}
\begin{align*}
& \le \tfrac{\lambda^{1-\beta_0}}{\beta_0} \|A\|_{L^{\frac{N}{2(s-1)}}(B_1')}|B_1'|^{\frac{2(s-1)[N-2(s-1)]}{N[N+2(s-1)]}} \left(\int_{B_{r_U}'} |U(x,0)|^{\beta_0(2^*-1)} dx \right)^{\frac 1{2^*-1}}\\
& +\tfrac{\lambda(\beta_0-1)}{\beta_0} \Big(\|A\|_{L^{\frac
N{2(s-1)}}(B_1')}\!\!\!+
\|B\|_{L^{\frac N{2(s-1)}}(B_1')} \Big) \!\! \left(
 \int_{B_1'} \left|\eta_V(x,0)(V^n(x,0))^{\frac{\beta_0-2}2}V(x,0)\right|^{2^*} \!\!\! dx\right)^{\frac 2{2^*}} \\
&  +\tfrac{\lambda^{1-\beta_0}}{\beta_0}\,
\|B\|_{L^{\frac N{2(s-1)}}(B_1')}|B_1'|^{\frac{N-2(s-1)}N}
+\left(2+C(\beta_0)\tfrac{\beta_0+2}2\right)\int_{B_1^+} t^b
(V^n)^{\beta_0-2}|V|^2 |\nabla \eta_V|^2 dz \, ,
\end{align*}
and finally by \eqref{eq:Sobolev-traccia}
\begin{align*}
& \le \tfrac{\lambda^{1-\beta_0}}{\beta_0} \|A\|_{L^{\frac{N}{2(s-1)}}(B_1')}|B_1'|^{\frac{2(s-1)[N-2(s-1)]}{N[N+2(s-1)]}} \left(\int_{B_{r_U}'} |U(x,0)|^{\beta_0(2^*-1)} dx \right)^{\frac 1{2^*-1}}\\
& +\tfrac{\lambda(\beta_0-1)}{\beta_0} \Big(\|A\|_{L^{\frac
N{2(s-1)}}(B_1')}+\|B\|_{L^{\frac
N{2(s-1)}}(B_1')} \Big)
\widetilde S(N,b)^{-1} \int_{B_1^+} t^b \left|\nabla\Big(\eta_V (V^n)^{\frac{\beta_0-2}2} V\Big)\right|^2 dz  \\
& +\tfrac{\lambda^{1-\beta_0}}{\beta_0}\,
\|B\|_{L^{\frac N{2(s-1)}}(B_1')}|B_1'|^{\frac{N-2(s-1)}N}
+\left(2+C(\beta_0)\frac{\beta_0+2}2\right)\int_{B_1^+} t^b
(V^n)^{\beta_0-2}|V|^2 |\nabla \eta_V|^2 dz \, .
\end{align*}
Choosing $\lambda>0$ small enough, in such a way that the constant
$$
K:=C(\beta_0)-\tfrac{\lambda(\beta_0-1)}{\beta_0}
\Big(\|A\|_{L^{\frac N{2(s-1)}}(B_1')}+\|B\|_{L^{\frac N{2(s-1)}}(B_1')} \Big)
\widetilde S(N,b)^{-1}
$$
becomes positive, we obtain
\begin{multline}\label{eq:stima-BK}
 K \int_{B_1^+} t^b \left|\nabla\Big(\eta_V (V^n)^{\frac{\beta_0-2}2} V\Big)\right|^2 dz \\
 \le \tfrac{\lambda^{1-\beta_0}}{\beta_0}
 \|A\|_{L^{\frac{N}{2(s-1)}}(B_1')}|B_1'|^{\frac{2(s-1)[N-2(s-1)]}{N[N+2(s-1)]}}
 \left(\int_{B_{r_U}'} |U(x,0)|^{\beta_0(2^*-1)} dx \right)^{\frac
   1{2^*-1}}\\
+\tfrac{\lambda^{1-\beta_0}}{\beta_0}\, \|B\|_{L^{\frac N{2(s-1)}}(B_1')}|B_1'|^{\frac{N-2(s-1)}N}
+\left(2+C(\beta_0)\tfrac{\beta_0+2}2\right)\int_{B_1^+} t^b
(V^n)^{\beta_0-2}|V|^2 |\nabla \eta_V|^2 dz \, .
\end{multline}
We observe that by \eqref{eq:primo-passo} and the definition of
$\beta_0$ we have that the integral in the right hand side of \eqref{eq:stima-BK}
involving the function $U$ is finite and so it is the one involving the function $V$ since
$V\in L^{\beta_0}(B_1^+;t^b)$ being $\beta_0 \in (2,2^{**}(b))$.

Passing to the limit as $n\to +\infty$, by Fatou Lemma, we have that
$\nabla\Big(\eta_V |V|^{\frac{\beta_0-2}2} V\Big)\in
L^2(B_1^+;t^b)$ and hence $\eta_V |V|^{\frac{\beta_0-2}2} V\in
H^1(B_1^+;t^b)$. By \eqref{eq:Sobolev-pesata} we then have $V\in
L^{\frac{2^{**}(b)\cdot \beta_0}2}(B_{r_V}^+;t^b)$.

Now we want to iterate the procedures previously applied to the functions $U$
and $V$ to improve their summability. To this
purpose we define two sequences of radii in the following way:
\begin{equation*}
\rho_0=\frac 34 \, , \quad r_0=\frac 78 \, , \quad
\rho_{k+1}:=\frac 12\left(\rho_k+\frac 12\right) \, , \quad
r_{k+1}:=\frac 12\left(\rho_k+\rho_{k+1}\right) \qquad \text{for
any } k\ge 0 \, .
\end{equation*}
Then we define two sequences of exponents in the following way:
\begin{equation*}
\alpha_{k+1}:=\beta_k\cdot \frac{2^{**}(b)}2 \, , \quad
\beta_{k+1}:=\frac{2^*\cdot \alpha_{k+1}}{2(2^*-1)}
\qquad \text{for any } k\ge 0 \, .
\end{equation*}
We observe that
\begin{equation} \label{eq:precisazione}
\alpha_{k+1}<\alpha_k \cdot \frac{2^{**}(b)}2 \qquad \text{and}
\qquad \beta_{k+1}<\beta_k \cdot \frac{2^{**}(b)}2 \, .
\end{equation}
We apply inductively the two procedures to $U$ and $V$
respectively, replacing every time $r_U$ with $r_k$, $r_V$ with
$\rho_k$, $\alpha_0$ with $\alpha_k$ and $\beta_0$ with $\beta_k$.

If after a certain step we obtained that $U(\cdot,0)\in
L^{\frac{\alpha_k\cdot 2^*}2}(B_{r_k}')$, $U\in
L^{\frac{\alpha_k\cdot 2^{**}(b)}2}(B_{r_k}^+;t^b)$ and $V\in
L^{\frac{\beta_k\cdot 2^{**}(b)}2}(B_{\rho_k}^+;t^b)$, then at the
beginning of the subsequent step, by \eqref{eq:precisazione}, we
have in particular $U\in L^{\alpha_{k+1}}(B_{r_k}^+;t^b)$ and $V\in
L^{\beta_{k+1}}(B_{\rho_k}^+;t^b)$. Applying the two procedures first
to $U$ and then to $V$, we obtain $U(\cdot,0)\in
L^{\frac{\alpha_{k+1}\cdot 2^*}2}(B_{r_{k+1}}')$, $U \in
L^{\frac{\alpha_{k+1}\cdot 2^{**}(b)}2}(B_{r_{k+1}}^+;t^b)$
and $V\in L^{\frac{\beta_{k+1}\cdot
2^{**}(b)}2}(B_{\rho_{k+1}}^+;t^b)$.

It is easy to check that $\beta_{k+1}/\beta_k=\frac{2^*\cdot 2^{**}(b)}{4(2^*-1)}>1$ so that
$\lim_{k\to +\infty} \alpha_k=\lim_{k\to +\infty}
\beta_k=+\infty$. Since $r_k>\rho_k>\frac 12$ for any $k$ the
proof of the lemma then follows.
\end{proof}

\begin{remark} \label{r:raggi} We observe that in Propositions \ref{p:reg-1}, \ref{p:reg-2}, \ref{p:Brezis-Kato}
the equations are set in the half ball in $\R^{N+1}$ of radius 1
and that the regularity or summability result is obtained in the
half ball of radius 1/2 or 1/4. The special choice of those radii
was made only for simplicity of notation but it easy to understand
that completely similar results still hold true with the equations
set in a half ball of arbitrary radius $R_1$ and with the
conclusion on regularity or summability obtained on a half ball of
arbitrary radius $R_2<R_1$. \endproof
\end{remark}

We now state a H\"older regularity result for solutions of system \eqref{eq:system-BK}.

\begin{Proposition} \label{p:6.5} Let $s\in (1,2)$, $b=3-2s\in (-1,1)$, $A\in L^{\bar q}(B_1')$, $B\in L^{\bar q}(B_1')$ for some $\bar q>\frac N{2(s-1)}$.
If $U,V\in H^1(B_1^+;t^b)$ weakly solve \eqref{eq:system-BK} then $U,V\in C^{0,\alpha}(\overline B_{1/2}^+)$ for some $\alpha\in (0,1)$ and moreover there exists a constant $K_2$ depending only on $N,b,\|A\|_{L^{\bar q}(B_1')},\|B\|_{L^{\bar q}(B_1')}, \|U\|_{L^{2^{**}(b)}(B_1^+;t^b)}$ and $\|V\|_{L^{2^{**}(b)}(B_1^+;t^b)}$ such that
\begin{align*}
& \|U\|_{C^{0,\alpha}(\overline B_{1/2}^+)} \le K_2 \, , \qquad  \|V\|_{C^{0,\alpha}(\overline B_{1/2}^+)} \le K_2 \, .
\end{align*}
\end{Proposition}

\begin{proof} We first apply Proposition \ref{p:Brezis-Kato} to $U$
  and $V$ and,
 taking into account Remark \ref{r:raggi}, we obtain $U,V\in L^q(B_r^+;t^b)$ and $U,V\in L^q(B_r')$ for any $1\le q<\infty$ and $r\in (1/2,1)$. Then, by \eqref{eq:system-BK}, by the assumptions on $A$ and $B$, by Proposition \ref{p:reg-1} (i) applied to $U$ and $V$ respectively and by Remark \ref{r:raggi}, we obtain $U,V\in C^{0,\alpha}(\overline B_{1/2}^+)$ for some $\alpha\in (0,1)$.
\end{proof}

We are now ready to prove a H\"older regularity estimate for
derivatives of solutions $(U,V)$ of \eqref{eq:system-BK}.

\begin{Proposition} \label{p:6.4}
Let $s\in (1,2)$, $b=3-2s\in (-1,1)$, $A,B\in W^{1,\bar q}(B_1')$ for some $\bar q>\frac N{2(s-1)}$.
Then the following statements hold true:

\begin{itemize}

\item[$(i)$] if $U,V\in H^1(B_1^+;t^b)\cap C^{0,\alpha}(\overline
B_1^+)$, for some $\alpha\in (0,1)$, weakly solve
\eqref{eq:system-BK} then $\nabla_x U,\nabla_x V\in
C^{0,\beta}(\overline B_{1/2}^+)$ for some $\beta\in (0,\alpha)$
and moreover there exists a constant $K_3$ depending only on
$N,b,\|A\|_{W^{1,\bar q}(B_1')}, \|B\|_{W^{1,\bar q}(B_1')},
\|U\|_{C^{0,\alpha}(\overline B_1^+)}$ and
$\|V\|_{C^{0,\alpha}(\overline B_1^+)}$ such that
\begin{align*}
& \|\nabla_x U\|_{C^{0,\beta}(\overline B_{1/2}^+)} \le K_3 \, , \qquad  \|\nabla_x V\|_{C^{0,\beta}(\overline B_{1/2}^+)} \le K_3 \, .
\end{align*}

\item[$(ii)$] if we also assume $A,B \in C^{0,\alpha}(B_1')$
for some $\alpha\in (0,1)$ and if $U,V\in H^1(B_1^+;t^b)\cap
C^{0,\alpha}(\overline B_1^+)$ weakly solve \eqref{eq:system-BK} in
$B_1^+$ then $t^b U_t,t^b V_t\in C^{0,\beta}(\overline B_{1/2}^+)$
for some $\beta\in (0,\alpha)$
and moreover there exists a constant $K_4$ depending only on $N,b,\|A\|_{C^{0,\alpha}(B_1')}, \|B\|_{C^{0,\alpha}(B_1')}, \\
    \noindent
    \|U\|_{C^{0,\alpha}(\overline B_1^+)}$ and $\|V\|_{C^{0,\alpha}(\overline B_1^+)}$ such that
\begin{align*}
& \|t^b U_t \|_{C^{0,\beta}(\overline B_{1/2}^+)} \le K_4 \, , \qquad  \|t^b V_t\|_{C^{0,\beta}(\overline B_{1/2}^+)} \le K_4 \, .
\end{align*}
\end{itemize}
\end{Proposition}

\begin{proof} In order to prove (i) we proceed as in the proof of Lemma 3.3 in \cite{FF}. We define for any $\xi\in \R^N$ with $|\xi|$ small enough the functions
\begin{equation*}
U^\xi(x,t):=\frac{U(x+\xi,t)-U(x,t)}{|\xi|} \qquad \text{and} \qquad V^\xi(x,t):=\frac{V(x+\xi,t)-V(x,t)}{|\xi|}
\end{equation*}
for any $(x,t)\in B_{3/4}^+$. Then we have
\begin{align*}
\begin{cases}
{\rm div} (t^b \nabla U^\xi)=t^b V^\xi  & \qquad \text{in } B_{3/4}^+ \, , \\
{\rm div} (t^b \nabla V^\xi)=0  & \qquad \text{in } B_{3/4}^+ \, , \\
\lim_{t\to 0^+} t^b U^\xi_t=0   & \qquad \text{on } B_{3/4}' \, , \\
-\lim_{t\to 0^+} t^b V^\xi_t=A(x)U^\xi+B_\xi & \qquad \text{on }
B_{3/4}' \, ,
\end{cases}
\end{align*}
where
\begin{equation*}
B_\xi(x):=\frac{A(x+\xi)-A(x)}{|\xi|}\,
U(x+\xi,0)+\frac{B(x+\xi)-B(x)}{|\xi|} \, .
\end{equation*}
We observe that
\begin{align*}
&\|B_\xi\|_{L^{\bar q}(B_{3/4}')}\le \|A\|_{W^{1,\bar q}(B_1')}\,
\|U\|_{C^{0,\alpha}(\overline B_1^+)}+\|B\|_{W^{1,\bar q}(B_1')} \, .
\end{align*}
Applying Proposition \ref{p:6.5} to $U^\xi$ and
$V^\xi$ and taking into account Remark \ref{r:raggi}, we infer
that $\|U^\xi\|_{C^{0,\beta}(\overline B_{1/2}^+)},
\|V^\xi\|_{C^{0,\beta}(\overline B_{1/2}^+)}$ are uniformly
bounded with respect to $\xi$ small for some $\beta\in
(0,\alpha)$. Passing to the limit as $\xi\to 0$, by the
Ascoli-Arzel\`a Theorem we deduce that $\nabla_x U,\nabla_x V\in
C^0(\overline B_{1/2}^+)$. Finally, exploiting the uniform
H\"older estimates for $U^\xi$ and $V^\xi$, passing to the limit
as $\xi\to 0$, we obtain the validity of the H\"older estimates
for $\nabla_x U$ and $\nabla_x V$ on $\overline B_{1/2}^+$. This
completes the proof of (i).

It remains to prove (ii). We first observe that
$A(x)U(x,0)+B(x)\in C^{0,\alpha}(B_1')$ and hence by \cite[Lemma
4.5]{CaSi}, applied to the function $V$, we obtain $t^b V_t\in
C^{0,\beta}(\overline B_{1/2}^+)$ for some $\beta\in (0,\alpha)$.
In turn, applying Proposition \ref{p:reg-2} to the function $U$,
we also obtain the H\"older continuity of the function $t^b U_t$
over $\overline B_{1/2}^+$. This completes the proof of (ii).
\end{proof}

\subsection{Properties of Bessel functions} \label{ss:7.3}
We start by recalling an asymptotic estimate for first kind Bessel
functions as $t\to +\infty$:
\begin{equation} \label{eq:Bessel-infty}
J_\nu(t)=O(t^{-1/2}) \qquad \text{as } t\to +\infty \, .
\end{equation}
This property can be deduced from the asymptotic expansion
\cite[(4.8.5)]{AAR}. In order to obtain a similar estimate for
derivatives of $J_\nu$ we start from the following identity
\begin{equation} \label{eq:iterative-Bessel}
J_\nu'(t)=-J_{\nu+1}(t)+\nu t^{-1} J_\nu(t) \, ,
\end{equation}
see for example \cite[Section 4.6]{AAR}. From this identity we
immediately see that $J_\nu'(t)=O(t^{-1/2})$ ad $t\to +\infty$.

Using iteratively \eqref{eq:iterative-Bessel}, we deduce that
\begin{equation} \label{eq:iterative-Bessel-2}
\frac{d^n J_\nu}{dt^n}(t)=O(t^{-1/2}) \qquad \text{as } t\to
+\infty \, .
\end{equation}

We conclude this subsection with an asymptotic estimate for the zeros of $J_\nu$
as $m\to +\infty$:
\begin{equation} \label{eq:zero-bessel}
j_{\nu,m}\sim \pi \, m \qquad \text{as } m\to +\infty  \, .
\end{equation}
For more details on \eqref{eq:zero-bessel}, see \cite[Page
506]{Watson} and also \cite[Eq. (1.5)]{Elbert}.

\end{document}